\documentclass[a4paper]{article}

\usepackage{graphicx}
\usepackage{amsmath}
\usepackage{amssymb}
\usepackage{url}
\usepackage{amscd}
\usepackage{color}

\usepackage[a4paper]{geometry}
\usepackage{amsthm}
\theoremstyle{plain}
\newtheorem{Thm}{Theorem}[section]
\newtheorem{Prop}[Thm]{Proposition}
\newtheorem{Lem}[Thm]{Lemma}
\newtheorem{Cor}[Thm]{Corollary}
\theoremstyle{definition}
\newtheorem{Def}[Thm]{Definition}
\theoremstyle{remark}
\newtheorem{Rmk}[Thm]{Remark}

\newtheorem*{proof*}{proof}

\newcommand{\Gr}[1]{R_1 ^{(#1)}}
\newcommand{\Grdel}[1]{R_{1, \delta} ^{(#1)}}
\newcommand{\sumDiall}
{
 \sum_{
 \substack{
  D_i \subset \{1, \cdots, k\},
  \\ 
  \# D_i \leq \eta k
  \text{ for all } i 
  }
  \hspace{1mm}
 }
}
\newcommand{\sumDisome}
{
 \sum_{
 \substack{
  D_i \subset \{1, \cdots, k\},
  \\ 
  \# D_i > \eta k
  \text{ for some } i 
  }
 }
}
\newcommand{\sumEiall}
{
 \sum_
 {
  \hspace{1mm}
  \substack
  {
   E_i \subset D_i ^c
  ,
  \\
   \# (D_i ^c \setminus E_i ) \leq \gamma k 
   \text{ for all } i
  }
  \hspace{1mm}
 }
}

\newcommand{\sumNiDE}
{
 \sum_
 {
  \hspace{1mm}
  \substack
  {
   \mathcal{N}^{(i)}
  \in
   \left(\mathcal{R}^c\right)^{D_i ^c \setminus E_i} 
  ,
  \\
   i=1,\cdots,p
  }
  \hspace{1mm}
 }
}
\newcommand{\sumNiE}
{
 \sum_
 {
  \hspace{1mm}
  \substack
  {
   \mathcal{N}^{(i)} |_{E_i}
  \in  
   \mathcal{R}^{E_i}
  ,
  \\
   i=1,\cdots, p
  }
  \hspace{1mm}
 }
}
\newcommand{\sumNiEJ}
{
 \sum_
 {
  \hspace{1mm}
  \substack
  {
   \mathcal{N}^{(i)} 
  \in 
   \mathcal{R}^{E_i\setminus J_i}
  ,
  \\
   i=1,\cdots, p
  }
  \hspace{1mm}
 }
}
\newcommand{\sumNiJ}
{
 \sum_
 {
  \hspace{1mm}
  \substack
  {
   \mathcal{N}^{(i)} 
  \in 
   \mathcal{R}^{J_i}
  , 
  \\
   i=1,\cdots, p
  }
  \hspace{1mm}
 }
}
\newcommand{\sumFipr}
{
 \sum_
 {
  \hspace{1mm}
  \substack
  {
   F_i ' \subset F_i
  ,
  \\
   \# F_i ' = \# S^*
   \text{ for all }i
  }
  \hspace{1mm}
 }
}
\newcommand{\sumsigmaS}
{
 \sum_
 {
  \hspace{1mm}
  \substack
  {
   {\sigma_i |}_{{S^*}}
  \in
   \mathrm{Bij}({S^*}, {F_i '})
  ,
  \\
   i=1,\cdots, p
  }
  \hspace{1mm}
 }
}
\newcommand{\sumsigmaSc}
{
 \sum_
 {
  \hspace{1mm}
  \substack
  {
   {\sigma_i |}_{{S^*} ^c}
  \in
   \mathrm{Bij}({S^*} ^c, {F_i '}^c)
  ,
  \\
   i=1,\cdots, p
  }
  \hspace{1mm}
 }
}

\makeatletter
 
 \@addtoreset{equation}{section}
\makeatother

\usepackage{hyperref}

\begin{document}

\allowdisplaybreaks[4]
\abovedisplayskip=4pt
\belowdisplayskip=4pt
\large

\author
{%
 Takahiro Mori%
 \thanks
 {
  Research Institute for Mathematical Sciences, Kyoto University, Kyoto, 606-8502, JAPAN.
  \url{tmori@kurims.kyoto-u.ac.jp}
 }
}
\title{Large deviations for intersection measures of some Markov processes}
\date{}
\maketitle

\begin{abstract}
Consider 
an intersection measure $\ell_t ^{\mathrm{IS}}$ of 
$p$ independent (possibly different) 
$m$-symmetric Hunt processes up to time $t$ 
in a metric measure space $E$ with a Radon measure $m$.

We 
derive a Donsker-Varadhan type large deviation principle 
for the normalized intersection measure 
$t^{-p}\ell_t ^{\mathrm{IS}}$ on the set of finite measures 
on $E$ as $t \rightarrow \infty$,
under the condition that $t$ is smaller than life times
of all processes.

This 
extends earlier work by W. K\"onig and C. Mukherjee \cite{MR2999298}, 
in which the large deviation principle was established 
for the intersection measure of $p$ independent $N$-dimensional 
Brownian motions before exiting some bounded open set 
$D \subset \mathbb{R}^N$.

We
also obtain the asymptotic behaviour of logarithmic moment 
generating function, which is related to the 
results of X. Chen and J. Rosen \cite{MR2165257}
on the intersection measure of independent Brownian motions
or stable processes.

Our 
results rely on assumptions about the heat kernels and 
the 1-order resolvents of the processes, 
hence include rich examples. 
For 
example, the assumptions hold for $p\in \mathbb{Z}$ with 
$2\leq p < p_*$ when the processes enjoy (sub-)Gaussian 
type or jump type heat kernel estimates,
where 
$p_*$ is determined by the Hausdorff dimension of 
$E$ and the so-called walk dimensions of the processes. 

\medskip

\noindent\textbf{Keywords:} 
intersection measure; large deviations; heat kernel

\end{abstract}




\section{Introduction and main results}
\label{Sec_intro}

\subsection{Introduction}

Let 
$p$ be an integer grater than or equal to 2,
$E$ be a locally compact, separable metric space and
$m$ be a Radon measure on $E$ with $\text{supp}[m]=E$.
Let $X^{(1)}, \cdots, X^{(p)}$ 
be $p$ independent irreducible Hunt processes 
on $E$, with life times
$\zeta^{(1)}, \cdots, \zeta^{(p)}$, respectively.
We do not require that all $X^{(1)}, \cdots, X^{(p)}$ 
have the same law.
For each $t>0$, under the condition that all life times
$\zeta^{(1)}, \cdots, \zeta^{(p)}$ are less than $t$, 
{\itshape the intersection measure}
$
 \ell_t ^{\mathrm{IS}}
$
is formally written as
\begin{align}
\label{eq_ISformally}
 \ell_t ^{\mathrm{IS}}
 (A)
``
=
" 
 \int_A
 \biggl[
 \int_{[0, t]^p}
  \prod_{i=1} ^p
  \delta_{x}(X^{(i)}(s_i))
 ds_1\cdots ds_p
 \biggr]
 m(dx)
\quad
 \text{ for }
 A \subset E \text{ Borel}
.
\end{align}
Here and in the following, the superscript ``IS" means ``InterSection".

The 
intersection measure is firstly introduced by
Le Gall \cite{MR1229519}, 
when the processes are independent Brownian motions.
The 
large deviation result for the intersection measures are obtained
by K\"onig and Mukherjee \cite{MR2999298}, for the case of independent 
Brownian motions before exiting a bounded domain $D\subset \mathbb{R}^N$ 
with a smooth boundary with $N-p(N-2)>0$. 
This is roughly written as
an asymptotic behavior of the conditional probability
\begin{equation}
\label{eq_LDPformally}
 \mathbb{P}
 \biggl(
  \Bigl(
   \frac{1}{t^p} \ell_{t} ^{\mathrm{IS}}
  ;
   \frac{1}{t} \ell_{t} ^{(1)}
  ,
   \cdots
  ,
   \frac{1}{t} \ell_{t} ^{(p)}
  \Bigr)
 \approx
  \boldsymbol{\mu}
 \hspace{1mm}
 \bigg|
 \hspace{1mm}
  t < \tau^{(1)}\wedge \cdots \wedge \tau^{(p)} 
 \biggr)
\approx
 \exp
 \biggl\{
  -t 
  \sum_{i=1} ^p  
  \|
   \nabla \psi_i
  \|_{L^2} ^2
 \biggr\}
\end{equation}
as $t\rightarrow \infty$.
(%
 See Definition \ref{Def_LDP} 
 for a precise definition of 
 the large deviation principle.%
)
Here
$\ell_t ^{(i)}$ and $\tau^{(i)}$
are the occupation measure and the exit time from $D$ 
of independent Brownian motion $X^{(i)}$ respectively,
and
$
 \text{\mathversion{bold}$\mu$}
=
 (\mu ; \mu_1, \cdots, \mu_p) 
\in
 \mathcal{M}_f(D) \times (\mathcal{M}_{1}(D))^p 
$,
a tuple of a finite measure and 
$p$ probability measures on $D$,
is of the form
$
 \psi_i 
= 
 \sqrt{\frac{d\mu_i}{dm}}
\in
 H_0 ^1 (D)
$
and
$
 \frac{d\mu}{dm}
=
 \prod_{i=1} ^p
 \frac{d\mu_i}{dm}
$,
where $H_0 ^1(D)$ is the Sobolev space 
with zero boundary condition in $D$.

The aim of this paper is
to prove Theorem \ref{Thm_LDP},
in which we extend such large deviation results 
\eqref{eq_LDPformally} for intersection measures
to general Markov processes on metric measure spaces,
replacing
$H_0 ^1(D)$ and $\|\nabla \cdot \|_{L^2} ^2$
by Dirichlet forms 
$(\mathcal{F}^{(i)}, \mathcal{E}^{(i)})$ 
and replacing $\tau^{(i)}$ 
by the life times $\zeta^{(i)}$ 
of the corresponding Hunt processes $X^{(i)}$.
Main tools of such generalization are 
Dirichlet form techniques
(see \cite{MR2778606} for instance).

The 
asymptotics of the logarithmic moment generating function
\begin{equation}
\label{eq_logMGF}
 \lim_{t\rightarrow \infty}
 \frac{1}{t}
 \log
 \mathbb{E}
 [
  \exp
  \{
   \theta \ell_t ^{\mathrm{IS}}(E) ^{1/p}
  \}
 ]
\quad
 \text{ for }
 \theta > 0
\end{equation}
is calculated in \cite{MR2094445} 
for the case of independent Brownian motions,
and in \cite{MR2165257} 
for the case of independent stable processes.
In these papers, they point out that the limit 
in \eqref{eq_logMGF} is related to the best constants 
of the Gagliardo-Nirenberg type interpolation inequalities.
Our second result is to obtain 
the logarithmic moment generating functions 
of intersection measures for more general processes.
In Proposition \ref{Prop_mgf}, 
we calculate a limit similar to \eqref{eq_logMGF} 
and represent it as a variational formula.

We emphasise that, so far, 
the analysis of the intersection measure
is limited to the cases that the processes are 
independent Brownian motions or stable processes.

In the next Section \ref{Sec_assumptions}, 
we introduce our assumptions and give some examples 
in Section \ref{Sec_Examples}.
In Section \ref{Sec_IS}, 
we introduce some notations
about intersection measures,
and we state our main results on large deviations 
in the following Section \ref{Sec_Main_results}.
In Section \ref{Sec_Preliminaries},
we give basic lemmas and calculations, 
used in the proof of our main results.
From Section \ref{Sec_existence} 
to Section \ref{Sec_proof_mgf}, 
we will prove our main results, respectively.
In Section \ref{Sec_7}, 
we will check the examples 
in Section \ref{Sec_Examples} satisfy 
the assumptions.

This paper is based on 
the author's master thesis
(unpublished, available only at Kyoto university).

\subsection{Assumptions}
\label{Sec_assumptions}

Let $p$ be an integer with $p\geq 2$.
Let 
$X$ be an irreducible $m$-symmetric Hunt process on $E$, 
with life time $\zeta$.
Let
$p_t(x, dy)$, $t>0$ be its transition probability
and
$\{T_t\} = \{T_t : t\geq 0\}$ 
be the associated strongly continuous 
contraction semigroup of symmetric, Markovian linear
operators on $L^2(E; m)$.
Let
$R_1$ be the 1-order resolvent of $\{T_t\}$.
By 
the Markovian property of $\{T_t\}$, $R_1$
can be considered as an operator on $L^\infty(E; m)$.

We now make six assumptions on $X$:

\begin{itemize}
\item[\textbf{(A1)}]
$X$ has the following tightness property:
\begin{equation}
\label{eq_tight}
 \text{for all } \varepsilon >0, 
 \text{ there exists a compact set } K
 \text{ such that }
 \sup_{x\in E}
 R_1 1_{K^c} (x) \leq \varepsilon
.
\end{equation}

\item[\textbf{(A2)}]
For each $t>0$ and $x\in E$, the measure $p_t(x, dy)$
on $E$ is absolutely
continuous with respect to $m(dy)$, 
and
its density $p_t(\cdot, \cdot)$ is continuous and 
uniformly bounded on $E\times E$.

\item[\textbf{(A2')}]
For 
each $t>0$ and $x\in E$, the measure $p_t(x, dy)$
on $E$ is absolutely continuous with respect to $m(dy)$, 
and
its density $p_t(\cdot, \cdot)$ is 
\textit{uniformly} continuous and 
uniformly bounded on $E\times E$.

\item[\textbf{(A3)}]
There exist 
$\rho >0$, $t_0>0$ and $C>0$ such that
\begin{equation}
\label{eq_trace}
 C^{-1}
 t^{-\rho/2}
\leq
 \int_E p_t (x, x) m(dx)
\leq
 C
 t^{-\rho/2}
\quad
 \text{for all } t\in (0, t_0]
.
\end{equation}

\item[\textbf{(A4)}]
There exist 
$\mu \in (2, \frac{2p}{p-1})$, 
$C>0$ and $t_0 >0$ such that
\begin{equation}
\label{eq_ultra-cont}
 \|T_t\|_{1\rightarrow\infty}
\leq
 C
 t^{-\mu/2}
\quad
 \text{for all } t\in (0, t_0]
.
\end{equation}
Here 
$\|\cdot\|_{1\rightarrow\infty}$ is the operator norm
from $L^1(E; m)$ to $L^\infty(E; m)$.

\item[\textbf{(A5)}]
$X$ satisfies
\begin{equation}
\label{eq_Green1}
 \sup_{x \in E}
 \int_E
  R_1(x, y)^p
 m(dy)
<
 \infty
\end{equation}
and
\begin{equation}
\label{eq_Green2}
 \lim_{\delta \downarrow 0}
 \sup_{x \in E}
 \int_E
  R_{1, \delta}(x, y)^p
 m(dy)
=
 0
,
\end{equation}
where
\begin{equation*}
 R_1 (x, y)
=
 \int_0 ^\infty
 e^{-t}
 p_t (x, y)
 dt
,\quad
 R_{1, \delta} (x, y)
=
 \int_0 ^\delta
 e^{-t}
 p_t (x, y)
 dt
\quad
 \text{for }
 x, y\in E
.
\end{equation*}
\end{itemize}
We 
say that $X$ satisfies Assumption \textbf{(A)} if
$X$ satisfies
\textbf{(A1)},
\textbf{(A2)},
\textbf{(A3)},
\textbf{(A4)} and
\textbf{(A5)}.
We 
say that $X$ satisfies Assumption \textbf{(A')} if
$X$ satisfies
\textbf{(A1)},
\textbf{(A2')},
\textbf{(A3)},
\textbf{(A4)} and
\textbf{(A5)}.
When we need to clarify constants, 
we denote the above assumptions as
$
 \textbf{(A3; $\rho, t_0, C$)}
$,
$
 \textbf{(A4; $\mu, t_0, C$)}
$,
$
 \textbf{(A; $\rho, \mu, t_0, C$)}
$
and
$
 \textbf{(A'; $\rho, \mu, t_0, C$)}
$.

\begin{Rmk}
\ \par
\vspace{-2mm}
\begin{itemize}
\setlength{\itemsep}{0mm}
\item[i)]
When $E$ is compact, clearly (A2) and (A2') are equivalent.

\item[ii)]
If $m(E)<\infty$, the ultra-contractivity (A4) 
implies the tightness (A1); 
see \cite{MR3685590} for example.

\item[iii)]
In the proof of \cite[Theorem7]{MR2165257}, 
Chen and Rosen used a stronger condition than (A2'),
namely the global Lipschitz continuity of $p_t(\cdot, \cdot)$.

\item[iv)]
\cite[Theorem 3.1]{MR2465826}
says that
the ultra-contractivity (A4) implies the existence of 
the bounded and continuous density on $E\setminus N$, 
where $N$ is a properly exceptional set. 
\end{itemize}
\end{Rmk}

\subsection{Examples}
\label{Sec_Examples}

Let $(E, d)$ be a 
locally compact, separable, bounded and connected metric space. 
For simplicity, suppose 
$
 \sup
 \{
  d(x, y)
 :
  x, y\in E
 \}
=
 1
$.
Let 
$m$ be a finite measure on $E$ with $\text{supp}[m] = E$
and 
$X$ be an $m$-symmetric Hunt process on $E$ with 
the transition density $p_t(\cdot, \cdot)$.

In view of 
the proof of \cite[Exercise 4.6.3]{MR2778606},
the connectivity of $E$ and the absolute continuity of $p_t$
imply the irreducibility of $X$.

Suppose 
that there exist positive constants
$t_0$, 
$c_1, \cdots, c_4$, 
$d_{\rm f}$ and $d_{\rm w}$
such that
the following \eqref{eq_ex1}, \eqref{eq_ex2} 
and either \eqref{eq_ex31} or \eqref{eq_ex32}
hold:

\noindent
\begin{itemize}
\setlength{\itemsep}{0mm}
\item
The uniformly volume growth condition:
\begin{equation}
\label{eq_ex1}
 c_1  ^{-1}
 r^{d_{\rm f}}
\leq
 m( B(x, r) )
\leq
 c_1
 r^{d_{\rm f}}
\quad
 \text{for all }
 x \in E
 \text{ and }
 0<r\leq 1
.
\end{equation}

\item
The on-diagonal heat kernel lower bound estimate:
\begin{equation}
\label{eq_ex2}
 c_2
 t^{-d_{\rm f} / d_{\rm w}}
\leq
 p_t(x, x)
\quad
 \text{for all }
 x \in E \text{ and } t<t_0 
.
\end{equation}

\item
The (sub-)Gaussian type 
heat kernel upper bound estimate:
\begin{equation}
\label{eq_ex31}
 \quad
 p_t (x, y)
\leq
 c_3
 t^{-d_{\rm f} / d_{\rm w}}
 \exp
 \left\{
  -c_4
  \left(
   \frac{d(x, y)^{d_{\rm w}}}{t}
  \right)^{1/(d_{\rm w} -1)}
 \right\}
\quad
 \text{for all }
 x, y\in E
 \text{ and }
 t<t_0
.
\end{equation}

\item
The jump-type
heat kernel upper bound estimate:
\begin{equation}
\label{eq_ex32}
 p_t (x, y)
\leq
 c_3
 \left\{
  t^{-d_{\rm f} / d_{\rm w}}
 \wedge
  \frac{t}{d(x, y)^{d_{\rm f} + d_{\rm w}}}
 \right\}
\quad
 \text{for all }
 x, y\in E
 \text{ and }
 t<t_0
.
\end{equation}

\end{itemize}

Then, 
as we will see in Section \ref{Sec_7},
Assumption \textbf{(A)} holds if
\begin{equation*}
 d_{\rm s} - p(d_{\rm s} - 2) >0
,
\end{equation*}
where 
$d_{\rm s} := 2d_{\rm f}/d_{\rm w}$ 
is the so-called spectral dimension.
Many 
processes satisfy the above conditions.
Here are some examples.

\begin{enumerate}
\item
Let $E$ 
be the closure of a bounded Lipschitz domain
$D\subset \mathbb{R}^N$ with $N=2\text{ or }3$,
i.e., there exist positive constants $R$ and $\Lambda$ 
such that, for every $z\in \partial D$,
there exists a Lipschiz function 
$
 \psi_z 
: 
 \mathbb{R}^{N-1}\rightarrow \mathbb{R}
$
such that
$
 \text{Lip}(\psi_z) 
\leq 
 \Lambda
$
and 
$B(z, R)\cap \partial D$ is represented 
as a part of the graph of $\psi_z$.
Let
$X$ be the reflecting Brownian motion on $E$,
i.e., a Hunt process on $E$ with the Dirichlet form
\begin{align*}
 \mathcal{F}
=
 W^{1,2}(E)
,
\quad
 \mathcal{E}(f, f)
=
 \int_E |\nabla f|^2 dm
\quad
 \text{for }
 f \in \mathcal{F}
,
\end{align*}
where $W^{1,2}(E)$ is the Sobolev space on $E$.
Then
\eqref{eq_ex1}, \eqref{eq_ex2} and \eqref{eq_ex31} 
hold with 
$d_{\rm f} = N$ and $d_{\rm w} = 2$.
For details, see \cite[Theorem 3.10]{MR2807275} 
for instance.

\item
Let
$E$ be a compact complete Riemannian manifold $(M, g)$ with
the dimension $N$ and the nonnegative Ricci curvature,
and
$m$ be the volume measure of $(M, g)$.
Also let $X$ be the Brownian motion on $(M, g)$,
i.e. a Hunt process on $E$
with the Dirichlet form
\begin{align*}
 \mathcal{F}
=
 W^{1,2}(M)
,
\quad
 \mathcal{E}(f, f)
=
 \int_E |\nabla f|^2 dm
\quad
 \text{for }
 f \in \mathcal{F}
,
\end{align*}
where $W^{1,2}(M)$ is the Sobolev space on $(M, g)$.
Then 
\eqref{eq_ex1}, \eqref{eq_ex2} and \eqref{eq_ex31} 
hold with 
$d_{\rm f} = N$ and $d_{\rm w} = 2$.
See \cite{MR834612} for details.

\item
Let 
$(E, d, m)$ be a compact $RCD^*(K, N)$ space 
with $K\in\mathbb{R}$ and $N\in [1, \infty)$,
which generalizes the notations that
``the Ricci curvature being bounded from below by $K$'' 
and 
``the dimension being bounded from above by $N$''.
(For precise definitions, 
see \cite{MR3205729} for instance.)
Suppose
\eqref{eq_ex1} holds with $d_{\rm f} = N$
and
$X$ is the Brownian motion on $E$,
i.e. 
a Hunt process on $E$ with the Dirichlet form 
\begin{align*}
&
 \mathcal{E}(f,f)
=
 \frac{1}{2}
 \inf
 \biggl\{
  \liminf_{j\rightarrow \infty}
  \int_E
   |\nabla f_j|^2
  dm
 :
  f_j\in \text{Lip}(E)
 ,
  f_j\rightarrow f \text{ in } L^2(E; m)
 \biggr\}
,
\\
&
 \mathcal{F}
=
 \{
  f\in L^2(E; m)
 :
  \mathcal{E}(f, f) < \infty
 \}
,
\end{align*}
where 
\begin{equation*}
 |\nabla f_j|(x)
:=
 \limsup_
 {
  \substack
  {
   y\rightarrow x,
   y\not= x
  }
 }   
 \frac
 {
  |f(x) - f(y)|
 }
 {
  d(x, y)
 }
\quad
 \text{for }
 x\in E
\end{equation*}
is the local Lipschitz constant of $f_j$.
Then 
\eqref{eq_ex1}, \eqref{eq_ex2} and \eqref{eq_ex31} 
hold with 
$d_{\rm f} = N$ and $d_{\rm w} = 2$.
See \cite{MR3489857} for details.

\item
Let
$E\subset \mathbb{R}^2$ be the compact Sierpi\'nski gasket,
$d$ be the Euclidean distance, 
and
$m$ be the Hausdorff measure $\mathcal{H}^{\alpha}$
of $E$ with $\alpha={\log 3}/{\log 2}$.
Also let
$X$ be the Brownian motion on $E$.
Then
\eqref{eq_ex1}, \eqref{eq_ex2} and \eqref{eq_ex31} 
hold with 
$d_{\rm f} = {\log 3}/{\log 2}$ and 
$d_{\rm w} = {\log 5}/{\log 2}$. 
See \cite{MR966175} for details, 
and see \cite{MR1668115}
for other diffusions on fractals.

\item
Let 
$E\subset \mathbb{R}^N$ be a bounded $C^{1, 1}$ open subset,
i.e., 
there exist positive constants $R$ and $\Lambda$ 
such that, for every $z\in \partial E$,
there exists a $C^1$-function 
$
 \psi_z 
: 
 \mathbb{R}^{N-1}\rightarrow \mathbb{R}
$
such that
$
 \|\psi_z\|_\infty
\leq
 \Lambda
$,
$
 \text{Lip}(\nabla\psi_z) 
\leq 
 \Lambda
$
and 
$B(z, R)\cap \partial E$ is represented 
as a part of the graph of $\psi_z$.
For
$0<\alpha\leq 2$, let
$X$ is a symmetric $\alpha$-stable process
before exiting $E$.
Then 
\eqref{eq_ex1}, \eqref{eq_ex2} and \eqref{eq_ex32} 
hold with $d_{\rm f} = N$ and $d_{\rm w} = \alpha$.
See \cite{MR3237737} and \cite{MR1900329} for details.

\item
Let 
$E\subset \mathbb{R}^N$ be a compact set 
and suppose $m$ satisfies 
\eqref{eq_ex1} with $d_{\rm f} >0$.
For 
$\alpha \in (0, 2)$
and a measurable function 
$c: E\times E \rightarrow \mathbb{R}$ such that
\begin{equation*}
 0< M^{-1} \leq c(x, y)=c(y, x) \leq M <\infty
\quad
 \text{for $m$-a.e. } x, y\in E
,
\end{equation*}
we can define the regular Dirichlet form
\begin{align*}
&
 \mathcal{F^{\alpha}}
:=
 \left\{
  u\in L^2(E; m)
 ;\quad
  \int_{E\times E}
   \frac{(u(x)-u(y))^2}{|x-y|^{N+\alpha}}
  m(dx)m(dy)
 <
  \infty
 \right\}
,
\\
&
 \mathcal{E^{\alpha}}(u, v)
:=
 \int_{F\times F}
  \frac{c(x, y)(u(x)-u(y))(v(x)-v(y))}
       {|x-y|^{N+\alpha}}
 m(dx)m(dy)
\quad
 \text{for }
 u, v\in\mathcal{F}^\alpha
.
\end{align*}
An associated Hunt process $X$ 
is called a stable-like process on $E$,
and 
it satisfies 
\eqref{eq_ex2} and \eqref{eq_ex32} with
$d_{\rm f} = N$ and $d_{\rm w} = \alpha$. 
See \cite{MR2008600} for details.

\end{enumerate}

\begin{Rmk}
\ \par
\vspace{-2mm}
\begin{itemize}
\setlength{\itemsep}{0mm}
\item[i)]
When $E$ is unbounded, some assumptions are not trivial.
For example, 
Brownian motion on entire space $\mathbb{R}^N$
does not satisfy Assumption (A3).

\item[ii)]
In example 2 and 3, we cannot relax the compactness of $E$. 
Indeed, on $RCD^*(K, N)$, Assumption (A3) implies
the compactness of $E$.
See \cite[Theorem 3.1]{MR3489857}.
\end{itemize}
\end{Rmk}

\subsection{Intersection measures}
\label{Sec_IS}

Before stating our main results 
about large deviation principles, 
we introduce some definitions and notations 
about intersection measures.

Let
$X^{(1)}, \cdots, X^{(p)}$ be independent $m$-symmetric Hunt processes,
starting at 
$
 x_0 ^{(1)}, \cdots, x_0 ^{(p)}
\in
 E
$
with the life times $\zeta^{(1)}, \cdots, \zeta^{(p)}$, respectively.
Suppose each $X^{(i)}$ has the transition density 
$p^{(i)}_t(\cdot, \cdot)$.
Fix their starting points
$
 x_0
=
 (x_0 ^{(1)}, \cdots, x_0 ^{(p)})
\in
 E^p
$.
The normalized probability measure 
$\widetilde{\mathbb{P}}_t$ up to time $t>0$ 
is defined by
\begin{equation*}
 \widetilde{\mathbb{P}}_t(F)
:=
 \frac
 {
  \mathbb{P}(F \cap \{t<\zeta^{(1)} \wedge \cdots \wedge \zeta^{(p)}\})
 }
 {
  \mathbb{P}(t<\zeta^{(1)} \wedge \cdots \wedge \zeta^{(p)})
 }
\quad
 \text{for }
 F\subset \Omega
.
\end{equation*}

\noindent
The occupation measure $\ell_t ^{(i)}$ of $X^{(i)}$ 
up to time $t>0$ is defined by 
\begin{equation*}
 \ell^{(i)} _{t} (A)
:=
 \int_0 ^t
  1_A (X^{(i)} (s))
 ds
\end{equation*}
for $A \subset E$,
on the event $\{t < \zeta^{(i)} \}$.
Note that $t^{-1}\ell^{(i)}_t$ is in $\mathcal{M}_1(E)$,
the set of probability measures on $E$.
For each $\varepsilon > 0$,
the approximated occupation measure 
$\ell_{\varepsilon, t} ^{(i)}$ of $X^{(i)}$ 
up to $t$ is defined by
\begin{equation*}
 \ell^{(i)} _{\varepsilon, t} (A)
:=
 \int_{A}
  \left[
   \int_{[0, t]}
    p^{(i)} _\varepsilon (X^{(i)} (s), x)
   ds
  \right]
 m(dx)
\end{equation*}
for $A \subset E$,
on the event $\{t < \zeta^{(i)} \}$.
Note that $t^{-1}\ell^{(i)}_{\varepsilon, t}$ is in 
$\mathcal{M}_{\leq1}(E)$, the set of sub-probability measures on $E$.

\begin{Def}
For each $\varepsilon > 0$, 
the approximated (mutual) intersection measure
$\ell_{\varepsilon, t} ^{\mathrm{IS}}$ 
of $X^{(1)}, \cdots, X^{(p)}$ up to $t$
is defined by
\begin{equation*}
 \ell_{\varepsilon, t} ^{\mathrm{IS}} (A)
:=
 \int_{A}
  \left[
   \int_{[0, t]^p}
    \prod_{i=1} ^p
    p^{(i)} _\varepsilon (X^{(i)} (s_i), x)
   ds_1 \cdots ds_p
  \right]
 m(dx)
\end{equation*}
for $A \subset E$, on the event
$\{t < \zeta^{(1)}\wedge\cdots\wedge\zeta^{(p)} \}$.
\end{Def}

Note 
that $\ell^{\mathrm{IS}}_{\varepsilon, t}$ is in 
$\mathcal{M}_{f}(E)$, the set of finite measures on $E$
equipped with the vague topology.

\begin{Def}
Fix $t>0$.
If the family of random measures 
$
 \{
  \ell_{\varepsilon, t} ^{\mathrm{IS}}
 ;
  \varepsilon>0
 \}
\subset
 \mathcal{M}_f(E)
$
converges 
in distribution as $\varepsilon \rightarrow 0$
with respect to the probability measure $\widetilde{\mathbb{P}}_t$,
we write the limit as $\ell_t ^{\mathrm{IS}}$ and call
it the (mutual) intersection measure
of $X^{(1)}, \cdots, X^{(p)}$ up to $t$.
\end{Def}

As in \eqref{eq_ISformally}, 
the intersection measure $\ell_t ^{\mathrm{IS}}$ 
can be formally written as
\begin{equation*}
 \ell_{t} ^{\mathrm{IS}} (A)
``
=
"
 \int_{A}
  \left[
   \int_{[0, t]^p}
    \prod_{i=1} ^p
    \delta_x (X^{(i)} (s_i))
   ds_1 \cdots ds_p
  \right]
 m(dx)
\end{equation*}
for $A \subset E$, on the event
$\{t < \zeta^{(1)}\wedge\cdots\wedge\zeta^{(p)} \}$.
Here $\delta_x$ is the Dirac delta function at $x$.
For the case of Brownian motion, 
there are several ways of constructing
the intersection measure.
See \cite{MR1229519} and \cite{MR1944002} 
for example.
The following proposition ensures 
that the intersection measure
also exists in our setting.
We will prove this in Section \ref{Sec_existence}.
\begin{Prop}
[Existence of the mutual intersection measure]
\label{Prop_existence}
Suppose each $X^{(i)}$ satisfies 
\eqref{eq_Green1} of Assumption (A5)
and let $t>0$.
Then,
there exists a random measure 
$
 \ell_t ^{\mathrm{IS}}
\in
 \mathcal{M}_f(E)
$ 
such that, 
in the vague topology of $\mathcal{M}_f(E)$,
\begin{equation*}
 \ell_{t, \varepsilon} ^{\mathrm{IS}}
\rightarrow
 \ell_{t} ^{\mathrm{IS}}
\quad
 \text{ in distribution}
,
 \text{ as }
 \varepsilon \rightarrow 0
\end{equation*}
with respect to the probability measure 
$\widetilde{\mathbb{P}}_t$.

Furthermore, for all $f\in C_K ^+(E)$
and all integer $k\geq 1$, it holds that
\begin{equation}
\label{eq_convergence_any_moment}
 \langle
  f, \ell_{\varepsilon, t} ^{\mathrm{IS}}
 \rangle
\rightarrow  
 \langle
  f, \ell_{t} ^{\mathrm{IS}}
 \rangle
\quad
 \text{ in } L^k(\widetilde{\mathbb{P}}_t)
,
 \text{ as }
 \varepsilon \rightarrow 0
.
\end{equation}
\end{Prop}

\subsection{Main results: Large deviations}
\label{Sec_Main_results}

We first 
recall a definition of the large deviation principle.
Let
$\mathcal{X}$ be a topological space
and 
$\{\mathbb{P}_t\}_{t>0}$ 
be a family of probability measures 
on a common sample space. 
\begin{Def}
[\cite{MR1619036}, Section 1.2]
\label{Def_LDP}
\ \par
\begin{enumerate}
\item[$\cdot$]
A function $I: \mathcal{X}\rightarrow [0, \infty]$ is said to be 
a rate function (resp. a good rate function) 
if 
the set $I^{-1}[0, \alpha]$ is closed (resp. compact) 
for all $\alpha \geq 0$.

\item[$\cdot$]
We say that
a family of $\mathcal{X}$-valued random variables
$\{Z_t\}$ satisfies the large deviation principle 
(LDP in abbreviation)
as $t\rightarrow \infty$ 
with probabilities $\{\mathbb{P}_t\}$, 
scale $t$ and a rate function $I$ 
if, 
for all Borel sets $\Gamma\subset \mathcal{X}$,
it holds that
\begin{equation*}
 -
 \inf_{x\in \text{int}(\Gamma)}
 I(x)
\leq
 \liminf_{t\rightarrow\infty}
 \frac{1}{t}
 \log
 \mathbb{P}_t (Z_t \in \Gamma)
\leq
 \limsup_{t\rightarrow\infty}
 \frac{1}{t}
 \log
 \mathbb{P}_t (Z_t \in \Gamma)
\leq
 -
 \inf_{x\in \overline{\Gamma}}
 I(x)
.
\end{equation*}
\end{enumerate}
\end{Def}

\vspace{1eM}

Let $X^{(1)}, \cdots, X^{(p)}$ 
be as in Section \ref{Sec_IS}
and suppose 
each $X^{(i)}$ has the associated regular Dirichlet form
$(\mathcal{E}^{(i)}, \mathcal{F}^{(i)})$. 

We introduce some notations about large deviation rate functions. 
The bottom of the spectrum $\lambda_1 ^{(i)}$ of $X^{(i)}$
is defined by
\begin{equation*}
 \lambda_1 ^{(i)}
:=
 \inf
 \biggl\{
  \mathcal{E}^{(i)} (\psi,\psi)
 : \quad
  \psi \in \mathcal{F}^{(i)}
  ,
  \int_E
   \psi^2
  dm
 =
  1
 \biggr\}
 .
\end{equation*}

\noindent
We define the function
$
 I^{(i)} 
:
 \mathcal{M}_1(E) \rightarrow [0, \infty]
$ 
by
\begin{align*}
 I^{(i)} (\mu)
:=&
 \begin{cases}
  \mathcal{E}^{(i)} (\psi, \psi)
 &
  \text{if }
   \psi 
  =
   \displaystyle
   \sqrt{\frac{d\mu}{dm}} 
  \in
   \mathcal{F}^{(i)}
 \\
  \infty
 &
  \text{otherwise} 
 \end{cases}
\end{align*}
for $\mu\in \mathcal{M}_1(E)$, 
and define the function
$
 J^{(i)} 
:
 \mathcal{M}_1(E) \rightarrow [0, \infty]
$ 
by
$
 J^{(i)}
:=
 I^{(i)} - \lambda_1 ^{(i)}
$.
We define the function
$
 \mathbf{J}
:
 \mathcal{M}_f(E) \times (\mathcal{M}_{1}(E))^p 
\rightarrow 
 [0, \infty]
$ 
by
\begin{equation}
\label{eq_bfJ}
 \mathbf{J}(\mu ; \mu_1, \cdots, \mu_p)
:=
 \begin{cases}
  \displaystyle
  \sum_{i=1} ^p
  J^{(i)} (\mu_i)
 &
  \text{if }
  \psi_i 
 = 
  \displaystyle
  \sqrt{\frac{d\mu_i}{dm}}\in \mathcal{F}^{(i)}
 \text{ and }
  \displaystyle
  \prod_{i=1} ^p
  \frac{d\mu_i}{dm}
 =
  \frac{d\mu}{dm}
 ,
 \\
  \infty
 &
  \text{otherwise}
 \end{cases}
\end{equation}
for 
$
 (\mu ; \mu_1, \cdots, \mu_p) 
\in
 \mathcal{M}_f(E) \times (\mathcal{M}_{1}(E))^p 
$.

We now state the main theorem of this paper.
We will prove it in Section \ref{Sec_4}. 
\begin{Thm}
[Large deviation principle]
\label{Thm_LDP}
Suppose each $X^{(i)}$ satisfies Assumption 
$\textbf{(A)}$.
Then
the tuple
\begin{equation*}
 \left(
  \frac{1}{t^p} \ell_{t} ^{\mathrm{IS}}
 ;
  \frac{1}{t} \ell_{t} ^{(1)}
 ,
  \cdots
 ,
  \frac{1}{t} \ell_{t} ^{(p)}
 \right)
\in
 \mathcal{M}_f(E) \times (\mathcal{M}_{1}(E))^p 
\end{equation*}
satisfies the 
LDP as $t\rightarrow \infty$, 
with probability $\widetilde{\mathbb{P}}_t$,
scale $t$ and the good rate function 
$\mathbf{J}$.
\end{Thm}
Note that
each occupation measure $t^{-1}\ell_t ^{(i)}$ 
satisfies the LDP as $t\rightarrow \infty$ 
with probability $\widetilde{\mathbb{P}}_t$,
scale $t$ 
and the good rage function $J^{(i)}$
(see Section \ref{Sec_Preliminaries}).
It can be regarded
as a special case $p=1$ of the above theorem.

For the proof the main theorem,
Proposition \ref{Prop_exp_app} 
plays an important role.
This proposition roughly says that
the approximated intersection measure 
$\ell_{\varepsilon, t} ^{\mathrm{IS}}$ 
is a ``good'' approximation 
of the intersection measure
$\ell_{t} ^{\mathrm{IS}}$.

Our second result is another application of 
Proposition \ref{Prop_exp_app},
which is an extension of 
\cite[Theorem 1]{MR2165257} in some sense.
We will prove this in Section \ref{Sec_proof_mgf}.
\begin{Prop}
[Asymptotics of the moment generating function]
\label{Prop_mgf}
Suppose each $X^{(i)}$ satisfies Assumption 
$
 (\textbf{A}';  \rho^{(i)}, \mu^{(i)}, t_0, C)
$.
Let $h \in \mathcal{B}_b(E)$ be nonnegative and 
compactly supported.
Then, for any $\theta>0$, it holds that
\begin{align*}
 \lim_{t\rightarrow \infty}
 \frac{1}{t}
 \log
 \widetilde{\mathbb{E}}_t
 \exp
  {
   \bigl\{
    \theta 
    \langle \ell_t ^{\mathrm{IS}}, h   \rangle^{1/p}
   \bigr\}
  }
=&
 \frac{1}{p}
 \sum_{i=1} ^p
 \sup_{\psi\in \mathcal{F}^{(i)} , \|\psi\|_2=1}
 \biggl\{
  \theta
  \Bigl(
   \int_E
    \psi^{2p}
   \cdot
    h
   dm
  \Bigr)^{1/p}
 -
  p\mathcal{E}^{(i)} (\psi,\psi) 
 +
  p\lambda_1 ^{(i)}
 \biggr\}
.
\end{align*}
\end{Prop}

Chen and Rosen consider
the stable processes on $\mathbb{R}^N$
and show this formula
by using the Fourier transformation method,
but
we use Proposition \ref{Prop_exp_app} instead.
Compare our Lemma \ref{Lem_6-3} 
and \cite[Theorem 6]{MR2165257}.


\section{Preliminaries}
\label{Sec_Preliminaries}

Before proving main results, we state some basic facts 
and give some easy calculations.

\subsection{Lemmas for our assumptions}

The following lemma is obtained from our assumptions, 
which forms some sufficient conditions for our main results.

\begin{Lem}
\label{Lem_pre}
\ \par
\vspace{-2.5mm}
\begin{enumerate}

\item
Assume Assumption (A2).
Then \eqref{eq_trace} implies that
the following eigenfunction expansion of the heat kernel;
there exist $L^2$-normalized and essentially bounded
$\psi_n$ and nonnegative $\lambda_n \uparrow \infty$ 
such that, 
$
 T_t \psi_n = e^{-\lambda_n t} \psi_n
$
for all $t>0$, $n\geq 1$
and
\begin{align}
\label{eq_eigenfct}
 p_t (x, y)
=
 \sum_{n=1} ^\infty 
 e^{-\lambda_n t}
 \psi_n (x) \psi_n (y)
\quad
 \text{for all }
 t>0
,
\end{align}
where the convergence is $L^2$ and
locally uniformly in $E\times E$.

\hspace{4mm}
Furthermore,
there exist an integer $N\geq 1$ 
and a positive constant $C$ such that
\begin{align}
\label{eq_growth_of_eigen}
 \|\psi_n\|_\infty 
\leq 
 C \lambda_n ^{\rho/2}
 \text{ and }
 C^{-1} n^{1/\rho}
 \leq
 \lambda_n
 \leq
 C n^{1/\rho}
\quad
 \text{for all } n\geq N
.
\end{align}
Here $\rho$ is the constant as in \eqref{eq_trace}.

\item
\eqref{eq_ultra-cont} implies that, 
for all $\delta >0$ there exists $C_\delta >0$ such that
\begin{align}
\label{eq_Sobolev}
 \|f\|_{2p} ^2
\leq
 C_\delta \|f\|_2 ^2
+
 \delta \mathcal{E}(f, f)
\quad
 \text{for all }f\in \mathcal{F}
.
\end{align}

\item
The first inequality \eqref{eq_Green1} in Assumption (A5) 
implies that the mappings
\begin{align}
\label{eq_compactp}
 (\mathcal{F}, \|\cdot\|_{\mathcal{E}_1})
\hookrightarrow
 L^2(E; m) 
,
 L^p(E; m) 
\end{align}
are compact embeddings.

Furthermore, under \eqref{eq_ultra-cont}, the mappings 
\begin{equation}
\label{eq_compact2p}
 (\mathcal{F}, \|\cdot\|_{\mathcal{E}_1})
\hookrightarrow
 L^2(E; m) 
,
 L^{2p}(E; m) 
\end{equation}
are compact embeddings.
\end{enumerate}
\end{Lem}

\begin{proof}
For \eqref{eq_eigenfct}, see \cite[Theorem 7.2.5]{MR2359869}. 
For \eqref{eq_growth_of_eigen}, see \cite{MR793651}.
For \eqref{eq_Sobolev}, use 
\cite[Theorem 2.16]{MR898496}
and H\"older's inequality.
\end{proof}

\begin{Rmk}
Assume Assumption (A2). 
In the proof of Proposition \ref{Prop_exp_app} and 
Theorem \ref{Thm_LDP}, we only use
 \eqref{eq_tight},
 \eqref{eq_Green1},
 \eqref{eq_Green2},
 \eqref{eq_eigenfct},
 \eqref{eq_growth_of_eigen},
 \eqref{eq_Sobolev} and
 \eqref{eq_compact2p}.


\end{Rmk}


\subsection{Basic facts and calculations}

\subsubsection{Basics on large deviations}

We recall basic definitions and facts on large deviations.
In this subsection, 
let 
$(\mathcal{X}, d)$ be a separable metric space,
$\mathcal{Y}$ be a Hausdorff topological space,
$\{(\Omega, \mathcal{B}, \mathbb{P}_t)\}_{t>0}$ 
be a family of probability spaces. 
We also let
$\{Z_t\}_{t>0}$ and $\{Z_{t, m}\}_{t>0}$, $m=1, 2, \dots$ be 
families of $\mathcal{X}$-valued random variables.

%

\begin{Thm}
[Contraction principle, {\cite[Theorem 4.2.1]{MR1619036}}]
\label{Thm_DZ_contraction}
Let 
$f:\mathcal{X}\rightarrow\mathcal{Y}$ be a continuous function.
Consider
$\{Z_t\}$ satisfies the LDP as $t\rightarrow \infty$ 
with probabilities $\{\mathbb{P}_t\}$, scale $t$ 
and a good rate function $I$.
Then,
$\{f(Z_t)\}$ satisfies the LDP as $t\rightarrow \infty$ 
with probabilities $\{\mathbb{P}_t\}$, scale $t$ 
and the good rate function
\begin{equation*}
 \bar{I}(y)
:=
 \inf_{x\in f^{-1}(y)}
 I(x)
\quad
 \text{for }y\in \mathcal{Y}
.
\end{equation*}
\end{Thm}

\begin{Def}
[\cite{MR1619036}, Definition 4.2.14]
We say that
$\{Z_{t, m}\}_t, m=1, 2, \dots$ are 
exponentially good approximations of $\{Z_t\}_t$ 
with respect to probability measures $\{\mathbb{P}_t\}$ if, 
for every $\delta>0$,
\begin{equation*}
 \lim_{m\rightarrow \infty}
 \limsup_{t\rightarrow \infty}
 \frac{1}{t}
 \log
 \mathbb{P}_t
 (
  d(Z_{t, m}, Z_{t}) > \delta
 )
=
 -\infty
.
\end{equation*}
\end{Def}

\begin{Thm}
[\cite{MR1619036}, Theorem 4.2.16]
\label{Thm_DZ_expgood}
Suppose that
for every $m$,
$\{Z_{t, m}\}_t$ satisfies the LDP as $t\rightarrow \infty$ 
with probabilities $\{\mathbb{P}_t\}$, scale $t$ 
and a good rate function $I_m$,
and that
$\{Z_{t, m}\}_t, m=1, 2, \dots$ are 
exponentially good approximations of $\{Z_t\}_t$ 
with respect to probability measures $\{\mathbb{P}_t\}$. 
Assume
that the function
\begin{equation*}
 \widetilde{I}(x)
:=
 \sup_{\delta > 0}
 \liminf_{m\rightarrow \infty}
 \inf_{z\in B(x; \delta)}
 I_m(z)
\quad
 \text{for }x\in \mathcal{X}
\end{equation*}
is a good rate function, and assume that
for every closed set $F\subset \mathcal{X}$,
it holds that
\begin{equation*}
 \inf_{x\in F}
 \widetilde{I}(x)
\leq
 \limsup_{m\rightarrow\infty}
 \inf_{x\in F}
 I_m(x)
.
\end{equation*}
Then 
$\{Z_t\}$ satisfies the LDP as $t\rightarrow \infty$ 
with probabilities $\{\mathbb{P}_t\}$, scale $t$ 
and the good rate function $\widetilde{I}$.
\end{Thm}

\begin{Lem}
[Varadhan's integral lemma, {\cite[Theorem 4.3.1]{MR1619036}}]
\label{Lem_DZ_Varadhan}
Let 
$\phi:\mathcal{X}\rightarrow \mathbb{R}$ be a continuous function.
Suppose
$\{Z_t\}$ satisfies the LDP as $t\rightarrow \infty$ 
with probabilities $\{\mathbb{P}_t\}$, scale $t$ 
and a good rate function $I$.
If
\begin{equation*}
 \limsup_{t\rightarrow\infty}
 \frac{1}{t}
 \log
 \mathbb{E}_t[e^{\gamma t \phi(Z_t)}]
<
 \infty
\end{equation*}
for some $\gamma > 1$, then it holds that
\begin{equation*}
 \lim_{t\rightarrow\infty}
 \frac{1}{t}
 \log
 \mathbb{E}_t[e^{\gamma t \phi(Z_t)}]
=
 \sup_{x\in\mathcal{X}}
 \bigl\{
  \phi(x) - I(x)
 \bigr\}
.
\end{equation*}
\end{Lem}

\subsubsection{Large deviation principles for occupation measures}
\label{Sec_LDP_for_occ_meas}

We 
recall large deviation principles
for the occupation measures of $m$-symmetric Hunt processes, proved in 
\cite{MR2778606}, Section 6.3, 
and we make some remarks about it.

Let
$E$ be a locally compact, second countable Hausdorff space, 
$m$ be a $\sigma$-finite Radon measure on $E$ with $\text{supp}[m]=E$ and
$X$ be an $m$-symmetric Hunt process on $E$ with the associated 
regular Dirichlet form $(\mathcal{F}, \mathcal{E})$ 
on $L^2(E; m)$.
Let 
$\ell_t$, $\widetilde{\mathbb{P}}_t$, $J$ be introduced
in Section \ref{Sec_intro} and omit index $^{(i)}$.
We 
site the large deviation principle for 
the occupation measure $\ell_t$ of $X$:

\begin{Thm}
[\cite{MR2778606}, Theorem 6.4.6]
\label{Thm_LDP_FOT}
Suppose 
\begin{enumerate}
\item[I.]
(Irreducibility)
$X$ is irreducible,
\item[I\hspace{-0.1em}I.]
(Resolvent strong Feller property)
$
 R_1(\mathcal{B}_b(E))
\subset
 C_b(E)
$ 
and
\item[I\hspace{-0.1em}I\hspace{-0.1em}I.]
(Tightness)
 for all $\varepsilon >0$,
 there exists a compact set $K$
 such that
 $
  \sup_{x\in E}
  R_1 1_{K^c} (x) 
 \leq 
  \varepsilon
 $.
\end{enumerate}
Then 
$\ell_t$ satisfies the LDP
as $t\rightarrow \infty$ with probability 
$
 \widetilde{\mathbb{P}}_t
$,
scale $t$ and the good rage function $J$.
\end{Thm}

This large deviation principle can be regarded as 
$p=1$ version of Theorem \ref{Thm_LDP}.

\begin{Rmk}
\ \par
In view of Section 6.1--6.4 in \cite{MR2778606},
The resolvent strong Feller property is used
for deriving the following two properties:
\begin{enumerate}
\item[1.]
(page 347)
$R_\alpha(x, \cdot)$ is absolutely continuous 
with respect to $m$ for each $\alpha > 0$ and $x\in E$.

\item[2.]
(page 348)
For any function $\phi \in \mathcal{D}^+ (A)$
and for all $x \in E$,
$\phi(x) > 0$,
where
\begin{equation*}
 \mathcal{D}^+(A)
:=
 \{
  R_\alpha f
 :
  \alpha > 0
 ,
  f\in L^2(E; m)\cap C_b ^+(E)
 \text{ and }
  f\not= 0
 \}
.
\end{equation*}
\end{enumerate}

Property 1 easily follows from our assumption (A2).
Property 2 holds when 
$R_\alpha f$ is lower semicontinuous for all nonnegative 
Borel function $f$, and indeed this follows from
our assumption (A2) and Fatou's lemma.
\end{Rmk}

\subsubsection{Extension of the $L^2$ operator}

Let 
$E$ be a locally compact, second countable Hausdorff space and
$m$ be a $\sigma$-finite Radon measure on $E$. 

The aim of this subsection is to show the following two propositions:

\begin{Prop}
\label{Prop_Lp-cont-op}
Let
$T$ be a symmetric, Markovian contraction linear operator on $L^2(E; m)$.
Then, for all $p\in [1, \infty)$,
$T$ can be extended to a bounded linear operator on $L^p(E; m)$
with $\|T\|_p \leq 1$.
\end{Prop}

\begin{Prop}
\label{Prop_Lp-cont-semigrp}
Let
$\{T_t\}$ be a $C_0$-semigroup of symmetric, Markovian contraction 
linear operators on $L^2(E; m)$.
Then, for all $p\in (1, \infty)$,
$\{T_t\}$ can be extended to a contraction $C_0$-semigroup on $L^p(E; m)$. 
\end{Prop}

The next lemma will be used in the proof of these propositions.
\begin{Lem}
\label{Lem_Lp-cont}
Let $T$ be a positivity preserving linear operator on $L^1(E; m)$,
let $p, q\in [1, \infty]$ with $p^{-1} + q^{-1} =1$
and let $0\leq f\in L^p(E; m)$, $0\leq g\in L^q(E; m)$.
Then it holds that 
\begin{equation*}
 T[fg](x)
\leq
 (T[f^p](x))^{1/p}
 (T[g^q](x))^{1/q}
\quad
 \text{for $m$-a.e. }x\in E
.
\end{equation*}
\end{Lem}

\begin{proof}
We follow the proof of \cite[Theorem 7.24]{MR3410920}.

Since $T:L^1 \rightarrow L^1$ is positivity preserving, 
the inequality 
\begin{equation*}
 ab 
\leq 
 \frac{1}{p}
 \frac{a^p}{s^p}
+
 \frac{1}{q}
 b^q s^q
\quad
 \text{for all }
 a, b \geq 0
 \text{ and } s>0
\end{equation*}
implies that, there exists $N\subset E$ 
such that $m(N)=0$ and
\begin{equation*}
 T[f g](x)
\leq 
 \frac{1}{p}
 \frac{T[f^p](x)}{s^p}
+
 \frac{1}{q}
 T[g^q](x) s^q
\quad
 \text{for all } 
 x\in E\setminus N
 \text{ and }
 s\in \mathbb{Q}
 \text{ with }s>0
.
\end{equation*}
For each $x\in E\setminus N$, 
letting 
$
 s
\rightarrow
 \left(
 \frac{T[f^p](x)}{T[g^q](x)}
 \right)^{1/pq}
$
and obtain the desired conclusion.
\end{proof}

\begin{proof}
[Proof of Proposition \ref{Prop_Lp-cont-op}]
First, we extend the operator $T$ on $L^2$
to a bounded operator on $L^p$, $p\in [1,\infty]$.
The 
following arguments in Step 1 and 2 
are based on those in Section 4.1 of \cite{BA8478266X}:

\noindent
\textbf{Step 1}
(Extension to a $L^\infty$ operator)%
\textbf{.}

Take 
$\eta\in L^1(E; m)$ such that $\eta > 0$ $m$-a.e. and define
$\eta_n := (n\eta)\wedge 1$.
We 
extend the operator $T$ on $L^2(E; m)\cap L^\infty(E; m)$
to an operator on $L^\infty(E; m)$ by
\begin{equation*}
\begin{cases}
 Tf := \lim_{n\rightarrow \infty}T(f\eta_n)
&
 \text{for } f\in L^\infty_+ (E; m)
,
\\
 Tf := Tf^+ - Tf^-
&
 \text{for } f\in L^\infty (E; m)
.
\end{cases}
\end{equation*}
Here 
$
 L_+ ^\infty (E; m) 
= 
 \{
  f\in L^\infty(E; m) : f\geq 0 \text{ $m$-a.e.} 
 \}
$
and
$f^+ := f\vee 0$, $f^- := (-f)\vee 0$.
This extension is well-defined due to the Markov property 
\begin{equation*}
 0\leq Tf \leq 1 \text{ $m$-a.e.,}
 \text{ for all }f\in L^2(E; m)
 \text{ with }0\leq f\leq 1 \text{ $m$-a.e.}
\end{equation*}
of $T$, 
and is unique under the relation 
\begin{equation*}
 \int f (Tg) dm = \int (Tf) g dm
\quad
 \text{for all } 
 f\in L^1(E; m)\cap L^\infty(E; m)
 \text{ and }
 g\in L^\infty(E; m)
.
\end{equation*}


\noindent
\textbf{Step 2}
(Extension to a $L^1$ operator)%
\textbf{.}

Since 
$L^1(E; m)\cap L^\infty(E; m)$ is dense in $L^1(E; m)$,
we can also extend the operator $T$ on $L^1(E; m)\cap L^\infty(E; m)$
to an operator on $L^1(E; m)$. 
It is
easy to see that $T: L^1(E; m) \rightarrow L^1(E; m)$ is positivity preserving, 
Markovian and contractive.
By the Markov property of $T$, 
we 
can also see the following, 
\begin{equation*}
 \int f (Tg) dm = \int (Tf) g dm
\quad
 \text{for all } 
 f\in L^1(E; m)
 \text{ and }
 g\in L^\infty(E; m)
.
\end{equation*}


\noindent
\textbf{Step 3}
(Extension to a $L^p$ operator)%
\textbf{.}

By the Marcinkiewicz interpolation theorem,
we can extend $T$ to a bounded linear operator 
on $L^p(E; m)$, for any $p\in (1, \infty)$. 
\\

Next
we show that the contraction property $\|T\|_{p\rightarrow p}\leq 1$.
Let 
$f\in L^p(E; m)$ and take $E_n$ such that 
$m(E_n)<\infty$, $E_n\uparrow E$.
By 
Lemma \ref{Lem_Lp-cont} and the properties of 
the $L^1$ operator $T$, we have
\begin{align*}
 (T[f 1_{E_n}](x) )^p
\leq&
 (T[|f| 1_{E_n}](x) )^p
\\
\leq&
 T[|f|^p](x) 
 (T[1_{E_n} ^q](x) )^{p/q} 
\leq
 T[|f|^p](x) 
\end{align*}
for a.e. $x\in E$.
Taking 
integral and using the $L^1$-contractivity of $T$, 
we have
\begin{equation*}
 \|
  T[f 1_{E_n}]
 \|_p ^p
\leq
 \|T[|f|^p]\|_1
\leq
 \|f\|_p ^p
.
\end{equation*}
Letting $n\rightarrow \infty$, 
the dominate convergence theorem concludes
$
 \|T[f]\|_p
\leq
 \|f\|_p
$.
\end{proof}

\vspace{1eM}

\begin{proof}
[Proof of Proposition \ref{Prop_Lp-cont-semigrp}]
By Proposition \ref{Prop_Lp-cont-op}, we have already seen that
each $T_t$ can be extended to a contraction operator on $L^p(E; m)$,
and hence it is sufficient to show the continuity in $L^p$ 
with respect to $t$.

Fix 
$p\in (1, \infty)$, $f\in L^p(E; m)$ and $\varepsilon >0$.
Since 
$E$ is a Lusin space, $L^p(E; m)\cap C_b(E)$ is dense in $L^p(E; m)$.
Hence we can choose $f_\varepsilon \in L^1(E; m) \cap L^\infty(E; m)$
such that
$
 \|
  f - f_\varepsilon
 \|_{p}
<
 \varepsilon
$.
Write 
\begin{equation*}
 q
:=
 \begin{cases}
  2p & \text{when $p\geq2$},
 \\
  1  & \text{when $p<2$},
 \end{cases}
\end{equation*}
and take $\theta\in (0, 1]$ such that
$
 \frac{\theta}{2}
+
 \frac{1-\theta}{q}
=
 \frac{1}{p} 
$.
Then 
by H\"older's inequality and 
Proposition \ref{Prop_Lp-cont-op}, we have
\begin{align*}
 \|
  f - T_t f
 \|_{p}
\leq&
 \|
  f - f_\varepsilon
 \|_{p}
+
 \|
  f_\varepsilon - T_t f_\varepsilon
 \|_{p}
+
 \|
  T_t f_\varepsilon - T_t f
 \|_{p}
\\
\leq &
 2
 \|
  f - f_\varepsilon
 \|_{p}
+
 \|
  f_\varepsilon - T_t f_\varepsilon
 \|_{2} ^\theta
 \|
  f_\varepsilon - T_t f_\varepsilon
 \|_{q} ^{1-\theta}
\\
\leq&
 2\varepsilon
+
 \|
  f_\varepsilon - T_t f_\varepsilon
 \|_{2} ^\theta
 ( 
  2 \|f_\varepsilon\|_{q}
 )^{1-\theta} 
\end{align*}
and hence
$
 \limsup_{t\rightarrow 0}
 \|
  f - T_t f
 \|_{p}
\leq
 2\varepsilon
$.
Since $\varepsilon>0$ can be chosen arbitrary, 
we conclude
$
 \lim_{t\rightarrow 0}
 \|
  f - T_t f
 \|_{p}
=
 0
$.
\end{proof}

\subsubsection{Permutated tensor product}
\label{Sec_tensor}

In this subsection, 
we introduce a permutated tensor product of linear operators
and discuss about them.
Using this notation, we can simplify our proof of the main results.

Let $\mathfrak{S}_k$ be the symmetric group of degree $k$.
For bounded linear operators $T_1, \cdots, T_k$ on $L^2(E; m)$
and for $\sigma\in\mathfrak{S}_k$,
define two operators
$
 T_1 \otimes \cdots \otimes T_k
$
and
$
 T_1 \otimes \underset{\sigma}{\cdots} \otimes T_k
$
on $L^2(E^k; m^{\otimes k})$ by
\begin{align*}
 [
  T_1 \otimes \cdots \otimes T_k
 ]
 (
  g_1 \otimes \cdots \otimes g_k
 )
=&
 T_1 g_1
 \otimes
 \cdots
 \otimes
 T_k g_k
,
\\
 [
  T_1 \otimes \underset{\sigma}{\cdots} \otimes T_k
 ]
 (
  g_{1} \otimes \cdots \otimes g_{k}
 )
=&
 T_1 g_{\sigma(1)}
 \otimes
 \cdots
 \otimes
 T_k g_{\sigma(k)}
\end{align*}
for $g_1, \dots, g_k \in L^2(E; m)$.
We say the operator
$
 T_1 \otimes \underset{\sigma}{\cdots} \otimes T_k
$
the permutated tensor product of $T_1, \cdots, T_k$ with respect to $\sigma$.

In particular, for two bounded linear operators 
$S$, $T$ on $L^2(E; m)$ and $m\leq k$, 
we have 
\begin{align*}
&
 \Bigl(
  [
   S^{\otimes m}
  \underset{\sigma}{\otimes}
   T^{\otimes (k-m)}
  ]
 (
  g_1 \otimes \cdots \otimes g_k
 )
 \Bigr)
 (
  x_1, \cdots, x_k
 )
\\
=&
 S g_{\sigma(1)} (x_1)
 \cdots
 S g_{\sigma(m)} (x_m)
\cdot
 T g_{\sigma(m+1)} (x_{m+1})
 \cdots
 T g_{\sigma(k)} (x_k)
\\
=&
 U_1 g_1 (x_{\sigma^{-1}(1)})
 \cdots
 U_k g_k (x_{\sigma^{-1}(k)})
\\
=&
 \left(
  [U_1 \otimes \cdots \otimes U_k]
  (g_1 \otimes \cdots \otimes g_k)
 \right)
 (
  x_{\sigma^{-1}(1)}
 ,\cdots,
  x_{\sigma^{-1}(k)}
 )
,
\end{align*}
for $g_1, \dots, g_k \in L^2(E; m)$, where
\begin{equation}
\label{eq_tensorU}
 U_{j} = 
\begin{cases}
 S
&
 \text{when }
 \sigma^{-1}(j) \leq m
,
\\
 T
&
 \text{when }
 \sigma^{-1}(j) \geq m+1
,
\end{cases}
\end{equation}

and hence, for $F\in L^2(E^k; m^{\otimes k})$,
\begin{align*}
&
 \Bigl(
  [
   S^{\otimes m}
  \underset{\sigma}{\otimes}
   T^{\otimes (k-m)}
  ]
  F
 \Bigr)
 (
  x_1, \cdots, x_k
 )
=
 \left(
  [
   U_1 \otimes \cdots \otimes U_k
  ]
  F
 \right)
 (
  x_{\sigma^{-1}(1)}
 ,\cdots,
  x_{\sigma^{-1}(k)}
 )
.
\end{align*}

Furthermore, if $S$ and $T$ have kernels $s$ and $t$ respectively,
then by writing the kernel of $U_j$ as $u_j$,
we have
\begin{flalign*}
 \Bigl(
  [
   S^{\otimes m}
  \underset{\sigma}{\otimes}
   T^{\otimes (k-m)}
  ]
  F
 \Bigr)
 (
  x_1, \cdots, x_k
 )
=&
 \left(
  [
   U_1 \otimes \cdots \otimes U_k
  ]
  F
 \right)
 (
  x_{\sigma^{-1}(1)}
 ,\cdots,
  x_{\sigma^{-1}(k)}
 )
\\
=&
 \int_{E^k}
  \prod_{i=1} ^k 
  u_i(x_{\sigma^{-1}(i)}, y_i)
  F(y_1, \cdots, y_k)
 m(dy_1) \cdots m(dy_k)
\\
=&
 \int_{E^k}
  \prod_{i=1} ^m 
  s(x_{i}, y_{\sigma(i)})
  \prod_{i={m+1}} ^k 
  t(x_{i}, y_{\sigma(i)})
  F(y_1, \cdots, y_k)
 m(dy_1) \cdots m(dy_k)
.
&
\end{flalign*}


\subsubsection{Lemmas about nonnegative integer valued measures}

In
this subsection, we state some basic facts
about nonnegative integer valued measures.
From
now on, denote $\mathbb{Z}_{\geq 0}$ (resp. $\mathbb{Z}_{>0}$) 
as the set of nonnegative (resp. positive) integers.
\\

The following lemma is used in the computation
which yields \eqref{eq_55a-3} in Section \ref{Sec_55a}. 

\begin{Lem}
\label{Lem_meas1}
Let $\mathcal{X}$ be a finite set
,
$\pi$ be a $\mathbb{Z}_{\geq 0}$-valued measure on $\mathcal{X}$
and  
$
 f
:
 \mathbb{Z}_{\geq 0}\times \mathbb{Z}_{\geq 0}
\rightarrow 
 \mathbb{R}
$.
Write
\begin{equation*}
 \mathcal{M}_{\leq \pi}
:=
 \left\{
  \rho
 :
  \begin{array}{l}
   \text{$\rho$ is a $\mathbb{Z}_{\geq 0}$-valued measure on $\mathcal{X}$}
   \\
   \text{with $\rho(x)\leq \pi(x)$ for all $x\in \mathcal{X}$}
  \end{array}
 \right\}
.
\end{equation*}
Then, it holds that
\begin{equation*}
 \sum_{\rho \in \mathcal{M}_{\leq \pi}}
 \prod_{x\in \mathcal{X}}
 f(\pi(x), \rho(x))
=
 \prod_{x\in \mathcal{X}}
 \sum_{s=0} ^{\pi(x)}
 f(\pi(x), s)
.
\end{equation*}
\end{Lem}

\begin{proof}
We have
\begin{flalign*}
 \sum_{\rho \in \mathcal{M}_{\leq \pi}}
 \prod_{x\in \mathcal{X}}
 f(\pi(x), \rho(x))
=&
 \sum_
 {
  \substack
  {
   \{s_x\}_{x\in \mathcal{X}}
  ;
  \\
   0\leq s_x \leq \pi(x)
  \\
   \text{for all } x\in \mathcal{X}
  }
 }
 \prod_{x\in \mathcal{X}}
 f(\pi(x), s_x)
\\
=&
 \prod_{x\in \mathcal{X}}
 \Biggl(
  \sum_
  {
   0\leq s_x \leq \pi(x)
  }
  f(\pi(x), s_x)
 \Biggr)
=
 \prod_{x\in \mathcal{X}}
 \sum_{s=0} ^{\pi(x)}
 f(\pi(x), s)
.
&
\end{flalign*}
\end{proof}

The following two lemmas are used in the proof of 
Lemma \ref{Lem_meas5}.
They are obtained by easy inductions,
so we omit the proofs.  

\begin{Lem}
\label{Lem_meas2}
Let $\mathcal{X}$ be a finite set and 
$
 \{ b_i \}_{i=1} ^n
\subset
 \mathcal{X}
$,
$m\leq n$.
For 
$\sigma\in\mathfrak{S}_n$, 
we define measures
\begin{align*}
 \pi
:=
 \sum_{i=1} ^n
 \delta_{b_i}
,\quad
 \pi_m
:=
 \sum_{i=1} ^m
 \delta_{b_i}
,\quad
 \pi_m ^\sigma
:=
 \sum_{i=1} ^m
 \delta_{b_{\sigma(i)}}
.
\end{align*}

Then, it holds that
\begin{equation*}
 \left\{
   \rho
  \in
   \mathcal{M}_{\leq \pi}
  :
   \rho(\mathcal{X}) = m
 \right\}
=
 \{
  \pi_m ^\sigma
 ;
  \sigma\in\mathfrak{S}_n
 \}
.
\end{equation*}
\end{Lem}

\begin{Lem}
\label{Lem_meas3}
Let $\mathcal{X}$, $\mathcal{Y}$ be finite sets and 
$
 \{
  (a_i, b_i)
 \}_{i=1} ^n
\subset
 \mathcal{X} \times \mathcal{Y}
$.
For 
$\sigma\in \mathfrak{S}_n$, 
we define measures on $\mathcal{X} \times \mathcal{Y}$ 
\begin{align*}
 \pi
:=
 \sum_{i=1} ^n
 \delta_{(a_i, b_i)}
,\quad
 \pi^\sigma
:=
 \sum_{i=1} ^n
 \delta_{(a_i, b_{\sigma(i)})}
\end{align*}
and measure on $\mathcal{X}$ and $\mathcal{Y}$
\begin{align*}
 \pi_{\mathcal{X}}
:=
 \pi\circ (\mathrm{proj}_{\mathcal{X}} )^{-1} 
=
 \sum_{i=1} ^n
 \delta_{a_i}
,\quad
 \pi_{\mathcal{Y}}
:=
 \pi\circ (\mathrm{proj}_{\mathcal{Y}} )^{-1} 
=
 \sum_{i=1} ^n
 \delta_{b_i}
\end{align*}
respectively.
Set
\begin{equation*}
 \mathcal{M}(\pi_{\mathcal{X}}, \pi_{\mathcal{Y}})
:=
 \left\{
   \rho
  :
   \begin{array}{l}
    \rho \text{ is a }
    \mathbb{Z}_{\geq 0} \text{-valued measure on }
    \mathcal{X} \times \mathcal{Y}
   \\
    \text{with marginals }\pi_{\mathcal{X}} \text{ and }\pi_{\mathcal{Y}}
   \end{array}
 \right\}
.
\end{equation*}
Then, it holds that 
\begin{equation*}
 \mathcal{M}(\pi_{\mathcal{X}}, \pi_{\mathcal{Y}})
=
 \{
  \pi^\sigma
 ;
  \sigma\in\mathfrak{S}_n
 \}
.
\end{equation*}
\end{Lem}

\vspace{1eM}

The next Lemma \ref{Lem_meas4} is used only 
in the proof of Lemma \ref{Lem_56-2}.
\begin{Lem}
\label{Lem_meas4}
Suppose $\mathcal{X}$ be a finite set.
Let
$A$ be a $\mathbb{Z}_{\geq 0}$-valued measure on $\mathcal{X}^2$
with $n:=A(\mathcal{X}^2)<\infty$. 
Fix
$x_{n+1}\in \mathcal{X}$.
Then, it holds that 
\begin{equation*}
 \#
 \left\{
  \{x_j\}_{j=1} ^n
 \subset
  \mathcal{X}
 \left|
  A
 =
  \sum_{j=1} ^n
  \delta_{(x_j, x_{j+1})}
 \right.
 \right\}
\leq
 k
\cdot
 \frac
 {
  \displaystyle 
  \prod_{l_1 \in \mathcal{X}}
  \overline{A}(l_1) !
 }
 {
  \displaystyle 
  \prod_{l \in \mathcal{X}^2}
  A(l) !
 }
,
\end{equation*}
where $\overline{A}$ is a measure on $\mathcal{X}$
determined by
$
 \overline{A}(\cdot)
:=
 A(\cdot\times \mathcal{X}) 
$.
\end{Lem}

\begin{proof}
This can be proved by similar arguments as 
in Chapter II.2, p.17 of \cite{MR1739680}.
\end{proof}

\vspace{1eM}

Let 
$\mathcal{X}$ and $T$ be finite sets, 
$S^*$ be a subset of $T$
and $p$ be a positive integer.
For
each $i=1, \cdots, p$, fix 
$F'_i\subset T$ with $\#S^* = \#F'_i $
and fix a disjoint partition $\{S_1 ^*, S_2 ^*\}$ of $S^*$.
Let
$A$ and $r$ be $\mathbb{Z}_{\geq 0}$-valued measures 
on $\mathcal{X}^p$ satisfying
$
 A(\mathcal{X}^p)=\#S^*
$,
$
 r(\mathcal{X})=\#S_1 ^*
$
and
$
 r(x)\leq A(x)
$
for all
$
 x\in \mathcal{X}^p
$.
Let $\{a_j ^{(i)}\}_{j\in T} \subset \mathcal{X}$
and simply denote
\begin{equation*}
 a_j 
=
 \{a_j ^{(i)}\}_{i=1} ^p
\in 
 \mathcal{X}^p
,
\quad
 a_{\sigma(j)} 
=
 \{a_{\sigma_i (j)} ^{(i)}\}_{i=1} ^p
\in
 \mathcal{X}^p
\end{equation*}
for $j\in S^*$.
Write
\begin{equation*}
 \Psi_p (A, r, a)
:=
 \left\{
  (\sigma_i)_{i=1} ^p 
 \in 
  \prod_{i=1} ^p
  \mathrm{Bij}(S^*, F'_i )
 \left|
 \quad
  A
 =
  \sum_{j\in S^*}
  \delta_{a_{\sigma(j)}}
 ,\quad
  r
 =
  \sum_{j\in S_1 ^*}
  \delta_{a_{\sigma(j)}}
 \right.
 \right\}
,
\end{equation*}
where
for two sets $T_1$ and $T_2$,
$\mathrm{Bij}(T_1, T_2)$
is the set of bijections from $T_1$ to $T_2$.

The 
next lemma is used in the computation
which yields \eqref{eq_55a-1} in Section \ref{Sec_55a}. 

\begin{Lem}
\label{Lem_meas5}
If
$\Psi_p(A, r, a) \not= \varnothing$,
then it holds that
\begin{equation*}
 \# \Psi_p(A, r, a)
=
 \# S_1 ^*! \# S_2 ^*! 
 \frac
 {
  \displaystyle
  \prod_{i=1}^p
  \prod_{x^{(i)}\in \mathcal{X}}
  A_i(x^{(i)})!
 }
 {
  \displaystyle
  \prod_{x\in \mathcal{X}^p}
  A(x)!
 }
 \prod_{x\in \mathcal{X}}
 \binom{A(x)}{r(x)}
.
\end{equation*}
\end{Lem}

The next corollary is given for the comparison with 
(3.35) in \cite[Lemma 3.6]{MR2999298}.
It seems that 
$
 \prod_{i=1} ^p
 \#(W_i \setminus S^*)!
$
is missing in (3.35).

\begin{Cor}
\label{Cor_meas5}
Let $F_i, W_i \subset T$ with 
$S^* \subset W_i$ and $\# W_i = \# F_i$.
We 
similarly define
\begin{equation*}
 \widetilde{\Psi}_p (A, r, a)
:=
 \left\{
  (\sigma_i)_{i=1} ^p
 \in 
  \prod_{i=1} ^p
  \mathrm{Bij}(W_i , F_i )
 \left|
 \quad
  A
 =
  \sum_{j\in S^*}
  \delta_{a_{\sigma(j)}}
 ,\quad
  r
 =
  \sum_{j\in S_1 ^*}
  \delta_{a_{\sigma(j)}}
 \right.
 \right\}
.
\end{equation*}

If
$\widetilde{\Psi}_p (A, r, a) \not= \varnothing$,
then it holds that
\begin{equation*}
 \# \widetilde{\Psi}_p
 (A, r, a)
=
 \# S_1 ^*! \# S_2 ^*! 
 \frac
 {
  \displaystyle
  \prod_{i=1}^p
  \prod_{x^{(i)}\in \mathcal{X}}
  A_i(x^{(i)})!
 }
 {
  \displaystyle
  \prod_{x\in \mathcal{X}^p}
  A(x)!
 }
 \prod_{x\in \mathcal{X}}
 \binom{A(x)}{r(x)}
\cdot
 \prod_{i=1} ^p
 \#(W_i \setminus S^*)!
.
\end{equation*}
\begin{flushright}
 $\square$ 
\end{flushright}
\end{Cor}

The following proof is the same as the proof of 
\cite[Lemma 3.6]{MR2999298}.
\begin{proof}
[Proof of Lemma \ref{Lem_meas5}]
\ \par
\noindent
\textbf{Case 1; $p=1$}
\ \par

In this case, we have
\begin{align*}
 \Psi_1(A, r, a)
=&
 \left\{
  \sigma
 \in
  \mathrm{Bij}(S^*, F') 
 ;\quad
  A
 =
  \sum_{j\in S^*}
  \delta_
  {
   a_{\sigma (j)}
  }
 ,
\quad
  r
 =
  \sum_{j\in S_1^*}
  \delta_
  {
   a_{\sigma (j)}
  }
 \right\}
.
\end{align*}

First, fix $\sigma \in \mathrm{Bij}(S^*, F')$.
For $x\in \mathcal{X}$, write
$
 I_x
:=
 \{
  j\in S^*
 ;\quad
  a_{\sigma (j)} = x
 \}
$,
then $\{I_x\}_{x\in \mathcal{X}}$ is a disjoint partition of $S^*$.

There are 
$\# S_1 ^* !$ and $\# S_2 ^* !$ permutations in 
$S^* _1$ and $S^* _2$, respectively.
Take
$
 \tau_{1}\in \mathfrak{S}(S_1 ^*)
$,
$
 \tau_{2}\in \mathfrak{S}(S_2 ^*)
$
and set
$
 \tau = (\tau_1, \tau_2 )
\in
 \mathfrak{S}(S^*)
$.
For each 
$x\in \mathcal{X}$, there are $\binom{A(x)}{r(x)}$ ways
to decompose $I_x$ into $S^* _1$ and $S^* _2$.
Take
$
 f \in \mathfrak{S}(S^*)
 \text{ with }
 f(I_x) = I_x
$
for all $x$.
Then we find that
$
 \sigma \circ f \circ \tau
\in
 \Psi_1 (A, r, a)
$.
Except for the duplication, we obtain 
$
 \# S_1 ^* ! \# S_1 ^* ! 
 \prod_{x\in \mathcal{X}} 
 \binom{A(x)}{r(x)}
$
many elements of $\Psi_1(A, r, a)$ of the form
$
 \sigma \circ f \circ \tau
$.
Hence
\begin{equation*}
 \# \Psi_1(A, r, a)
\leq
 \# S_1 ^* ! \# S_1 ^* ! 
 \prod_{x\in \mathcal{X}} 
 \binom{A(x)}{r(x)}
.
\end{equation*}

For the converse inequality, 
fix $\sigma\in \Psi_1(A, r, a)$.
Then 
clearly we can decompose $\sigma$ as the above form,
and hence
%
\begin{equation*}
 \# \Psi_1(A, r, a)
\geq
 \# S_1 ^* ! \# S_1 ^* ! 
 \prod_{x\in \mathcal{X}} 
 \binom{A(x)}{r(x)}
.
\end{equation*}



\noindent
\textbf{Case 2; $p = 2$}
\ \par


We define the marginals of measures
\begin{equation*}
 A_1(\cdot)
:=
 A(\cdot\times \mathcal{X})
,\quad
 r_1(\cdot)
:=
 r(\cdot\times \mathcal{X})
,
\end{equation*}
\begin{equation*}
 A_2(\cdot)
:=
 A(\mathcal{X}\times\cdot)
,\quad
 r_2(\cdot)
:=
 r(\mathcal{X}\times\cdot)
,
\end{equation*}
and write
\begin{equation*}
 \Psi_1(A_1, r_1, a^{(1)})
:=
 \left\{
  \sigma_1
 \in
  \mathrm{Bij}(S^*, F') 
 ;\quad
  A_1
 =
  \sum_{j\in S^*}
  \delta_
  {
   a^{(i)} _{\sigma_1 (j)}
  }
 ,
\quad
  r
 =
  \sum_{j\in S_1^*}
  \delta_
  {
   a^{(i)} _{\sigma_1 (j)}
  }
 \right\}
.
\end{equation*}

\noindent
We claim that
\begin{equation}
\label{eq_claim}
 H
:=
 \left\{
  \sigma_1 \in \mathrm{Bij}(S^*, F'_1)
 ;
  \begin{array}{l}
   \text{ there exists }
   \sigma_2 \in \mathrm{Bij}(S^*, F'_2)
   ,
  \\
   \text{ such that }
   (\sigma_1, \sigma_2) 
  \in
   \Psi_2 (A, r, a)
  \end{array}
 \right\}
=
 \Psi_1(A_1, r_1, a^{(1)})
.
\end{equation}

We first prove $H \subset \Psi_1(A_1, r_1, a^{(1)})$.
Fix 
$\sigma_1$ and $\sigma_2$ as in $H$.
We have
\begin{align*}
 A_1(x^{(1)})
=&
 \sum_{x^{(2)} \in \mathcal{X}}
 A(x^{(1)}, x^{(2)})
\\
=&
 \sum_{x^{(2)} \in \mathcal{X}}
 \#
 \{
  j\in S^*
 ;\quad
  x^{(1)} = a_{\sigma_1(j)} ^{(1)}
 ,
  x^{(2)} = a_{\sigma_2(j)} ^{(2)}
 \}
\\
=&
 \#
 \{
  j\in S^*
 ;\quad
  x^{(1)} = a_{\sigma_1(j)} ^{(1)}
 \}
.
\end{align*}
By the same way, 
we obtain the similar equality about $r_1$.

We next prove $H\supset \Psi_1(A_1, r_1, a^{(1)})$.
Fix 
$\sigma_1 \in \Psi_1(A_1, r_1, a^{(1)})$.
Since
$\Psi_2(A, r, a)\not=\varnothing$, we can choose 
$\sigma_2\in \mathrm{Bij}(S^*, F'_2)$ 
such that
\begin{equation*}
 A_2 
=
 \sum_{j\in S^*}
 \delta_
 {
  a^{(2)} _{\sigma_2 (j)}
 }
.
\end{equation*}
By 
the definition of $\Psi_1(A_1, r_1,a^{(1)})$, 
we have
\begin{align*}
 A_1
=
 \sum_{j\in S^*}
 \delta_{a^{(1)}_{\sigma_1(j)} }
,\quad
 r_1
=
 \sum_{j\in S_1^*}
 \delta_{a^{(1)}_{\sigma_1(j)} }
.
\end{align*}
First, 
there exists a permutation
(see Lemma \ref{Lem_meas2})
$\tau_1 \in \mathfrak{S}(S^*)$ such that
\begin{align*}
 A_2
=
 \sum_{j\in S^*}
 \delta_{a^{(2)}_{\sigma_2\circ\tau_1(j)} }
,\quad
 r_2
=
 \sum_{j\in S_1^*}
 \delta_{a^{(2)}_{\sigma_2\circ\tau_1(j)} }
.
\end{align*}
Second,
there exists a permutation 
(see Lemma \ref{Lem_meas3})
$\tau'_2 \in \mathfrak{S}(S_1^*)$
such that
\begin{align*}
 r_1
=
 \sum_{j\in S_1 ^*}
 \delta_{a^{(1)} _{\sigma_1(j)} }
,\quad
 r_2
=
 \sum_{j\in S_1 ^*}
 \delta_
 {
  a^{(1)} _
  {
   \sigma_2\circ \tau_1\circ \tau'_2 (j)
  } 
 }
,\quad
 r
=
 \sum_{j\in S_1 ^*}
 \delta_
 {
  (
   a^{(1)} _{\sigma_1(j)} 
  ,
   a^{(2)} _{\sigma_2\circ \tau_1 \circ \tau_2' (j)} 
  ) 
 }
.
\end{align*}
Similarly, 
we can take a permutation 
$\tau'' _2\in \mathfrak{S}(S_2 ^*)$
which have the same property as $\tau' _2$. 
Finally, we find that 
$
 \widetilde{\sigma}_2
:=
 \sigma _2
 \circ
 \tau_1
 \circ
 (\tau' _2, \tau'' _2)
\in
 \mathrm{Bij}(S^*, F'_2)
$
is a required bijection.

\vspace{1eM}

By the above claim, we obtain
\begin{align*}
 \#\Psi_2(A, r, a)
=&
 \sum_{\sigma_1 \in \mathrm{Bij}(S^*, F'_1)}
 \#
 \{
  \sigma_2
 \in
  \mathrm{Bij}(S^*, F'_2)
 ;\quad
  (\sigma_1, \sigma_2)
  \in
  \Psi_2(A,r,a)
 \}
\\
=&
 \sum_
 {
  \sigma_1 
 \in 
  \Psi_1(A_1, r_1, a^{(1)}) 
 }
 \hspace{-5mm}
 \#
 \{
  \sigma_2
 \in
  \mathrm{Bij}(S^*, F'_2)
 ;\quad
  (\sigma_1, \sigma_2)
  \in
  \Psi_2(A,r,a)
 \}
.
\end{align*}

From now on, fix 
$
 \sigma_1 
\in
 \Psi_1(A_1, r_1, a^{(1)})
$
and we will calculate
$
 \#
 \{
  \sigma_2
 \in
  \mathrm{Bij}(S^*, F'_2)
 ;
  (\sigma_1, \sigma_2)
  \in
  \Psi_2(A,r,a)
 \}
$.

First, we construct such $\sigma_2$'s.
For each $x^{(1)} \in \mathcal{X}$,
take a disjoint partition
$\{D, \overline{D}\}$ of 
$
 \{
  j\in S^*
 ;
  a_{\sigma_1(j)} ^{(1)}
 =
  x^{(1)}
 \}
$
such that
\begin{align*}
 \{
  j\in S_1^*
 ;
  a_{\sigma_1(j)} ^{(1)}
 =
  x^{(1)}
 \}
=&
 \bigcup_{x^{(2)}\in \mathcal{X}}
 D (x^{(1)}, x^{(2)} )
,\quad
 \#
 D (x^{(1)}, x^{(2)} )
=
 r (x^{(1)}, x^{(2)} )
,
\\
 \{
  j\in S_2^*
 ;
  a_{\sigma_1(j)} ^{(1)}
 =
  x^{(1)}
 \}
=&
 \bigcup_{x^{(2)} \in \mathcal{X}}
 \overline{D} (x^{(1)}, x^{(2)} )
,\quad
 \#
 \overline{D} (x^{(1)}, x^{(2)} )
=
 A (x^{(1)}, x^{(2)} )
-
 r (l^{(1)}, l^{(2)} )
.
\end{align*}

There are
\begin{equation*}
 \prod_{x^{(1)} \in \mathcal{R}^2}
 \left(
  \frac
  {r_1(x^{(1)}) !}
  {
   \prod_{x^{(2)}\in \mathcal{X}}
   r(x^{(1)}, x^{(2)} ) !
  }
  \frac
  {(A-r)_1(x^{(1)}) !}
  {
   \prod_{x^{(2)} \in \mathcal{X}}
   (A-r)(x^{(1)}, x^{(2)} ) !
  }
 \right)
\end{equation*}
ways 
to choose the couple $\{D, \overline{D}\}$.
Note that 
for $x^{(2)}\in \mathcal{X}$,
\begin{equation*}
 \#
 \Biggl\{
  \bigcup_{x^{(1)} \in \mathcal{X}}
  \left(
   D (x^{(1)}, x^{(2)} )
   \cup
   \overline{D} (x^{(1)}, x^{(2)} )
  \right)
 \Biggr\}
=
 A_2 (x^{(2)})
.
\end{equation*}
For fixed $\{D, \overline{D}\}$, 
take $\sigma_2\in \mathrm{Bij}(S^*, F'_2)$ such that
\begin{equation*}
  \{
   j\in S^*
  ;\quad
   a_{\sigma_2(j)  } ^{(2)}
  =
   x^{(2)}
  \}
=
 \bigcup_{x^{(1)} \in \mathcal{X}}
 \left(
  D (x^{(1)}, x^{(2)} )
  \cup
  \overline{D} (x^{(1)}, x^{(2)} )
 \right)
\quad
 \text{for all } x^{(2)} \in \mathcal{X}
.
\end{equation*}
There are
$
 \prod_{x^{(2)} \in \mathcal{X}}
  A_2(x^{(2)} ) !
$
ways to choose such $\sigma_2$.
Then we have 
$
 (\sigma_1, \sigma_2) 
\in 
 \Psi_2(A, r, a)
$. 
Indeed, we can find that for
$
 (x^{(1)}, x^{(2)}) \in \mathcal{X}^2
$,
\begin{align*}
 \{
  j\in S_1 ^*
 ;
  a_{\sigma_i(j)} ^{(i)}
 =
  x^{(i)} 
  \hspace{2mm}
  \text{for }i = 1, 2
 \}
=&
 \bigcap_{i=1,2}
 \{
  j\in S_1 ^*
 ;
  a_{\sigma_i(j)} ^{(i)}
 =
  x^{(i)} 
 \}
=
 D(x^{(1)}, x^{(2)})
.
\end{align*}
Similarly we can find that
$
 \{
  j\in S_2 ^*
 ;
  a_{\sigma_i(j)} ^{(i)}
 =
  x^{(i)} 
  \hspace{2mm}
  \text{for }i = 1, 2
 \}
=
 \overline{D}(x^{(1)}, x^{(2)})
$.
Hence we have
\begin{align*}
&
 \#
 \{
  \sigma_2
 \in
  \mathrm{Bij}(S^*, F'_2)
 ;
  (\sigma_1, \sigma_2)
  \in
  \Psi_2(A,r,a)
 \}
\\
\geq&
 \prod_{x^{(1)} \in \mathcal{X}}
 \left(
  \frac
  {r_1(x^{(1)}) !}
  {
   \prod_{x^{(2)}\in \mathcal{X}}
   r(x^{(1)}, x^{(2)} ) !
  }
  \frac
  {(A-r)_1(x^{(1)}) !}
  {
   \prod_{x^{(2)}\in \mathcal{X}}
   (A-r)(x^{(1)}, x^{(2)} ) !
  }
 \right)
 \prod_{x^{(2)} \in \mathcal{X}}
  A_2(x^{(2)} ) !
.
\end{align*}
We can 
show the converse of the above inequality. 
Indeed, for 
$(\sigma_1, \sigma_2)\in \Psi_2(A, r, a)$,
write
\begin{align*}
 D(x^{(1)}, x^{(2)})
:=&
 \{
  j\in S_1^*
 ;
  a_{\sigma_i(j)} ^{(i)}
 =
  x^{(i)} 
 \quad
  \text{for }i = 1, 2
 \}
\\
 \overline{D}(x^{(1)}, x^{(2)})
:=&
 \{
  j\in S_2^*
 ;
  a_{\sigma_i(j)} ^{(i)}
 =
  x^{(i)} 
 \quad
  \text{for }i = 1, 2
 \}
\end{align*}
and obtain the same partition.


Therefore, we have
\begin{align*}
&
 \#\Psi_2(A, r, a)
\\
=&
 \sum_
 {
  \sigma_1 
 \in
  \Psi_1(A_1, r_1, a^{(1)}) 
 }
 \#
 \{
  \sigma_2
 \in
  \mathrm{Bij}(S^*, F'_2)
 ;\quad
  (\sigma_1, \sigma_2)
  \in
  \Psi_2(A,r,a)
 \}
\\
=&
 \#\Psi_1(A_1, r_1, a^{(1)})
 \prod_{x^{(1)} \in \mathcal{X}}
 \left(
  \frac
  {r_1(x^{(1)}) !}
  {
   \prod_{x^{(2)}\in \mathcal{X}}
   r(x^{(1)}, x^{(2)} ) !
  }
  \frac
  {(A-r)_1(x^{(1)}) !}
  {
   \prod_{x^{(2)}\in \mathcal{X}}
   (A-r)(x^{(1)}, x^{(2)} ) !
  }
 \right)
 \prod_{x^{(2)} \in \mathcal{X}}
  A_2(x^{(2)} )!
\\
\begin{split}
=&
 m_1 ! m_3 !
 \frac
 {
  \prod_{x^{(1)}\in \mathcal{X}}
  A_1 (x^{(1)}) !
 }
 {
  \prod_{x^{(1)}\in \mathcal{X}}
  r_1 (x^{(1)}) !
  \prod_{x^{(1)}\in \mathcal{X}}
  (A_1 -r_1) (x^{(1)}) !
 }
\\
&
\qquad
 \frac
 {
  \prod_{x^{(1)} \in \mathcal{X}}
  r_1(x^{(1)}) !
 }
 {
  \prod_{x^{(1)} \in \mathcal{X}}
  \prod_{x^{(2)} \in \mathcal{X}}
  r(x^{(1)}, x^{(2)} ) !
 }
 \frac
 {
  \prod_{x^{(1)} \in  \mathcal{X}}
  (A-r)_1(x^{(1)}) !
 }
 {
  \prod_{x^{(1)} \in  \mathcal{X}}
  \prod_{x^{(2)}\in \mathcal{X}}
  (A-r)(x^{(1)}, x^{(2)} ) !
 }
 \prod_{x^{(2)} \in  \mathcal{X}}
  A_2(x^{(2)} )!
\end{split}
\\
=&
 m_1 ! m_3 !
 \frac
 {
  \prod_{x^{(1)}\in \mathcal{X}}
  A_1 (x^{(1)}) !
  \prod_{x^{(2)}\in \mathcal{X}}
  A_2 (x^{(2)}) !
 }
 {
  \prod_{x^{(1)} \in  \mathcal{X}}
  \prod_{x^{(2)}\in \mathcal{X}}
  r(x^{(1)}, x^{(2)} ) !
  \prod_{x^{(1)} \in  \mathcal{X}}
  \prod_{x^{(2)}\in \mathcal{X}}
  (A-r)(x^{(1)}, x^{(2)} ) !
 }
\\
=&
 m_1 ! m_3 !
 \frac
 {
  \prod_{x^{(1)}\in \mathcal{X}}
  A_1 (x^{(1)}) !
  \prod_{x^{(2)}\in \mathcal{X}}
  A_2 (x^{(2)}) !
 }
 {
  \prod_{x \in \mathcal{X}^2 }
  A(x) !
 }
 \prod_{x \in \mathcal{X}^2 }
 \binom{ A(x) }{ r(x) }
.
\end{align*}


\vspace{1eM}

We can prove inductively for the case $p\geq 3$.
\end{proof}


\section{Proof of Proposition \ref{Prop_existence} }
\label{Sec_existence}

In this section, we give the proof of Proposition \ref{Prop_existence}.
The following proof is given by the same strategy 
as the proof of \cite[Theorem 2.2.3]{MR2584458}.

In the following, 
we abbreviate the measure $m(dx)$ just as $dx$.

Note that
$
 \{t < \zeta^{(i)}\}
=
 \{ X^{(i)}_t\in E\}
$.
For $f\in C_b (E)$, we recall that
\begin{equation*}
 \langle 
  f
 ,
  \ell_{\varepsilon, t} ^{\mathrm{IS}}
 \rangle
=
 \int_{E}
  f(x)
  \left[
   \int_{[0,t]^p}
    \prod_{i=1} ^p
    p^{(i)}_\varepsilon (X^{(i)}(s_i), x)
   ds_1 \cdots ds_p
  \right]
 m(dx)
.
\end{equation*}

\noindent
For each $i$, define
\begin{equation}
\label{Def_Ht(i)}
 H_t ^{(i)}
 (x_1, \cdots, x_k)
:=
 \int_{[0,t]^k}
 \int_E
  1_
  {
   \left\{
    \sum_{j=1} ^k r_j \leq t
   \right\}
  }
  \prod_{j=1} ^{k+1}
  p^{(i)} _{r_j}
  (x_{j-1}, x_{j})
 dx_{k+1}
 dr_1 \cdots dr_k
,
\end{equation}
where
$
 r_{k+1} 
=
 t - \sum_{j=1} ^k r_j
$.
Then we find that $H_t ^{(i)}\in L^p(E^k)$. 
Indeed,
\begin{align*}
&
 \int_{E^k}
  \biggl[
   \int_{[0,t]^k}
   \int_E
    1_
    {
     \left\{
      \sum_{j=1} ^k r_j \leq t
     \right\}
    }
    \prod_{j=1} ^{k+1}
    p^{(i)} _{r_j}
    (x_{j-1}, x_{j})
   dx_{k+1}
   dr_1 \cdots dr_k
  \biggr]^p
 dx_1 \cdots dx_k
\\
\leq&
 \int_{E^k}
  \biggl[
   \int_{[0,t]^k}
    1_
    {
     \left\{
      \sum_{j=1} ^k r_j \leq t
     \right\}
    }
    \prod_{j=1} ^{k}
    p^{(i)} _{r_j}
    (x_{j-1}, x_{j})
   dr_1 \cdots dr_k
  \biggr]^p
 dx_1 \cdots dx_k
\\
\leq&
 e^{pt}
 \int_{E^k}
  \biggl[
   \int_{[0,t]^k}
    1_
    {
     \left\{
      \sum_{j=1} ^k r_j \leq t
     \right\}
    }
    \prod_{j=1} ^{k}
    e^{-r_j}
    p^{(i)} _{r_j}
    (x_{j-1}, x_{j})
   dr_1 \cdots dr_k
  \biggr]^p
 dx_1 \cdots dx_k
\\
\leq&
 e^{pt}
 \left[
  \sup_{y\in E}
  \int_E
   \Gr{i} (x, y)^p
  dx
 \right]^k
\\
<&
 \infty
.
\end{align*}

By Proposition \ref{Prop_Lp-cont-op}, 
we have for 
$\sigma_1, \cdots, \sigma_p\in\mathfrak{S}_k$,
\begin{flalign}
\notag
 \int_{E^k}
   f^{\otimes k}
  \prod_{i=1} ^p
  \left[
   [
    T_\varepsilon ^{(i)}
   \otimes
    \underset{\sigma_i}{\cdots}
   \otimes
    T_\varepsilon ^{(i)}
   ]
   H_t ^{(i)}
  \right]
 dm^{\otimes k}
\leq&
 \|f\|_{\infty} ^k
 \prod_{i=1}^p
 \|
  [
   T_\varepsilon ^{(i)}
  \otimes
   \underset{\sigma_i}{\cdots}
  \otimes
   T_\varepsilon ^{(i)}
  ]
  H_t ^{(i)}
 \|_{L^p(E^k)}
\\
\leq&
\label{eq_uni.int}
 \|f\|_{\infty} ^k
 \prod_{i=1}^p
 \|
  H_t ^{(i)}
 \|_{L^p(E^k)}
<
 \infty
.
&
\end{flalign}

Next, we denote
\begin{equation}
\label{eq_0t<}
 [0, t]_< ^k
:=
 \{
  (s_1, \cdots, s_k)\in [0, t]^k
 :
  s_1 < \cdots < s_k
 \}
\end{equation}
and we regard $s_0 =0$, $s_{k+1} = t$ and 
$\sigma(0) = 0$, 
$\sigma(k+1) = k+1$ for $\sigma\in \mathfrak{S}_k$.
Then we have
\begin{flalign}
\notag
&
 \mathbb{E}_{x_0}
 \left[
  \langle
   f
  ,
   \ell_{\varepsilon, t}^{\mathrm{IS}}
  \rangle^k
 ;
  t < \zeta^{(1)} \wedge \cdots \wedge \zeta^{(p)}
 \right]
\\
\notag
=&
 \mathbb{E}_{x_0}
 \left[
  \langle
   f
  ,
   \ell_{\varepsilon, t}^{\mathrm{IS}}
  \rangle^k
 ;
  X^{(i)}(t) \in E \text{ for all }i
 \right]
\\
\notag
=&
 \int_{E^k}
  \prod_{l=1} ^k
  f(x_l)
  \prod_{i=1} ^p
  \left\{
   \int_{[0,t]^k}
    \mathbb{E}_{x_0}
    \left[
     \prod_{j=1} ^k
     p^{(i)} _\varepsilon 
     (X^{(i)}(s_j), x_j)
    ;
     X^{(i)}(t) \in E
    \right]
   ds_1 \cdots ds_k
  \right\}
 dx_1 \cdots dx_k
\\
\notag
\begin{split}
=&
 \int_{E^{k}}
  \prod_{l=1} ^k
   f(x_l)
  \prod_{i=1} ^p
  \left[
   \int_{E^{k+1}}
    \left(
     \prod_{j=1} ^{k}
     p^{(i)} _\varepsilon 
     (z_j , x_j)
    \right)
    \left(
     \sum_{\sigma\in\mathfrak{S}_k}
     \int_{[0,t]^k _<}
      \prod_{j=1} ^{k+1}
      p^{(i)} _{s_j - s_{j-1}} 
      (z_{\sigma(j-1)}, z_{\sigma(j)})
     ds_1\cdots ds_k
    \right)
  \right.
\\
&
\hspace{125mm}
   dz_1\cdots dz_{k+1}
  \Biggr]
 dx_1\cdots dx_k
\end{split}
\\
\notag
=&
 \int_{E^k}
  \prod_{l=1} ^k
  f(x_l)
  \prod_{i=1} ^p
  \left[
   \int_{E^{k}}
    \left(
     \prod_{j=1} ^k
     p^{(i)} _\varepsilon 
     (z_j , x_j)
    \right)
    \left(
     \sum_{\sigma\in\mathfrak{S}_k}
     H_t ^{(i)}
     (z_{\sigma^{-1}(1)}, \cdots, z_{\sigma^{-1}(k)})
    \right)
   dz_1\cdots dz_{k}
  \right]
 dx_1\cdots dx_k
\\
\notag
=&
 \int_{E^k}
   f^{\otimes k}
  \prod_{i=1} ^p
  \left[
   \sum_{\sigma\in\mathfrak{S}_k}
   [
    T_\varepsilon ^{(i)}
   \otimes
    \underset{\sigma}{\cdots}
   \otimes
    T_\varepsilon ^{(i)}
   ]
   H_t ^{(i)}
  \right]
 dm^{\otimes k}
\\
\label{eq_convergence}
\rightarrow&
 \int_{E^k}
   f^{\otimes k}
  \prod_{i=1} ^p
  \left[
   \sum_{\sigma\in\mathfrak{S}_k}
   [
    \mathrm{id} 
   \otimes
    \underset{\sigma}{\cdots}
   \otimes
    \mathrm{id} 
   ]
   H_t ^{(i)}
  \right]
 dm^{\otimes k}
\quad
 \text{as }\varepsilon \rightarrow 0
.
&
\end{flalign}

\noindent
By a similar argument of \eqref{eq_convergence}, 
we also have
\begin{flalign*}
&
 \mathbb{E}_{x_0}
 \left[
 |
  \langle
   f
  ,
   \ell_{\varepsilon, t} ^{\mathrm{IS}}
  \rangle
 -
  \langle
   f
  ,
   \ell_{\varepsilon', t} ^{\mathrm{IS}}
  \rangle
 |^2
 ;
  t < \zeta^{(1)} \wedge \cdots \wedge \zeta^{(p)}
 \right]
\\
\begin{split}
=&
 \int_{E^2}
  f^{\otimes 2}
  \prod_{i=1} ^p
  \left[
   \sum_{\sigma\in\mathfrak{S}_2}
   [
    T_\varepsilon ^{(i)}
   \underset{\sigma}{\otimes}
    T_\varepsilon ^{(i)}
   ]
   H_t ^{(i)}
  \right]
 dm^{\otimes 2}
+
 \int_{E^2}
  f^{\otimes 2}
  \prod_{i=1} ^p
  \left[
   \sum_{\sigma\in\mathfrak{S}_2}
   [
    T_{\varepsilon'} ^{(i)}
   \underset{\sigma}{\otimes}
    T_{\varepsilon'} ^{(i)}
   ]
   H_t ^{(i)}
  \right]
 dm^{\otimes 2}
\\
&
\hspace{70mm}
-
 2
 \int_{E^2}
  f^{\otimes 2}
  \prod_{j=1} ^p
  \left[
   \sum_{\sigma\in\mathfrak{S}_2}
   [
    T_{\varepsilon} ^{(i)}
   \underset{\sigma}{\otimes}
    T_{\varepsilon'} ^{(i)}
   ]
   H_t ^{(i)}
  \right]
 dm^{\otimes 2}
\end{split}
&
\end{flalign*}
and this converges to zero as 
$\varepsilon$ and $\varepsilon'$ tend to zero.
Hence we have shown that 
$
 \{
  \langle
   f
  ,
   \ell_{\varepsilon, t}^{\mathrm{IS}}
  \rangle
 \}_\varepsilon
$
is Cauchy in $L^k(\widetilde{\mathbb{P}}_{t})$, for $k = 2$.
By 
using H\"older's inequality and \eqref{eq_uni.int},
the case $k\not=2$ is derived from the case $k=2$.
We
thus obtain for all $k$,
\begin{equation}
\label{eq_convergence_in_moments}
 \widetilde{\mathbb{E}}_{t}
 \left[
 |
  \langle
   f
  ,
   \ell_{\varepsilon, t} ^{\mathrm{IS}}
  \rangle
 -
  \langle
   f
  ,
   \ell_{\varepsilon', t} ^{\mathrm{IS}}
  \rangle
 |^k
 \right]
\rightarrow
 0
\quad
 \text{as } \varepsilon, \varepsilon'\downarrow 0
.
\end{equation}



Next, we will show the existence of the vague limit.
Recall that
one can choose a subset 
$\{f_n\}_n$ of $C_K ^+ (E)$
such that
\begin{equation}
\label{eq_metric}
 d_{\{f_n\}}(\mu, \nu)
:=
 \sum_{n=1} ^\infty
 \frac{1}{2^n}
 \{
  |
   \langle
    \mu, f_n 
   \rangle
  -
   \langle
    \nu, f_n 
   \rangle
  |
 \wedge
  1
 \}
\quad
 \text{for }
 \mu, \nu \in \mathcal{M}(E)
\end{equation}
metrizes the vague topology of $\mathcal{M}(E)$,
the space of Radon measures on $E$.
(See Section 15.7 of \cite{MR818219} for example.)
In the following, we fix such $\{f_n\}$ 
and simply denote $d=d_{\{f_n\}}$.

By \eqref{eq_convergence_in_moments},
we can find that
$
 \widetilde{\mathbb{E}}_{t}
 \left[
  d
  (
   \ell_{\varepsilon, t} ^{\mathrm{IS}}
  ,
   \ell_{\varepsilon', t} ^{\mathrm{IS}}
  )
 \right]
\rightarrow
 0
$
as $\varepsilon, \varepsilon'\rightarrow 0$.
Take a sequence $\varepsilon_l \downarrow 0$ such that
$
 \widetilde{\mathbb{P}}_{t}
 (
  d
  (
   \ell_{\varepsilon_l, t} ^{\mathrm{IS}}
  ,
   \ell_{\varepsilon_{l+1}, t} ^{\mathrm{IS}}
  )
 >
  2^{-l}
 )
<
 2^{-l}
$
for all $l$.
Since $(\mathcal{M}(E), d)$ is a complete
metric space, there exists a random measure 
$\ell_{t} ^\mathrm{IS}\in \mathcal{M}(E)$ such that
\begin{equation}
\label{eq_a.e.}
 d
 (
  \ell_{\varepsilon_l, t} ^{\mathrm{IS}}
 ,
  \ell_{t} ^{\mathrm{IS}}
 )
\rightarrow
 0
; \quad
 \widetilde{\mathbb{P}}_{t}
\text{ -a.e.},
\text{ as }l\rightarrow\infty
.
\end{equation}

\noindent
In particular,
$
 \ell_{\varepsilon_l, t} ^{\mathrm{IS}}
$
converges to
$
 \ell_{t} ^{\mathrm{IS}}
$
as $l \rightarrow \infty$ in distribution.

Now we can find that the limit $\ell_t ^{\mathrm{IS}}$
is independent of the choice of the sequence $\varepsilon_l$.
Indeed, 
for a given $f\in C_K ^+(E)$ and a bounded Lipschitz function 
$G:\mathbb{R}\rightarrow \mathbb{R}$, we have
by \eqref{eq_convergence_in_moments},
\begin{align*}
&
 |
  \widetilde{\mathbb{E}}_t
  G
  (
   \langle \ell_{\varepsilon, t} ^\mathrm{IS}, f\rangle
  )
 -
  \widetilde{\mathbb{E}}_t
  G
  (
   \langle \ell_{t} ^\mathrm{IS}, f\rangle
  )
 | 
\\
\leq&
 \text{Lip}(G)
 \widetilde{\mathbb{E}}_t
 |
  \langle \ell_{\varepsilon, t} ^{\mathrm{IS}}, f\rangle
 -
  \langle \ell_{\varepsilon_l, t} ^{\mathrm{IS}}, f\rangle
 | 
+
 |
  \widetilde{\mathbb{E}}_t
  G\langle \ell_{\varepsilon_l, t} ^{\mathrm{IS}}, f\rangle
  -
  \widetilde{\mathbb{E}}_t
  G\langle \ell_{t} ^{\mathrm{IS}}, f\rangle
 | 
\rightarrow 
 0
.
\end{align*}

\noindent
and hence 
$
 \ell_{\varepsilon, t} ^{\mathrm{IS}}
$
converges to 
$
 \ell_{t} ^{\mathrm{IS}}
$
in distribution.
Here, we used the fact that a sequence of 
random measures $\{\xi_n\}$ converges in distribution, 
if and only if, the sequence of the integral 
$\{\langle f, \xi_n  \rangle\}$ 
converges in distribution,
for any function $f$ of $C_K ^+ (E)$. 
(See for example \cite[Theorem 16.16]{MR1876169}.)

To show the finiteness of $\ell_t^{\mathrm{IS}}$,
take $g_n\in C_K ^+(E)$ with $g_n \uparrow 1_E$.
We regard
$
 \widetilde{\mathbb{E}}_{t}
 \left[
  \langle
   \cdot, \ell_{t} ^{\mathrm{IS}}
  \rangle
 \right]
$
as a Radon measure on $E$.
By combining Fatou's lemma with 
\eqref{eq_uni.int} and \eqref{eq_convergence}, 
we have
\begin{align*}
 \widetilde{\mathbb{E}}_{t}
 \left[
  \langle
   1_E, \ell_{t} ^{\mathrm{IS}}
  \rangle
 \right]
\leq
 \liminf_{n\rightarrow\infty}
 \widetilde{\mathbb{E}}_{t}
 \left[
  \langle
   g_n, \ell_{t} ^{\mathrm{IS}}
  \rangle
 \right]
\leq&
 \prod_{i=1}^p
 \|
  H_t ^{(i)}
 \|_{L^p(E)}
\cdot
 \mathbb{P}_{t}
 (
  t < \zeta^{(1)} \wedge \cdots \wedge \zeta^{(p)}
 )^{-1}
<
 \infty
.
\end{align*}

Finally,
to prove the second half of the claim,
fix $f\in C_K ^+ (E)$. 
By \eqref{eq_a.e.}, we can take a sequence 
$\varepsilon_l \downarrow 0$
such that
$
 \langle
  \ell_{\varepsilon_l, t} ^{\mathrm{IS}}
 ,
  f
 \rangle
$
converges to
$
 \langle
  \ell_{t} ^{\mathrm{IS}}
 ,
  f
 \rangle
$
for
$
 \widetilde{\mathbb{P}}_{t}
\text{-a.e.}
$,
as $l\rightarrow \infty$.
By 
combining Fatou's lemma with 
\eqref{eq_uni.int} and \eqref{eq_convergence}, 
we have
\begin{align*}
 \widetilde{\mathbb{E}}_{t}
 \left[
  \langle
   f, \ell_{t} ^{\mathrm{IS}}
  \rangle^{2k}
 \right]
\leq&
 \liminf_{l\rightarrow \infty}
 \widetilde{\mathbb{E}}_{t}
 \left[
  \langle
   f, \ell_{\varepsilon_l, t} ^{\mathrm{IS}}
  \rangle^{2k}
 \right]
\\
\leq&
 \sup_{l}
 \widetilde{\mathbb{E}}_{t}
 \left[
  \langle
   f, \ell_{\varepsilon_l, t} ^{\mathrm{IS}}
  \rangle^{2k}
 \right]
\leq
 \|f\|_{\infty} ^{2k}
 \prod_{i=1}^p
 \|
  H_t ^{(i)}
 \|_{L^p(E^{2k})}
\cdot
 \mathbb{P}_{t}
 (
  t < \zeta^{(1)} \wedge \cdots \wedge \zeta^{(p)}
 )^{-1}
.
\end{align*}
These inequalities imply the uniform integrability
of 
$
 \{
  |
   \langle
    f, \ell_{\varepsilon_l, t} ^{\mathrm{IS}}
   \rangle
  -
   \langle
    f, \ell_{t} ^{\mathrm{IS}}
   \rangle
  |^k
 \}
$
and hence 
$
 \langle
  f, \ell_{\varepsilon_l, t} ^{\mathrm{IS}}
 \rangle
$
converges to
$
 \langle
  f, \ell_{t} ^{\mathrm{IS}}
 \rangle
$
in $L^k(\widetilde{\mathbb{P}}_t)$.
Therefore, 
by combining with \eqref{eq_convergence_in_moments},
we conclude \eqref{eq_convergence_any_moment}.
\begin{flushright}
\vspace{-1.6eM}
 $\square$ 
\end{flushright}


\section{Proof of Theorem \ref{Thm_LDP} }
\label{Sec_4}

In this section, 
we give the proof of Theorem \ref{Thm_LDP},
assuming Proposition \ref{Prop_exp_app}.
Our proof is 
based on the proofs in Section 2.2, \cite{MR2999298}.  

First, 
recall that each $t^{-1}\ell_t ^{(i)}$ satisfies the
large deviation principle as $t\rightarrow \infty$,
with probability $\widetilde{\mathbb{P}}_t$,
scale $t$ and the good rate function $J^{(i)}$
(Theorem \ref{Thm_LDP_FOT}).
For each $\varepsilon > 0$, 
define the continuous mappings 
$
 p_\varepsilon ^{(i)}
:
 \mathcal{M}_1(E)\rightarrow \mathcal{M}_{\leq 1}(E)
$,
$i=1, \cdots, p$ and
$
 \Phi
:
 (\mathcal{M}_{1}(E) )^p
\rightarrow
 \mathcal{M}_f(E) 
\times
 (\mathcal{M}_{\leq 1}(E) )^p
$
by
\begin{equation*}
 \langle
  f
 ,
  p_\varepsilon ^{(i)}[\mu]
 \rangle
=
 \int_E
  f(x)
  \biggl[
   \int_E
   p_\varepsilon ^{(i)} (x, y) \mu(dy)
  \biggr]
 m(dx)
\quad
 \text{for }
 \mu\in\mathcal{M}_1(E)
,
 f\in C_b(E)
,
\end{equation*}
\begin{equation*}
 \Phi(\mu_1, \cdots, \mu_1)
:=
 \left(
  \Bigl(
  \prod_{i=1} ^p
  \frac
  {d p_\varepsilon ^{(i)} [\mu_i]}
  {dm}
  \Bigr)
  dm
 ;
 \hspace{2mm}
  p_\varepsilon ^{(1)}  [\mu_1]
 ,
  \cdots
 ,
  p_\varepsilon ^{(p)}  [\mu_p]
 \right)
\quad
 \text{for }
 \left(
  \mu_1
 ,
  \cdots
 ,
  \mu_p
 \right)
\in
 (\mathcal{M}_{1}(E))^p
.
\end{equation*}

By the contraction principle 
(Theorem \ref{Thm_DZ_contraction}), 
we find that the tuple 
$
 (
  t^{-p}\ell_{\varepsilon, t} ^{\mathrm{IS}} 
 ;
  t^{-1}\ell_{\varepsilon, t} ^{(1)}
 ,
  \cdots
 ,
  t^{-1}\ell_{\varepsilon, t} ^{(p)}
 )
$
satisfies the LDP
as $t\rightarrow \infty$, 
with probability $\widetilde{\mathbb{P}}_t$,
scale $t$ and the good rate function 
$\mathbf{J}_\varepsilon$ which is defined by
\begin{align*}
&
 \mathbf{J}_\varepsilon
 (
  \nu
 ;
  \nu_1, \cdots, \nu_p
 )
\\
:=&
 \inf
 \biggl\{
  \sum_{i=1} ^p
  J^{(i)}(\mu_i)
 ;
   \left(
    \mu_1, \cdots, \mu_p
   \right)
  \in
   (\mathcal{M}_{1} (E) )^p
  ,
   \Phi 
   \left(
    \mu_1, \cdots, \mu_p
   \right)
  =
   \left(
    \nu
   ;
    \nu_1, \cdots, \nu_p
   \right)
 \biggr\}
\\
=&
 \inf
 \biggl\{
  \sum_{i=1} ^p
  J^{(i)} (\mu_i) 
 ;\quad
  \mu_i \in \mathcal{M}_1(E)
 , 
  p_\varepsilon ^{(i)} [\mu_i] 
 =
  \nu_i
 ,
  \prod_{i=1} ^p
  \frac{d p_\varepsilon ^{(i)} [\mu_i]}{dm}
 =
  \frac{d\nu}{dm}
 \biggr\}
\\
=&
 \inf
 \left\{
  \sum_{i=1} ^p
  \{
   \mathcal{E}^{(i)}
   (\psi_i, \psi_i)
  -
   \lambda_1 ^{(i)}
  \}
 ;\quad
  \mu_i \in \mathcal{M}_1(E)
 , 
  p_\varepsilon ^{(i)}[\mu_i] 
 =
  \nu_i
 ,
  \prod_{i=1} ^p
  \frac{d p_\varepsilon ^{(i)} [\mu_i]}{dm}
 =
  \frac{d\nu}{dm}
 ,
  \psi_i = \sqrt{\frac{d\mu_i}{dm}} \in \mathcal{F}^{(i)}
 \right\}
\end{align*}
for 
$
 \left(
  \nu
 ;
  \nu_1
 ,
  \cdots
 ,
  \nu_p
 \right)
\in
 \mathcal{M}_f(E)
\times 
 (\mathcal{M}_{\leq 1} (E) )^p
$.
\\

Until 
the end of this section,
we fix the same  
$
 \{ f_n\}_{n=1} ^\infty 
\subset
 C_K ^+ (E)
$ 
as in \eqref{eq_metric},
and define a metric $\mathbf{d}$ 
on the product space
$
 \mathcal{M}_f(E)
\times
 (
  \mathcal{M}_{\leq 1}(E)
 )^p
$
as
\begin{equation*}
 \mathbf{d}
 \bigl(
  (\mu; \mu_1, \cdots, \mu_p) 
 ,
  (\nu; \nu_1, \cdots, \nu_p) 
 \bigr)
:=
 d(\mu, \nu)
+
 \sum_{i=1} ^p
 d(\mu_i, \nu_i)
.
\end{equation*}
We denote 
$\mathbf{B}_\delta$ as the open ball
of radius $\delta > 0$
in
$
 \mathcal{M}_f(E)
\times
 (
  \mathcal{M}_{\leq 1}(E)
 )^p
$
with respect to the metric $\mathbf{d}$,
and denote
$B_\delta$ as the open ball of radius $\delta > 0$
in
$\mathcal{M}_f(E)$
or
$\mathcal{M}_{\leq 1}(E)$
with respect to the metric $d$.

Fix $\delta > 0$.
We have
\begin{align*}
&
 \widetilde{\mathbb{P}}_{t}
 \left(
  \mathbf{d}
  \bigl(
   (
    t^{-p} 
    \ell_{t} ^{\mathrm{IS}}
   ,
    t^{-1} 
    \ell_{t} ^{(1)}
   ,
    \cdots
   ,
    t^{-1} 
    \ell_{t} ^{(p)}
   )
  ,
   (
    t^{-p} 
    \ell_{\varepsilon, t} ^{\mathrm{IS}}
   ,
    t^{-1} 
    \ell_{\varepsilon, t} ^{(1)}
   ,
    \cdots
   ,
    t^{-1} 
    \ell_{\varepsilon, t} ^{(p)}
   )
  \bigr)
 >
  (p+1)
  \delta
 \right)
\\
\leq&
 \widetilde{\mathbb{P}}_{t}
 \left(
  d
  (
   t^{-p} 
   \ell_{t} ^{\mathrm{IS}}
  ,
   t^{-p} 
   \ell_{\varepsilon, t} ^{\mathrm{IS}}
  )
 >
  \delta
 \right)
+
 \sum_{i=1} ^p
 \widetilde{\mathbb{P}}_{t}
 \left(
  d
  (
   t^{-1} 
   \ell_{t} ^{(i)}
  ,
   t^{-1} 
   \ell_{\varepsilon, t} ^{(i)}
  )
 >
  \delta
 \right)
.
\end{align*}
In order to prove 
Theorem \ref{Thm_LDP}, it is sufficient to show 
the following Proposition \ref{Prop_4-1} and Lemma \ref{Lem_4-2}.
(See \cite[Theorem 4.2.16]{MR1619036} for instance.) 

\begin{Prop}
\label{Prop_4-1}
Let $\mathbf{J}$ be as in \eqref{eq_bfJ}.
The following three statements hold:

\vspace{-2mm}
\begin{enumerate}
\item
For every $\mu \in \mathcal{M}_{f}(E)$ and 
$\mu_1, \cdots, \mu_p \in \mathcal{M}_{1}(E)$,
it holds that
\begin{equation}
\label{eq_Prop_4-1-1}
 \sup_{\delta>0}
 \liminf_{\varepsilon \downarrow 0}
 \inf_{\mathbf{B}_\delta (\mu ; \mu_1, \cdots, \mu_p)}
 \mathbf{J}_\varepsilon
=
 \mathbf{J}(\mu; \mu_1, \cdots, \mu_p)
.
\end{equation}

\item
For any closed set
$
 F
\subset 
 \mathcal{M}_f(E)
\times
 (
  \mathcal{M}_{\leq 1}(E)
 )^p
$,
it holds that
\begin{equation}
\label{eq_Prop_4-1-2}
 \inf_{(\mu; \mu_1, \cdots, \mu_p) \in F}
 \mathbf{J}(\mu; \mu_1, \cdots, \mu_p)
\leq
 \limsup_{\varepsilon \downarrow 0}
 \inf_{(\mu ; \mu_1, \cdots, \mu_p)\in F}
 \mathbf{J}_\varepsilon(\mu; \mu_1, \cdots, \mu_p)
.
\end{equation}

\item
$\mathbf{J}$
is a good rate function on
$
 \mathcal{M}_f(E)
\times 
 (\mathcal{M}_{1}(E))^p
$.
\end{enumerate}

\end{Prop}


\begin{Lem}
\label{Lem_4-2}
\ \par

\vspace{-2mm}
\begin{enumerate}
\item
For each $i=1, \cdots, p$ and for any $\delta >0$,
it holds that
\begin{equation*}
 \lim_{\varepsilon\rightarrow 0}
 \limsup_{t\rightarrow \infty}
 \frac{1}{t}
 \log
 \widetilde{\mathbb{P}}_{t}
 \left(
  d
  (
   t^{-1} 
   \ell_{t} ^{(i)}
  ,
   t^{-1} 
   \ell_{\varepsilon, t} ^{(i)}
  )
 >
  \delta
 \right)
=
 -\infty
.
\end{equation*}

\item
For given $f\in C_K(E)$ and $\delta>0$,
it holds that
\begin{equation}
\label{eq_Lem_4-2-2}
 \lim_{\varepsilon\rightarrow 0}
 \limsup_{t \rightarrow \infty}
 \frac{1}{t}
 \log
 \widetilde{\mathbb{P}}_{t}
 \left(
  |
  \langle 
   t^{-p}
   (\ell_t ^{\mathrm{IS}} - \ell_{\varepsilon, t} ^{\mathrm{IS}})
   , f
  \rangle
  |
 >
  \delta
 \right)
=
 -\infty
.
\end{equation}
\end{enumerate}

\end{Lem}

\begin{proof}
[Proof of Proposition \ref{Prop_4-1}]
We 
first prove the upper bound of \eqref{eq_Prop_4-1-1}.
Let 
$
 (\mu; \mu_1, \cdots, \mu_p)
\in
 \mathcal{M}_f(E)
\times (\mathcal{M}_{\leq 1}(E))^p
$
with
$
 \mathbf{J}
 (\mu; \mu_1, \cdots, \mu_p)
<
 \infty
$
be given.
Take 
$\psi_i \in \mathcal{F}^{(i)}$ such that
$
 \psi_i = \sqrt{\frac{d\mu_i}{dm} } \in \mathcal{F}^{(i)}
$
and
$
 \prod_{i=1} ^p \psi_i ^2
=
 \prod_{i=1} ^p
 \frac{d\mu_i}{dm}
=
 \frac{d\mu}{dm} 
$.
Fix 
$\delta > 0$ and take $\varepsilon >0$ so small 
such that
$
 p_\varepsilon ^{(i)} [\mu_i]dm
\in
 B_{\delta/2} (\mu_i)
$
and that
$
 \bigl(
  \prod_{i=1} ^p 
  \frac
  {p_\varepsilon ^{(i)} [\mu_i]}
  {dm}
 \bigr)
 dm
\in
 B_{\delta/2p} (\mu)
$.
This is possible, indeed, 
because of the inclusion
$
 \mathcal{F}^{(i)} 
\subset
 L^2(E;m) \cap L^{2p}(E;m)
$
(assertion 3 of Lemma \ref{Lem_pre})
and Proposition \ref{Prop_Lp-cont-semigrp},
we have
for fixed $f\in C_K (E)$,
\begin{align*}
 |
  \langle f, p_\varepsilon ^{(i)}[\mu_i] \rangle 
 -
  \langle f, \mu_i \rangle 
 |
=&
 |
  \langle 
   f
  ,
   p_\varepsilon ^{(i)}[\psi_i ^2] dm
  \rangle 
 -
  \langle f, \psi_i ^2 dm \rangle 
 |
\\
\leq&
 \|f\|_q 
 \cdot
 \|
  p_\varepsilon ^{(i)} [\psi_i ^2] 
 -
  \psi_i ^2
 \|_p
\rightarrow
 0
 \text{ as } \varepsilon \rightarrow 0
.
\end{align*}

\noindent
By
H\"older's inequality and the $L^p$-contractivity of 
$p_\varepsilon ^{(i)}$ (Proposition \ref{Prop_Lp-cont-semigrp}),
it holds that
\begin{align}
\notag
&
 \left\|
  \prod_{i=1} ^p 
  p_\varepsilon ^{(i)}[ \psi_i ^2 ]
 -
  \prod_{i=1} ^p 
  \psi_i ^2 
 \right\|_1
\\
\notag
\begin{split}
\leq&
 \left\|
  \prod_{i=1} ^p 
  p_\varepsilon ^{(i)}[ \psi_i ^2 ]
 -
  \psi_1 ^2 
  \prod_{i=2} ^p 
  p_\varepsilon ^{(i)}[ \psi_i ^2 ]
 \right\|_1
+
 \left\|
  \psi_1 ^2 
  \prod_{i=2} ^p 
  p_\varepsilon ^{(i)}[ \psi_i ^2 ]
 -
  \psi_1 ^2 
  \psi_2 ^2 
  \prod_{i=3} ^p 
  p_\varepsilon ^{(i)}[ \psi_i ^2 ]
 \right\|_1
+
\cdots
\\
&
\hspace{80mm}
 \cdots
+
 \left\|
  \prod_{i=1} ^{p-1} 
  \psi_i ^2 
  p_\varepsilon ^{(p)}[ \psi_p ^2 ]
 -
  \prod_{i=1} ^p 
  \psi_i ^2 
 \right\|_1
\end{split}
\\
\notag
\leq&
 \sum_{i=1} ^p
 \Biggl(
  \prod_{l< i}
  \|
   \psi_l ^2
  \|_p
 \cdot
  \|
   p_\varepsilon ^{(i)} [\psi_i ^2] - \psi_i ^2 
  \|_p
 \cdot
  \prod_{l>i}
  \|
   p_\varepsilon ^{(l)} [\psi_l ^2] 
  \|_p
 \Biggr)
\\
\label{eq_Prop4-1-3}
\leq&
 \sum_{i=1} ^p
 \Biggl(
  \|
   p_\varepsilon ^{(i)} [\psi_i ^2] - \psi_i ^2 
  \|_p
  \prod_{l\not=i} 
  \|\psi_l \|_{2p} ^2
 \Biggr)
.
\end{align}

\noindent
Hence we have
\begin{align*}
 \left|
  \langle 
   f
  ,
   \Bigl(
    \prod_{i=1} ^p 
    \frac
    {d p_\varepsilon ^{(i)} [\mu_i]}
    {dm}
    dm
   \Bigr)
  \rangle 
 -
  \langle f, \mu \rangle 
 \right|
=&
 \left|
  \langle 
   f
  ,
   \prod_{i=1} ^p 
   p_\varepsilon ^{(i)}[ \psi_i ^2 ]
   dm
  \rangle 
 -
  \langle 
   f
  , 
   \prod_{i=1} ^p 
   \psi_i ^2 
   dm
  \rangle 
 \right|
\\
\leq&
 \|f\|_\infty
 \cdot
 \left\|
  \prod_{i=1} ^p 
  p_\varepsilon ^{(i)}[ \psi_i ^2 ]
 -
  \prod_{i=1} ^p 
  \psi_i ^2 
 \right\|_1
\\
\leq&
 \|f\|_\infty
 \cdot
 \sum_{i=1} ^p
 \Biggl(
  \|
   p_\varepsilon ^{(i)} [\psi_i ^2] - \psi_i ^2 
  \|_p
  \prod_{l\not=i} 
  \|\psi_l \|_{2p} ^2
 \Biggr)
\\
\rightarrow&
 0
 \text{ as } \varepsilon \rightarrow 0
.
\end{align*}

\noindent
We thus obtain
$
 \left(
  \bigl(
   \prod_{i=1} ^p 
   \frac
   {p_\varepsilon ^{(i)} [\mu_i]}
   {dm}
  \bigr)
  dm
 ;
  \hspace{3mm}
  p_\varepsilon ^{(1)} [\mu_1]dm
 ,
  \cdots
 ,
  p_\varepsilon ^{(p)} [\mu_p]dm
 \right)
\in
 B_\delta (\mu; \mu_1, \cdots, \mu_p)
$
and conclude
\begin{align*}
 \inf_{\mathbf{B}_\delta(\mu; \mu_1, \cdots, \mu_p)}
 \textbf{J}_\varepsilon 
\leq&
 \textbf{J}_\varepsilon 
 \left(
  \Bigl(
   \prod_{i=1} ^p 
   \frac
   {p_\varepsilon ^{(i)} [\mu_i]}
   {dm}
  \Bigr)
  dm
 ;
  \hspace{3mm}
  p_\varepsilon ^{(1)} [\mu_1]dm
 ,
  \cdots
 ,
  p_\varepsilon ^{(p)} [\mu_p]dm
 \right)
\leq
 \mathbf{J}(\mu; \mu_1, \cdots, \mu_p)
.
\end{align*}


\vspace{1eM}

We next 
prove the lower bound of \eqref{eq_Prop_4-1-1}.
Let 
$
 (\mu; \mu_1, \cdots, \mu_p)
\in
 \mathcal{M}_f(E) \times (\mathcal{M}_{1}(E))^p
$
be given.
Without loss of generality, we may assume that
$
 \sup_{\delta>0}
 \liminf_{\varepsilon \downarrow 0}
 \inf_{\mathbf{B}_\delta (\mu ; \mu_1, \cdots, \mu_p)}
 \mathbf{J}_\varepsilon 
<
 \infty
$.

For $\delta > 0$ and $\varepsilon > 0$, we pick 
$
 \bigl(
  \nu^{(\delta, \varepsilon)}
 ;
  \nu_1^{(\delta, \varepsilon)}
 ,
  \cdots
 ,
  \nu_p^{(\delta, \varepsilon)}
 \bigr)
\in
 \mathbf{B}_\delta
 (\mu; \mu_1, \cdots, \mu_p)
$
such that
\begin{equation}
\label{eq_41-1}
 \inf_
 {
  B_\delta
  (\mu; \mu_1, \cdots, \mu_p)
 }
 \mathbf{J}_\varepsilon 
\geq
 \mathbf{J}_\varepsilon 
 \bigl(
  \nu^{(\delta, \varepsilon)}
 ;
  \nu_1^{(\delta, \varepsilon)}
 ,
  \cdots
 ,
  \nu_p^{(\delta, \varepsilon)}
 \bigr)
-
 \delta
.
\end{equation}

\noindent
By the definition of 
$\mathbf{J}_\varepsilon$, 
there are nonnegative, $L^2$-normalized 
$
 \psi_i ^{(\delta, \varepsilon)}
\in
 \mathcal{F}^{(i)}
$
for $i=1,\cdots, p$ such that
$
 \frac{d\nu_i ^{(\delta, \varepsilon)}}{dm}
=
 p_\varepsilon ^{(i)}
 [(\psi_i ^{(\delta, \varepsilon)})^2]
$,
$
 \frac{d\nu ^{(\delta, \varepsilon)}}{dm}
=
 \prod_{i=1} ^p
 p_\varepsilon ^{(i)}
 [(\psi_i ^{(\delta, \varepsilon)})^2]
$
and
$
 \mathbf{J}_\varepsilon
 \bigl(
  \nu^{(\delta, \varepsilon)}
 ;
  \nu_1^{(\delta, \varepsilon)}
 ,
  \cdots
 ,
  \nu_p^{(\delta, \varepsilon)}
 \bigr)
\geq
 \sum_{i=1} ^p
 \{
  \mathcal{E}^{(i)}
  (
   \psi_i ^{(\delta, \varepsilon)}
  ,
   \psi_i ^{(\delta, \varepsilon)}
  )
 -
  \lambda_1 ^{(i)}
 \}
-
 \varepsilon
$.
In particular, 
$
 \{
  \psi_i ^{(\delta, \varepsilon)}
 \}_{\varepsilon}
$
is bounded in 
$
 (
  \mathcal{F}^{(i)}
 ,
  \|\cdot \|_{\mathcal{E}^{(i)}_1}
 )
$.
By 
taking a subsequence, there exists 
$
 \psi_i ^{(\delta)}
\in
 \mathcal{F}^{(i)}
$
with
$
 \|\psi_i ^{(\delta)}\|_2 = 1
$,
such that
$
 \psi_i ^{(\delta, \varepsilon)}
$
converges to 
$
 \psi_i ^{(\delta)}
$,
as $\varepsilon\rightarrow 0$ 
and in
$
 L^2(E;m) \text{ and } L^{2p}(E;m)
$
and that
$
 \liminf_{\varepsilon\rightarrow 0}
 \mathcal{E}^{(i)}
 (
  \psi_i ^{(\delta, \varepsilon)}
 ,
  \psi_i ^{(\delta, \varepsilon)}
 )
\geq
 \mathcal{E}^{(i)}
 (
  \psi_i ^{(\delta)}
 ,
  \psi_i ^{(\delta)}
 )
$,
where we used \eqref{eq_compact2p} in Lemma \ref{Lem_pre}.
By taking 
$
 \liminf_{\varepsilon\rightarrow 0}
$
to \eqref{eq_41-1}, we have
\begin{equation}
\label{eq_41-2}
 \liminf_{\varepsilon\rightarrow 0}
 \inf_{\mathbf{B}_\delta} \mathbf{J}_\varepsilon
\geq
 \sum_{i=1} ^p
 \{
  \mathcal{E}^{(i)}
  (
   \psi_i ^{(\delta)}
  ,
   \psi_i ^{(\delta)}
  )
 -
  \lambda_1 ^{(i)}
 \}
-
 \delta
.
\end{equation}

Now, we can see that
\begin{equation*}
 \nu_i ^{(\delta, \varepsilon)}
\rightarrow
 \mu_i ^{(\delta)}
:=
 (\psi_i ^{(\delta)})^2 dm
\quad
 \text{ vaguely in } 
 \mathcal{M}_{\leq 1}(E)
,
 \text{ as }
 \varepsilon \rightarrow 0
,
\end{equation*}
because, for fixed $f\in C_K (E)$,
\begin{align*}
 \left|
  \langle f, \nu_i ^{(\delta, \varepsilon)} \rangle
 -
  \langle f, {\mu_i ^{(\delta)}} \rangle
 \right|
=&
 \left|
  \langle
   f
  ,
   p_\varepsilon ^{(i)} 
   [(\psi_i ^{(\delta, \varepsilon)})^2] 
  \rangle
 -
  \langle 
   f
  ,
   (\psi_i ^{(\delta)})^2
  \rangle
 \right|
\\
\leq&
 \left|
  \langle
   f
  ,
   p_\varepsilon ^{(i)} 
   [(\psi_i ^{(\delta, \varepsilon)})^2] 
  \rangle
 -
  \langle
   f
  ,
   p_\varepsilon ^{(i)} 
   [(\psi_i ^{(\delta)})^2] 
  \rangle
 \right|
+
 \left|
  \langle
   f
  ,
   p_\varepsilon ^{(i)} 
   [(\psi_i ^{(\delta)})^2] 
  \rangle
 -
  \langle f, (\psi_i ^{(\delta)})^2 \rangle
 \right|
\\
\leq&
 \|f\|_q 
 \cdot
 \|
  (\psi_i ^{(\delta, \varepsilon)})^2
 -
  (\psi_i ^{(\delta)})^2
 \|_p
+
 \|f\|_q
 \cdot
 \|
  p_\varepsilon ^{(i)}[(\psi_i ^{(\delta)})^2]
 -
  (\psi_i ^{(\delta)})^2
 \|_p
\\
\rightarrow&
 0
,
 \text{ as }
 \varepsilon \rightarrow 0
.
\end{align*}
Hence we have 
$\mu_i ^{(\delta)} \in B_\delta(\mu_i)$.
Similarly we can obtain
$\mu ^{(\delta)} \in B_\delta(\mu)$.


Now we let $\delta\downarrow 0$ and
by taking a subsequence of 
$\psi_i ^{(\delta)}$, there exists 
$
 \psi_i 
\in
 \mathcal{F}^{(i)}
$
such that
$
 \psi_i ^{(\delta)}
\rightarrow
 \psi_i 
$
in
$L^2(E;m)$ and $L^{2p}(E;m)$
and
\begin{equation}
\label{eq_41-3}
 \liminf_{\delta\rightarrow 0}
 \mathcal{E}^{(i)}
 (
  \psi_i ^{(\delta)}
 ,
  \psi_i ^{(\delta)}
 )
\geq
 \mathcal{E}^{(i)}
 (
  \psi_i 
 ,
  \psi_i
 )
,
\end{equation}
where we used again \eqref{eq_compact2p} in Lemma \ref{Lem_pre}.

Since $\mu_i ^{(\delta)} \in B_\delta (\mu_i)$, 
we have for fixed $f\in C_K(E)$,
\begin{align*}
 |
  \langle f, \psi_i ^2 \rangle
 -
  \langle f, \mu_i \rangle
 |
\leq&
 |
  \langle f, \psi_i ^2 \rangle
 -
  \langle f, (\psi_i ^{(\delta)} )^2 \rangle
 |
+
 |
  \langle f, \mu_i ^{(\delta)} \rangle
 -
  \langle f, \mu_i \rangle
 |
\\
\leq&
 \|f\|_q
 \cdot
 \|
  {\psi_i}^2 - (\psi_i ^{(\delta)} )^2 
 \|_p
+
 C \delta
\\
\rightarrow&
 0
 \text{ as }
 \delta \rightarrow 0
.
\end{align*}
Hence 
$\psi_i ^2$ must be a density of $\mu_i$,
and similarly, 
$\prod_{i=1} ^p \psi_i ^2$ 
must be a density of $\mu$.

Therefore, 
by the definition of $\mathbf{J}$ and
\eqref{eq_41-3}, we have
\begin{equation*}
 \liminf_{\delta\rightarrow 0}
 \sum_{i=1}^p
 \{
  \mathcal{E}^{(i)}
  (
   \psi_i ^{(\delta)}
  ,
   \psi_i ^{(\delta)}
  )
 -
  \lambda_1 ^{(i)}
 \}
\geq
 \mathbf{J}(\mu; \mu_1, \cdots, \mu_p)
.
\end{equation*}
Combining 
this with \eqref{eq_41-2}, we obtain \eqref{eq_Prop_4-1-1}.


\vspace{1eM}

We next prove \eqref{eq_Prop_4-1-2}.
Without loss of generality, 
we may assume that
$
 \sup_{\varepsilon < \varepsilon'}
 \inf_{(\mu ; \mu_1, \cdots, \mu_p) \in F}
 \mathbf{J}_\varepsilon 
<
 \infty
$
for some $\varepsilon'>0$.

We repeat 
the argument as in the proof of the upper bound 
of \eqref{eq_Prop_4-1-1}.
For 
$\delta > 0$ and $\varepsilon < \varepsilon'$, 
we pick 
$
 \bigl(
  \nu^{(\delta, \varepsilon)}
 ;
  \nu_1^{(\delta, \varepsilon)}
 ,
  \cdots
 ,
  \nu_p^{(\delta, \varepsilon)}
 \bigr)
\in
 F
$
such that
\begin{equation}
\label{eq_41-4}
 \inf_
 {
  (\mu; \mu_1, \cdots, \mu_p)
 \in
  F
 }
 \mathbf{J}_\varepsilon 
\geq
 \mathbf{J}_\varepsilon 
 \bigl(
  \nu^{(\delta, \varepsilon)}
 ;
  \nu_1^{(\delta, \varepsilon)}
 ,
  \cdots
 ,
  \nu_p^{(\delta, \varepsilon)}
 \bigr)
-
 \delta
.
\end{equation}

By the definition of 
$\mathbf{J}_\varepsilon$, 
there are nonnegative, $L^2$-normalized 
$
 \psi_i ^{(\delta, \varepsilon)}
\in
 \mathcal{F}^{(i)}
$
for $i=1,\cdots, p$ such that
$
 \frac{d\nu_i ^{(\delta, \varepsilon)}}{dm}
=
 p_\varepsilon ^{(i)}
 [(\psi_i ^{(\delta, \varepsilon)}) ^2]
$,
$
 \frac{d\nu ^{(\delta, \varepsilon)}}{dm}
=
 \prod_{i=1} ^p
 p_\varepsilon ^{(i)}
 [(\psi_i ^{(\delta, \varepsilon)}) ^2]
$
and
$
 \mathbf{J}_\varepsilon
 \bigl(
  \nu^{(\delta, \varepsilon)}
 ;
  \nu_1^{(\delta, \varepsilon)}
 ,
  \cdots
 ,
  \nu_p^{(\delta, \varepsilon)}
 \bigr)
\geq
 \sum_{i=1} ^p
 \{
  \mathcal{E}^{(i)}
  (
   \psi_i ^{(\delta, \varepsilon)}
  ,
   \psi_i ^{(\delta, \varepsilon)}
  )
 -
  \lambda_1 ^{(i)}
 \}
-
 \varepsilon
$.
In particular, 
$
 \{
  \psi_i ^{(\delta, \varepsilon)}
 \}_{\varepsilon}
$
is bounded in 
$
 (
  \mathcal{F}^{(i)}
 ,
  \|\cdot \|_{\mathcal{E}^{(i)}_1}
 )
$.

Then, by taking further subsequences,
there exists
$
 \psi_i ^{(\delta)}
\in
 \mathcal{F}^{(i)}
$
such that
$
 \|\psi_i ^{(\delta)}\|_2 = 1
$,
$
 \psi_i ^{(\delta, \varepsilon)}
\rightarrow
 \psi_i ^{(\delta)}
$
in
$L^2(E;m)$ and $L^{2p}(E;m)$
and that
$
 \liminf_{\varepsilon\rightarrow 0}
 \mathcal{E}^{(i)}
 (
  \psi_i ^{(\delta, \varepsilon)}
 ,
  \psi_i ^{(\delta, \varepsilon)}
 )
\geq
 \mathcal{E}^{(i)}
 (
  \psi_i ^{(\delta)}
 ,
  \psi_i ^{(\delta)}
 )
,
$
where we used 
\eqref{eq_compact2p} in Lemma \ref{Lem_pre}.
By taking 
$
 \liminf_{\varepsilon\rightarrow 0}
$
to \eqref{eq_41-4}, we have
\begin{equation}
\label{eq_41-5}
 \liminf_{\varepsilon\rightarrow 0}
 \inf_{F} \mathbf{J}_\varepsilon
\geq
 \sum_{i=1} ^p
 \{
  \mathcal{E}^{(i)}
  (
   \psi_i ^{(\delta)}
  ,
   \psi_i ^{(\delta)}
  )
 -
  \lambda_1 ^{(i)}
 \}
-
 \delta
.
\end{equation}

Since
\begin{equation*}
 \nu_i ^{(\delta, \varepsilon)}
\rightarrow
 \mu_i ^{(\delta)}
:=
 (\psi_i ^{(\delta)})^2 dm
\quad
 \text{ vaguely in } 
 \mathcal{M}_{\leq 1}(E)
,
 \text{ as }
 \varepsilon \rightarrow 0
\end{equation*}
and
\begin{equation*}
 \nu ^{(\delta, \varepsilon)}
\rightarrow
 \mu ^{(\delta)}
:=
 \prod_{i=1} ^p
 (\psi_i ^{(\delta)})^2 
 dm
\quad
 \text{ vaguely in } 
 \mathcal{M}_{f}(E)
,
 \text{ as }
 \varepsilon \rightarrow 0
,
\end{equation*}
we have 
$
 (
  \mu^{(\delta)}; \mu_1 ^{(\delta)}, \cdots, \mu_p ^{(\delta)}
 )
\in
 F
$.
Hence \eqref{eq_41-5} implies that
\begin{equation*}
 \limsup_{\varepsilon\rightarrow 0}
 \inf_{F} \mathbf{J}_\varepsilon
\geq
 \liminf_{\varepsilon\rightarrow 0}
 \inf_{F} \mathbf{J}_\varepsilon
\geq
 \sum_{i=1} ^p
 \{
  \mathcal{E}^{(i)}
  (
   \psi_i ^{(\delta)}
  ,
   \psi_i ^{(\delta)}
  )
 -
  \lambda_1 ^{(i)}
 \}
-
 \delta
\geq
 \inf_{F}
 \mathbf{J}
-
 \delta
.
\end{equation*}
Letting
$\delta \rightarrow 0$, we obtain \eqref{eq_Prop_4-1-2}.
\\


Finally we will prove 3.
Fix $a \geq 0$. 
Let
$
 \{
  (\mu_n; \mu^{(1)}_n, \cdots, \mu^{(p)}_n)
 \}_{n=1} ^\infty
\subset
 \{\mathbf{J}\leq a\}
$
and write
$
 \psi_n ^{(i)}
=
 \sqrt
 {
  \frac{d\mu_n ^{(i)}}{dm}
 }
\in
 \mathcal{F}^{(i)}
$
and
$
 \frac{d\mu_n}{dm}
=
 \prod_{i=1} ^p
 \frac{d\mu_n ^{(i)}}{dm}
$.
We find that each 
$\{\psi_n ^{(i)}\}_n$ is bounded in $\mathcal{F}^{(i)}$.
By
\eqref{eq_compact2p}, we may assume that
$\{\psi_n ^{(i)} \}_n$ converges in $L^2$ and $L^{2p}$ 
to some 
$
 \psi^{(i)}
\in
 \mathcal{F}^{(i)}
$.
Then, since $\mathcal{F}^{(i)}\subset L^2 \cap L^{2p}$, 
we have
$
 \mu_n ^{(i)}
\rightarrow
 \mu^{(i)}
:=
 (\psi^{(i)})^2 dm
$
and
$
 \mu_n
\rightarrow
 \mu
:=
 \bigl(
  \prod_{i=1} ^p
  \frac{d\mu^{(i)}}{dm}
 \bigr)
 dm
$
with respect to the metric $d$, as $n\rightarrow\infty$.
Hence $\mathbf{J}$ is a good rate function.
\end{proof}

\vspace{1eM}

\begin{proof}
[Proof of Lemma \ref{Lem_4-2}]
We first prove 1.
Fix $\varepsilon >0$ and for each $f\in C_K(E)$,
we write $f_\varepsilon := p_\varepsilon[f]$. 
Then we have
\begin{equation*}
 \langle
  f
 ,
  t^{-1} \ell^{(i)} _{t}
 \rangle
-
 \langle
  f
 ,
  t^{-1} \ell^{(i)} _{\varepsilon, t}
 \rangle
=
 \langle
  (f - p_\varepsilon ^{(i)} [f])
 ,
  t^{-1} \ell^{(i)} _{t}
 \rangle
=
 \langle
  (f - f_\varepsilon)
 ,
  t^{-1} \ell^{(i)} _{t}
 \rangle
.
\end{equation*}
By the contraction principle 
(Theorem \ref{Thm_DZ_contraction}), 
we can find that
$
 \langle
  (f - f_\varepsilon)
 ,
  t^{-1}\ell_t ^{(i)}
 \rangle
$
satisfies 
the LDP as $t\rightarrow \infty$,
with probability $\widetilde{\mathbb{P}}_t$,
scale $t$ and the good rate function
$J_{(f - f_\varepsilon)} ^{(i)}$, where
\begin{align*}
 J_{(f - f_\varepsilon)} ^{(i)} (a)
:=&
 \inf
 \left\{
  J^{(i)}(\mu)
 ;
  \quad
  \mu\in \mathcal{M}_1(E)
 ,
  \langle (f - f_\varepsilon), \mu\rangle = a
 \right\}
\\
=&
 \inf_
 {
  \substack
  {
   \psi\in \mathcal{F}^{(i)}, \|\psi\|_2 = 1
  ;
  \\
   \langle (f - f_\varepsilon), \mu\rangle = a
  }
 }
 \left\{
  \mathcal{E}^{(i)}(\psi, \psi)
 -
  \lambda_1 ^{(i)}
 \right\}
\end{align*}
for $a\in \mathbb{R}$.
Hence, 
for $\delta >0$,
\begin{align*}
 \inf_{ |a| > \delta}
 J_{(f - f_\varepsilon)} ^{(i)}(a)
=&
 \inf_
 {
  \substack
  {
   \psi\in \mathcal{F}^{(i)}, \|\psi\|_2 = 1
  ;
  \\
   |
    \langle (f - f_\varepsilon),\psi^2\rangle 
   |
  >
   \delta
  }
 }
 \left\{
  \mathcal{E}^{(i)}(\psi, \psi)
 -
  \lambda_1 ^{(i)}
 \right\}
.
\end{align*}

Now we will show that 
\begin{equation}
\label{eq_42-3}
 \lim_{\varepsilon \rightarrow 0}
 \inf_{ |a| > \delta}
 J_{(f-f_\varepsilon)} ^{(i)}(a)
=
 \infty
.
\end{equation}
Indeed, 
for each $\psi\in \mathcal{F}^{(i)}$ with 
$
 \|\psi\|_2
=
 1
$
and
$
 |
  \langle 
  (f-f_\varepsilon) , \psi^2
  \rangle
 |
>
 \delta
$,
applying \eqref{eq_Sobolev} with $\delta=1$
and we have
\begin{equation*}
 \delta
<
 \|f - f_\varepsilon\|_q
 \|\psi^2\|_{p} 
\leq
 \|f - f_\varepsilon\|_q
 (C_1 + 1\cdot \mathcal{E}^{(i)}(\psi,\psi)) 
,
\end{equation*}
where $p^{-1} + q^{-1} =1$.
Taking infimum over such $\psi$, we have
\begin{equation*}
 \delta
\leq
 \|f- f_\varepsilon\|_q
 (
  C_1
 +
  \inf_{ |a| > \delta}
  J_{(f-f_\varepsilon)} ^{(i)}(a)
 +
  \lambda_1 ^{(i)}
 ) 
.
\end{equation*}
%
By Proposition \ref{Prop_Lp-cont-semigrp}, we have
$
 \lim_{\varepsilon \rightarrow 0}
 \|f - f_\varepsilon\|_q
=
 0
$
and therefore \eqref{eq_42-3} follows.
\\

We have by the LDP,
\begin{align*}
 \limsup_{t\rightarrow \infty}
 \frac{1}{t}
 \log
 \widetilde{\mathbb{P}}_t
 \bigl(
  |
   \langle
    (
     t^{-1}\ell_{\varepsilon, t} ^{(i)}
    -
     t^{-1}\ell_{t} ^{(i)}
    )
   ,
    f
   \rangle
  |
 >
  \delta
 \bigr)
=&
 \limsup_{t\rightarrow \infty}
 \frac{1}{t}
 \log
 \widetilde{\mathbb{P}}_t
 \bigl(
  |
   \langle
    t^{-1}\ell_{t} ^{(i)}
   ,
    (f - f_\varepsilon)
   \rangle
  |
 >
  \delta
 \bigr)
\\
\leq&
 -
 \inf_{|a|>\delta}
 J_{(f-f_\varepsilon)}^{(i)}(a)
.
\end{align*}
Letting $\varepsilon\rightarrow 0$
and the conclusion follows from \eqref{eq_42-3}.


\vspace{1eM}

We next prove 2.
Assume 
Proposition \ref{Prop_exp_app} holds.
Fix 
$\delta > 0$ and $\eta > 0$.
Take 
$C(\varepsilon)$ as in Proposition \ref{Prop_exp_app}.
By writing $k= \lceil t \rceil$, 
we have for $t,\varepsilon >0$,
\begin{align*}
&
 \mathbb{P}_{x_0}
 \left(
  |
  \langle 
   t^{-p}
   (
    \ell_t ^{\mathrm{IS}}
   -
    \ell_{\varepsilon, t} ^{\mathrm{IS}}
   )
  ,
   f
  \rangle
  |
 >
  \delta
 ;
  t < \zeta^{(1)} \wedge \cdots \wedge \zeta^{(p)}
 \right)
\\
\leq&
 \delta^{-k}
 t^{-pk}
 {\mathbb{E}}_{x_0}
 \left[
  \left|
   \langle 
    \ell_t ^{\mathrm{IS}} - \ell_{\varepsilon, t} ^{\mathrm{IS}}
    , f
   \rangle
  \right|^k
 ;
  t < \zeta^{(1)} \wedge \cdots \wedge \zeta^{(p)}
 \right]
\\
\leq&
 \delta^{-k}
 t^{-pk}
\cdot
 (k!)^p
 e^{pt}
 C(\varepsilon)^k 
\\
\leq&
 \delta^{-k}
 ( t^{-1} k )^{kp}
 e^{pt}
 C(\varepsilon)^k 
.
\end{align*}

Therefore
\begin{equation*}
 \limsup_{t\rightarrow \infty}
 \frac{1}{t}
 \log
 \mathbb{P}_{x_0}
 \left(
  |
  \langle 
   t^{-p}
   (\ell_t ^{\mathrm{IS}} - \ell_{\varepsilon, t} ^{\mathrm{IS}})
   , f
  \rangle
  |
 >
  \delta
 ;
  t < \zeta^{(1)} \wedge \cdots \wedge \zeta^{(p)}
 \right)
\leq
-
 \log \delta
+
 p
+
 \log C(\varepsilon)
\end{equation*}
and consequently, \eqref{eq_Lem_4-2-2} follows because of 
\begin{equation*}
 \limsup_{t \rightarrow \infty}
 \frac{1}{t}
 \log
 \mathbb{P}_{x_0}
 \left(
  t < \zeta^{(1)} \wedge \cdots \wedge \zeta^{(p)}
 \right)
=
 -
 \sum_{i=1} ^p
 \lambda_1 ^{(i)}
,
\end{equation*}
which is proved in \cite[Corollary 6.4.2]{MR2778606}.
\end{proof}

\begin{Prop}
[Exponential approximation]
\label{Prop_exp_app}
Suppose each $X^{(i)}$ satisfies Assumption 
$(\textbf{A};$ $\rho^{(i)}, \mu^{(i)}, t_0, C)$.
Then for each $f\in C_K(E)$ and $\varepsilon > 0$,
there exists positive constant
$C(\varepsilon)$,
which depends on 
$
 p, f, t_0, C
$,
$
 p^{(i)}(\cdot, \cdot), \rho^{(i)}, \mu^{(i)}
;
 i=1, \cdots, p
$
and is independent of $t$ and $k$, such that  
\begin{equation}
\label{eq_exp_app_modif}
 \lim_{\varepsilon\downarrow 0}
 C(\varepsilon)
=
 0
\end{equation}
and
\begin{equation}
\label{eq_prop1.6_2}
 {\mathbb{E}}
 \left[
  \left|
   \langle 
    \ell_t ^{\mathrm{IS}} - \ell_{\varepsilon, t} ^{\mathrm{IS}}
    , f
   \rangle
  \right|^k
 ;
  t < \zeta^{(1)} \wedge \cdots \wedge \zeta^{(p)}
 \right]
\leq
 e^{pt}
 (k!)^p
 C(\varepsilon)^k
\end{equation}
for all
$
 \varepsilon >0
$,
$t>0$ and $k\in \mathbb{Z}$ with $k\geq 0$.
\end{Prop}

As we have noted in Section \ref{Sec_Main_results},
this proposition says that 
the approximated intersection measure 
$\ell_{\varepsilon, t} ^{\mathrm{IS}}$ 
is a ``good'' approximation of the intersection measure
$\ell_{t} ^{\mathrm{IS}}$.


This proposition 
has the same role as \cite[Proposition 2.3]{MR2999298}.
The factor $e^{pt}$
comes from the fact that we use 1-order resolvent
while they use 0-order resolvent.
We will prove this proposition 
in Section \ref{Sec_exp_app},
which is the main part of this paper.


\section{Proof of Proposition \ref{Prop_exp_app} }
\label{Sec_exp_app}

\subsection{Outline of the proof}

Our proof of Proposition \ref{Prop_exp_app} is 
motivated by the same heuristics 
as Section 3.2 of \cite{MR2999298}
and based on 
the proof of \cite[Proposition 2.3]{MR2999298}
with some modifications.
Here is a list of differences:
\begin{itemize}
\item[1.]
We do not use 0-order resolvents (Green functions),
but use 1-order resolvents for 
deriving the desired estimates.
By this modification, we do not need to take care of the divergence
of 0-order Green functions on the diagonal set.

\item[2.]
We change the definition of $G_t$.
Compare our \eqref{eq_def_G_t} with  (3.27) in page 288 of \cite{MR2999298}.

\item[3.]
We change the way of taking partition of $\sigma_i$
in order to apply Lemma \ref{Lem_meas5}.
Compare our \eqref{eq_bij} with page 290 of \cite{MR2999298}.

\item[4.]
We do not take the summation over 
$W_i\subset \{1, \cdots, k\} $ with $\# W_i = \# F_i$,
but over $F_i ' \subset F_i$ with $\# F_i ' = \# S^*$.
Compare our \eqref{eq_F_i'} with page 290 of \cite{MR2999298}.

\item[5.]
When we take the summation over 
$
 \mathcal{N}^{(i)}
\in
 \mathcal{R}^{E_i}
$,
we do not decompose as
$E_i = (E_i \setminus F_i) \cup F_i$,
but decompose as
$E_i = (E_i \setminus J_i) \cup J_i$.
Compare 
our \eqref{eq_J_i} with the last equality of 
page 290 of \cite{MR2999298}.
\end{itemize}

The changes in 2--5 are related to the comments
before Corollary \ref{Cor_meas5}. 
Due to some (relatively minor) possible problems
pointed out below, 
it seems such modifications are needed.
We will note these differences more concretely 
in Remarks \ref{Rmk_54-1}, \ref{Rmk_54-2} and
\ref{Rmk_55-1}.

\subsection{Preliminary computations}

Before proving Proposition \ref{Prop_exp_app},
we give the following lemma and proposition.
In the following, 
we abbreviate the measure $m(dx)$ just as $dx$.
Recall
the definition of $H_t ^{(i)}$ in \eqref{Def_Ht(i)}:
\begin{equation*}
 H^{(i)} _t (y_1, \cdots, y_k) 
:=
 \int_{[0, t]^k}
 \int_E
  1_
  {
   \left\{
    \sum_{j=1}^k
    r_j
   \leq
    t
   \right\}
  }
 \prod_{j=1} ^{k+1}
 p^{(i)} _{r_j}
 (y_{j-1}, y_{j})
 dy_{k+1}
 dr_1 \cdots dr_k
,
\end{equation*}
where
$
 r_{k+1} 
= 
 t - \sum_{j=1} ^k r_j
$.

\begin{Lem}
\label{Lem_moment}
For $m\leq k$, it holds that
\begin{flalign*}
&
 \mathbb{E}_{x_0} 
 \bigl[
  \langle f, \ell_{t} ^{\mathrm{IS}} \rangle^{m}
  \langle f, \ell_{\varepsilon, t} ^{\mathrm{IS}} \rangle^{k-m}
 ;
  t < \zeta^{(1)} \wedge \cdots \wedge \zeta^{(p)}
 \bigr]
\\
&=
 \int_{E^k}
  \prod_{l=1} ^k f(y_l)
  \cdot
  \prod_{i=1} ^p
  \Biggl[
   \sum_{\sigma\in \mathfrak{S}_k}
   [
    \mathrm{id} ^{\otimes m}
   \underset{\sigma}{\otimes}
    {T_\varepsilon ^{(i)} } ^{\otimes (k-m)}
   ]
   H^{(i)} _t
   (y_1, \cdots, y_k)
  \Biggr]
 dy_1 \cdots dy_k
\\
\begin{split}
&=
 \int_{E^k}
  \prod_{l=1} ^k f(y_l)
 \cdot
  \prod_{i=1} ^p
  \Biggl[
   \sum_{\sigma\in \mathfrak{S}_k}
   \int_{[0, t]_{<} ^k}
   \Biggl[
    \int_{E^{k-m+1} }
     \prod_{j=m+1} ^k
     p_{\varepsilon} ^{(i)} ( z_j, y_j)
     \prod_{j=1} ^{k+1}
     p_{s_{j} - s_{j-1} } ^{(i)} (w_{j-1}, w_{j})
\\
&
\hspace{100mm}
    dz_{m+1} \cdots dz_{k+1}
   \Biggr]
   ds_1 \cdots ds_k
  \Biggr]
  dy_1 \cdots dy_k
,
\end{split}
&
\end{flalign*}
where
$[0, t]^k _<$ is defined in \eqref{eq_0t<}
and we set
$\sigma(0)=0$,
$\sigma(k+1)=k+1$
and
\begin{equation*}
 w_{j} 
= 
 w_{j} ^{(i)} 
=
 \displaystyle
 \begin{cases}
  y_{\sigma^{-1}(j)}
 &
  \text{when } \sigma^{-1}(j) \leq m
 ,
 \\
  z_{\sigma^{-1}(j)}
 &
  \text{when } \sigma^{-1}(j) \geq m+1
 .
 \end{cases}
\end{equation*}
\end{Lem}


The above lemma immediately implies 
the following moment formula:
which is almost the same as \cite[Lemma 3.1]{MR2999298}
\begin{Prop}
[Moment formula]
\label{Prop_moment}
For $k\geq 1$, it holds that
\begin{flalign*}
&
 \mathbb{E}_{x_0} 
 \Bigl[
  \bigl(
   \langle f, \ell_{t} ^{\mathrm{IS}} \rangle
  -
   \langle f, \ell_{\varepsilon, t} ^{\mathrm{IS}} \rangle
  \bigr)^k
 ;
  t < \zeta^{(1)} \wedge \cdots \wedge \zeta^{(p)}
 \Bigr]
\\
=&
 \sum_{m=0} ^k
 (-1)^m
 \binom{k}{m}
 \int_{E^k}
  \prod_{l=1} ^k f(y_l)
  \cdot
  \prod_{i=1} ^p
  \Biggl[
   \sum_{\sigma\in \mathfrak{S}_k}
   [
    \mathrm{id} ^{\otimes m}
   \underset{\sigma}{\otimes}
    {T_\varepsilon ^{(i)} } ^{\otimes (k-m)}
   ]
   H^{(i)} _t
   (y_1, \cdots, y_k)
  \Biggr]
 dy_1 \cdots dy_k
.
&
\end{flalign*}
\end{Prop}


\begin{proof}
[Proof of Lemma \ref{Lem_moment}]
For $\eta > 0$, we have
\begin{flalign*}
&
 \mathbb{E}_{x_0} 
 \bigl[
  \langle f, \ell_{\eta, t} ^{\mathrm{IS}} \rangle^{m}
  \langle f, \ell_{\varepsilon, t} ^{\mathrm{IS}} \rangle^{k-m}
 ;
  t < \zeta^{(1)} \wedge \cdots \wedge \zeta^{(p)}
 \bigr]
\\
=&
 \mathbb{E}_{x_0} 
 \biggl[
  \langle f, \ell_{\eta, t} ^{\mathrm{IS}} \rangle^{m}
  \langle f, \ell_{\varepsilon, t} ^{\mathrm{IS}} \rangle^{k-m}
  \prod_{i=1} ^p 
  1_E (X^{(i)} _t) 
 \biggr]
\\
\begin{split}
=&
 \mathbb{E}_{x_0} 
 \Biggl[
  \prod_{j=1} ^m 
  \int_E f(y_j) 
  \biggl(
   \prod_{i=1} ^p 
   \int_0 ^t 
    p_{\eta} ^{(i)}( X^{(i)} _{s}, y_j)
   ds
  \biggr)
  dy_j
 \cdot
  \prod_{j=m+1} ^k 
  \int_E f(y_j) 
  \biggl(
   \prod_{i=1} ^p 
   \int_0 ^t 
    p_{\varepsilon} ^{(i)} ( X^{(i)} _{s}, y_j) 
   ds
  \biggr)
  dy_j
 \cdot
 \prod_{i=1} ^p 
 1_E (X^{(i)} _t) 
 \Biggr]
\end{split}
\\
=&
 \mathbb{E}_{x_0} 
 \Biggl[
  \int_{E^k}
   \prod_{l=1} ^k f(y_l)
  \cdot
   \prod_{i=1} ^p
   \Biggl(
    \prod_{j=1} ^m
    \int_0 ^t p_{\eta}^{(i)}( X^{(i)} _{s}, y_j) ds
   \cdot
    \prod_{j=m+1} ^k
    \int_0 ^t p_{\varepsilon}^{(i)}( X^{(i)} _{s}, y_j) ds
   \cdot
    1_E(X^{(i)} _t) 
   \Biggr)
  dy_1 \cdots dy_k
 \Biggr]
\\
=&
 \int_{E^k}
  \prod_{l=1} ^k f(y_l)
 \cdot
  \prod_{i=1} ^p
  \mathbb{E}_{x_0} 
  \Biggl[
   \int_{[0,t]^k}
    \prod_{j=1} ^m
    p_{\eta}^{(i)}( X^{(i)} _{s_j}, y_j)
    \prod_{j=m+1} ^k
    p_{\varepsilon}^{(i)}( X^{(i)} _{s_j}, y_j)
   \cdot
    1_E(X^{(i)} _t) 
   ds_1 \cdots ds_k
  \Biggr]
 dy_1 \cdots dy_k
.
&
\end{flalign*}

\noindent
Now, we calculate the expectation of the above equation:
\begin{flalign*}
&
 \mathbb{E}_{x_0} 
 \Biggl[
  \int_{[0,t]^k}
   \prod_{j=1} ^m
   p_{\eta}^{(i)}( X^{(i)} _{s_j}, y_j)
   \prod_{j=m+1} ^k
   p_{\varepsilon}^{(i)}( X^{(i)} _{s_j}, y_j)
  \cdot
   1_E(X^{(i)} _t) 
  ds_1 \cdots ds_k
 \Biggr]
\\
=&
 \mathbb{E}_{x_0} 
 \Biggl[
  \sum_{\sigma\in \mathfrak{S}_k}
  \int_{0\leq s_{\sigma(1)} < \cdots < s_{\sigma(k)} \leq t}
   \prod_{j=1} ^m
   p_{\eta}^{(i)}( X^{(i)} _{s_j}, y_j)
   \prod_{j=m+1} ^k
   p_{\varepsilon}^{(i)}( X^{(i)} _{s_j}, y_j)
   \cdot
   1_E(X^{(i)} _t) 
  ds_1 \cdots ds_k
 \Biggr]
\\
=&
 \sum_{\sigma\in \mathfrak{S}_k}
 \int_{[0, t]_{<} ^k}
  \mathbb{E}_{x_0} 
  \Biggl[
   \prod_{j=1} ^m
   p_{\eta} ^{(i)}      ( X^{(i)} _{ s_{\sigma(j)} }, y_j)
   \prod_{j=m+1} ^k
   p_{\varepsilon}^{(i)}( X^{(i)} _{ s_{\sigma(j)} }, y_j)
   \cdot
   1_E(X^{(i)} _t) 
  \Biggr]
 ds_1 \cdots ds_k
\\
=&
 \sum_{\sigma\in \mathfrak{S}_k}
 \int_{[0, t]_{<} ^k}
  \Biggl[
  \int_{E^{k+1}}
   \prod_{j=1} ^m
   p_{\eta} ^{(i)}      ( x_{\sigma(j)}, y_j)
   \prod_{j=m+1} ^k
   p_{\varepsilon}^{(i)}( x_{\sigma(j)}, y_j)
   \prod_{j=1} ^{k+1}
   p_{s_j - s_{j-1} }^{(i)} (x_{j-1}, x_j)
  dx_1 \cdots dx_{k+1}
  \Biggr]
 ds_1 \cdots ds_k
\\
=&
 \sum_{\sigma\in \mathfrak{S}_k}
 [
  {T_\eta        ^{(i)} }^{\otimes m}
 \underset{\sigma}{\otimes}
  {T_\varepsilon ^{(i)}} ^{\otimes (k-m)}
 ]
 H^{(i)} _t
 (y_1, \cdots, y_k)
,
&
\end{flalign*}
where in the fourth line, 
we set
$s_0 = 0$
and
$s_{k+1} = t$.

We have $H_t ^{(i)} \in L^p(E^k)$ and
$
 \displaystyle
 \lim_{\eta\downarrow 0}
 [
  {T_\eta ^{(i)}      } ^{\otimes m}
 \underset{\sigma}{\otimes}
  {T_\varepsilon^{(i)}} ^{\otimes (k-m)}
 ]
 H^{(i)} _t
=
 [
  \mathrm{id} ^{\otimes m}
 \underset{\sigma}{\otimes}
  {T_\varepsilon^{(i)} }^{\otimes (k-m)}
 ]
 H^{(i)} _t
$
in $L^p(E^k)$.
Since Proposition \ref{Prop_existence} says that
$\langle f, \ell_{\eta, t} ^{\mathrm{IS}} \rangle$ 
converges to $\langle f, \ell_{t} ^{\mathrm{IS}} \rangle$ 
in any moments with respect to the measure 
$
 \mathbb{P}_{x_0} 
 (
  \hspace{1mm}\cdot\hspace{1mm}
 ;
  t < \zeta^{(1)} \wedge \cdots \wedge \zeta^{(p)}
 )
$,
we conclude that
\begin{flalign*}
&
 \mathbb{E}_{x_0} 
 \bigl[
  \langle f, \ell_{t} ^{\mathrm{IS}} \rangle^{m}
  \langle f, \ell_{\varepsilon, t} ^{\mathrm{IS}} \rangle^{k-m}
 ;
  t < \zeta^{(1)} \wedge \cdots \wedge \zeta^{(p)}
 \bigr]
\\
=&
 \lim_{\eta \downarrow 0}
 \mathbb{E}_{x_0} 
 \bigl[
  \langle f, \ell_{\eta, t} ^{\mathrm{IS}} \rangle^{m}
  \langle f, \ell_{\varepsilon, t} ^{\mathrm{IS}} \rangle^{k-m}
 ;
  t < \zeta^{(1)} \wedge \cdots \wedge \zeta^{(p)}
 \bigr]
\\
=&
\lim_{\eta\downarrow 0}
 \int_{E^k}
  f^{\otimes k}
  \prod_{i=1} ^p
  \Biggl[
   \sum_{\sigma\in \mathfrak{S}_k}
   [
    {T_\eta ^{(i)}      } ^{\otimes m}
   \underset{\sigma}{\otimes}
    {T_\varepsilon^{(i)}} ^{\otimes (k-m)}
   ]
   H^{(i)} _t
  \Biggr]
 dm^{\otimes k}
\\
=&
 \int_{E^k}
  f^{\otimes k}
  \prod_{i=1} ^p
  \Biggl[
   \sum_{\sigma\in \mathfrak{S}_k}
   [
    \mathrm{id} ^{\otimes m}
   \underset{\sigma}{\otimes}
    {T_\varepsilon^{(i)}} ^{\otimes (k-m)}
   ]
   H^{(i)} _t
  \Biggr]
 dm^{\otimes k}
.
\end{flalign*}
\end{proof}

\subsection{Outline of the proof}

In the following, 
we abbreviate the measure $m(dx)$ just as $dx$.
Also,
the constants may change from line to line.

We 
now begin the proof of Proposition \ref{Prop_exp_app}.
It 
is sufficient to show the following inequality, 
replacing $|\cdots|$ with $(\cdots)$ in \eqref{eq_prop1.6_2}:
for all $t>0$ and $k\in \mathbb{Z}_{>0}$, it holds that
\begin{equation}
\label{eq_sufficient}
 {\mathbb{E}_{x_0}}
 \bigl[
   \langle 
    \ell_t ^{\mathrm{IS}} - \ell_{\varepsilon, t} ^{\mathrm{IS}}
    , f
   \rangle^k
 ;
  t < \zeta^{(1)} \wedge \cdots \wedge \zeta^{(p)}
 \bigr]
\leq
 e^{pt}
 (k!)^p
 C(\varepsilon)^k
.
\end{equation}

\noindent
Indeed, 
if $k$ is even, clearly \eqref{eq_prop1.6_2} and 
\eqref{eq_sufficient} are same.
When $k=2l+1$, we have by H\"older's inequality, 
\begin{align*}
&
 \mathbb{E}_{x_0}
 \bigl[
  \bigl|
   \langle f, \ell_{t} ^{\mathrm{IS}} \rangle
  -
   \langle f, \ell_{\varepsilon, t} ^{\mathrm{IS}} \rangle
  \bigr|^{2l+1}
 ;
  t < \zeta^{(1)} \wedge \cdots \wedge \zeta^{(p)}
 \bigr]
\\
\begin{split}
\leq&
 \mathbb{E}_{x_0}
 \bigl[
  \bigl(
   \langle f, \ell_{t} ^{\mathrm{IS}} \rangle
  -
   \langle f, \ell_{\varepsilon, t} ^{\mathrm{IS}} \rangle
  \bigr)^{2l}
 ;
  t < \zeta^{(1)} \wedge \cdots \wedge \zeta^{(p)}
 \bigr]^{1/2}
\\
&
\hspace{50mm}
\cdot
 \mathbb{E}_{x_0}
 \bigl[
  \bigl(
   \langle f, \ell_{t} ^{\mathrm{IS}} \rangle
  -
   \langle f, \ell_{\varepsilon, t} ^{\mathrm{IS}} \rangle
  \bigr)^{2l+2}
 ;
  t < \zeta^{(1)} \wedge \cdots \wedge \zeta^{(p)}
 \bigr]^{1/2}
\end{split}
\\
\leq&
 e^{pt}
 \{ (2l)! ^{1/2} (2l+2)! ^{1/2} \}^{p}
 C(\varepsilon)^{2l+1} 
\\
\leq&
 e^{pt}
 2^p \{ (2l+1)! \}^p
 C(\varepsilon)^{2l+1} 
\\
\leq&
 e^{pt}
 \{ (2l+1)! \}^p
 (2^p)^{2l+1}
 C(\varepsilon)^{2l+1} 
,
\end{align*}
and hence \eqref{eq_prop1.6_2} follows.

\vspace{1eM}

Recall the definition of $H^{(i)} _t$
in \eqref{Def_Ht(i)}: 
\begin{align*}
 H_t ^{(i)}
 (x_1, \cdots, x_k)
=
 \int_{[0, t]^k}
 \int_E
  1_
  {
   \left\{
    \sum_{j=1} ^k r_j \leq t
   \right\}
  }
  \prod_{j=1} ^{k+1}
  p_{r_j} ^{(i)} (x_{j-1}, x_j)
 dx_{k+1}
 dr_1 \cdots dr_k
,
\end{align*}
where
$
 r_{k+1}
=
 t - \sum_{j=1} ^k r_j
$.

Fix $\delta>0$.
For each $D_i \subset \{1, \cdots, k\}$, set
$
 D_i ^c
:=
 \{1, \cdots, k\}\setminus D_i
$.
Also for each
$
 \mathcal{N}^{(i)}
=
 (n^{(i)} _j)_{j\in D_i ^c}
\in
 (\mathbb{Z}_{>0})^{D_i ^c}
$, define
\begin{align*}
\begin{split}
&
 H_t ^{(i)} (D_i; x_1, \cdots, x_k)
\\
:=&
 \int_{[0, t]^k}
 \int_E
  1_
  {
   \left\{
    \sum_{j=1} ^k r_j \leq t
   \right\}
  }
  \prod_{j\in D_i}
  1_
  {
   \{
    r_j\leq \delta
   \}
  }
  \prod_{j\in D_i ^c}
  1_
  {
   \{
    r_j> \delta
   \}
  }
  \prod_{j=1} ^{k+1}
  p_{r_j} ^{(i)} (x_{j-1}, x_{j})
 dx_{k+1}
 dr_1 \cdots dr_k
,
\end{split}
\\
\begin{split}
&
 H_t ^{(i)} (D_i; \mathcal{N}^{(i)}; x_1, \cdots, x_k)
\\
:=&
 \int_{[0, t]^k}
 \int_E
  1_
  {
   \left\{
    \sum_{j=1} ^k r_j \leq t
   \right\}
  } 
  \prod_{j\in D_i}
  1_
  {
   \{
    r_j\leq \delta
   \}
  }
  \prod_{j\in D_i ^c}
  1_
  {
   \{
    r_j> \delta
   \}
  }
\\
&
\hspace{5mm}
  \prod_{j\in D_i}
  p_{r_j} ^{(i)} (x_{j-1}, x_j)
  \prod_{j\in D_i ^c}
  \exp
  \left\{
   -r_j \lambda_{n_j ^{(i)}} ^{(i)}
  \right\}
  \psi^{(i)} _{n^{(i)} _j}(x_{j-1})
  \psi^{(i)} _{n^{(i)} _j}(x_{j  })
  p_{r_{k+1}} ^{(i)} (x_{k}, x_{k+1})
 dx_{k+1}
 dr_1 \cdots dr_k
.
\end{split}
\end{align*}

\noindent
We remark that
our 
$H_t ^{(i)} (D_i; \mathcal{N}^{(i)}; x_1, \cdots, x_k)$
is obtained 
by 
multiplying $p_{r_j} ^{(i)}$'s and $\psi_{{n_j} ^{(i)}}^{(i)}$'s
by 
$H_r(\mathcal{N}^{(i)}; D_i)$ defined in (3.18) of \cite{MR2999298},
and
by integrating 
$
\int_{[0, t]^k}
\int_E
 1_
 {
  \left\{
   \sum_{j=1} ^k r_j \leq t
  \right\}
 } 
dx_{k+1}
dr_1 \cdots dr_k
$.
 
We furthermore fix $R\in\mathbb{Z}_{>0}$
and set 
$
 \mathcal{R}
:=
 \{1, \cdots, R\}
$
and
$
 \mathcal{R}^c 
:=
 \mathbb{Z}_{>0}\setminus \mathcal{R}
$.
By 
the eigenfunction expansion
\eqref{eq_eigenfct} in Lemma \ref{Lem_pre}, 
we have
\begin{flalign*}
 \prod_{j\in D_i ^c}
 p_{r_j} ^{(i)} (x_{j-1}, x_j)
=&
 \prod_{j\in D_i ^c}
 \sum_{n_j ^{(i)} =1} ^\infty 
 \exp
 {
  \left\{
   -r_j \lambda_{n_j ^{(i)}} ^{(i)}
  \right\}
 }
 \psi_{n_j ^{(i)}} ^{(i)} (x_{j-1})
 \psi_{n_j ^{(i)}} ^{(i)} (x_{j  })
\\
=&
 \prod_{j\in D_i ^c}
 \Bigl(
  \sum_{n_j ^{(i)} \in \mathcal{R}} 
 +
  \sum_{n_j ^{(i)} \in \mathcal{R}^c} 
 \Bigr)
 \exp
 {
  \left\{
   -r_j \lambda_{n_j ^{(i)}} ^{(i)} 
  \right\}
 }
 \psi_{n_j ^{(i)}} ^{(i)} (x_{j-1})
 \psi_{n_j ^{(i)}} ^{(i)} (x_{j  })
\\
=&
 \sum_{ E_i \subset D_i ^c }
 \sum_
 {
  \mathcal{N}^{(i)} |_{D_i ^c\setminus E_i}
 \in
  (\mathcal{R}^c)^{D_i ^c\setminus E_i}
 }
 \sum_
 {
  \mathcal{N}^{(i)} |_{E_i}
 \in
  \mathcal{R}^{E_i}
 }
 \prod_{j\in D_i ^c}
 \exp
 {
  \left\{
   -r_j \lambda_{n_j ^{(i)}} ^{(i)} 
  \right\}
 }
 \psi_{n_j ^{(i)}} ^{(i)} (x_{j-1})
 \psi_{n_j ^{(i)}} ^{(i)} (x_{j  })
.
\end{flalign*}

\noindent
Hence we can see that
\begin{align}
\notag
&
 H_t ^{(i)} (x_1, \cdots, x_k)
\\
\label{eq_H1}
=&
 \sum_{D_i \subset \{1, \cdots, k\}}
 H_t ^{(i)} (D_i; x_1, \cdots, x_k)
\\
\label{eq_H2}
=&
 \sum_{D_i \subset \{1, \cdots, k\}}
 \sum_{E_i \subset D_i ^c}
 \sum_
 {
  \mathcal{N}^{(i)} |_{D_i ^c\setminus E_i}
 \in
  (\mathcal{R}^c)^{D_i ^c\setminus E_i}
 }
 \sum_
 {
  \mathcal{N}^{(i)} |_{E_i}
 \in
  \mathcal{R}^{E_i}
 }
 H_t ^{(i)} (D_i, \mathcal{N}^{(i)}; x_1, \cdots, x_k)
.
\end{align}

Fix small $\eta, \gamma >0$ such that 
$1 - 2p(\eta + \gamma) >0$.
We can easily see that
\begin{align*}
 \sum_
 {
  \substack
  {
   D_i \subset \{1, \cdots, k\}
  \\
   i=1, \cdots, p
  }
 }
=&
 \sum_
 {
  \substack
  {
   D_i \subset \{1, \cdots, k\}
  ,
  \\ 
   \# D_i \leq \eta k
   \text{ for all } i
  }
 }
+
 \sum_
 {
  \substack
  {
   D_i \subset \{1, \cdots, k\}
  ,
  \\ 
   \# D_i > \eta k
   \text{ for some } i 
  }
 }
\\
 \sum_
 {\substack
  {
   E_i \subset D_i ^c,\\
   i=1, \cdots, p
  }
 }
=&
 \sum_
 {\substack
  {
   E_i \subset D_i ^c,\\
   \# (D_i ^c \setminus E_i ) \leq \gamma k 
   \text{ for all } i
  }
 }
+
 \sum_
 {\substack
  {
   E_i \subset D_i ^c,\\
   \# (D_i ^c \setminus E_i ) >    \gamma k 
   \text{ for some } i
  }
 }
.
\end{align*}

\noindent
Now 
we split the left hand side of \eqref{eq_sufficient}
in the manner of the above two equalities,
by 
using Proposition \ref{Prop_moment} and 
\eqref{eq_H1}--\eqref{eq_H2}:
\begin{flalign*}
&
 \mathbb{E}_{x_0} 
 \bigl[
  \bigl(
   \langle f, \ell_{t} ^{\mathrm{IS}} \rangle
  -
   \langle f, \ell_{\varepsilon, t} ^{\mathrm{IS}} \rangle
  \bigr)^k
 ;
  t < \zeta^{(1)} \wedge \cdots \wedge \zeta^{(p)}
 \bigr]
\\
=&
 (\textbf{I})_{t, k} (\delta, \varepsilon)
+
 (\textbf{II})_{t, k} (\delta, \varepsilon)
\\
=&
 (\textbf{Ia})_{t, k} (\delta, \varepsilon, R)
+
 (\textbf{Ib})_{t, k} (\delta, \varepsilon, R)
+
 (\textbf{II})_{t, k} (\delta, \varepsilon)
,
\end{flalign*}

\noindent
where
\begin{flalign}
\notag
\begin{split}
&
 (\textbf{I})_{t,k} (\delta, \varepsilon)
\\
:=&
 \sumDiall
 \sum_{m=0} ^k
 (-1)^m
 \binom{k}{m}  
 \sum_
 {
  \substack
  {
   \sigma_i \in \mathfrak{S}_k
  ,
  \\
   i=1, \cdots p
  }
 }
 \int_{E^k}
  f^{\otimes k}
  \cdot
  \prod_{i=1} ^p
  \biggl\{
   [
    \mathrm{id} ^{\otimes m}
   \underset{\sigma}{\otimes}
    {T_\varepsilon ^{(i)}} ^{\otimes (k-m)}
   ]
   H_t ^{(i)} (D_i; \cdot)
  \biggr\}
 dm^{\otimes k}
,
\end{split}
\\
\begin{split}
&
 (\textbf{Ia})_{t,k} 
 (\delta, \varepsilon, R)
\\
:=&
 \sumDiall
 \sumEiall
 \sumNiDE
 \sumNiE
 \sum_{m=0} ^k
 (-1)^m
 \binom{k}{m}  
\\
\label{eq_Ia}
&
\hspace{35mm}
 \sum_
 {
  \substack
  {
   \sigma_i \in \mathfrak{S}_k
  ;\\
   i=1, \cdots p
  }
 }
 \int_{E^k}
  f^{\otimes k}
 \cdot
  \prod_{i=1} ^p
  \left\{
   \left[
    \mathrm{id} ^{\otimes m}
   \underset{\sigma_{i}}{\otimes}
    {T_\varepsilon^{(i)}} ^{\otimes (k-m)}
   \right]
   H_t ^{(i)} (D_i, \mathcal{N}^{(i)} ; \cdot)
  \right\}
 dm^{\otimes k}
\end{split}
&
\end{flalign}

\noindent
and
\begin{align*}
 (\textbf{II})_{t, k} (\delta, \varepsilon)
:=&
 \mathbb{E}_{x_0} 
 \bigl[
  \bigl(
   \langle f, \ell_{t} ^{\mathrm{IS}} \rangle
  -
   \langle f, \ell_{\varepsilon, t} ^{\mathrm{IS}} \rangle
  \bigr)^k
 ;
  t < \zeta^{(1)} \wedge \cdots \wedge \zeta^{(p)}
 \bigr]
- 
 (\textbf{I})_{t, k} (\delta, \varepsilon)
,
\\
 (\textbf{Ib})_{t, k} (\delta, \varepsilon, R)
:=&
 (\textbf{I})_{t, k} (\delta, \varepsilon)
-
 (\textbf{Ia})_{t, k} (\delta, \varepsilon, R)
.
\end{align*}
We use the same symbol 
$(\mathbf{I})_{t, k}(\delta, \varepsilon, R)$
and so on,
as in Section 3.3 of \cite{MR2999298}.

\vspace{1eM}

We divide 
the proof of Proposition \ref{Prop_exp_app} 
into three parts:

\begin{Lem}
\label{Lem_A}
For sufficiently small $\delta >0$,
There exists $C(\delta)>0$, independent of 
$t$, $k$ and $\varepsilon$, such that 
$\lim_{\delta \rightarrow 0} C(\delta) = 0$
and
\begin{equation*}
 (\textbf{II})_{t, k} (\delta, \varepsilon)
\leq
 (k!)^p
 e^{tp}
 C(\delta)^k
.
\end{equation*}
\end{Lem}

\begin{Lem}
\label{Lem_B}
Fix small $\delta >0$ as in Lemma \ref{Lem_A}.
Then, for sufficiently large $R\in \mathbb{Z}_{\geq 0}$,
there exists $C_\delta(R)>0$, 
independent of $t$, $k$ and $\varepsilon$, 
such that 
$\lim_{R \rightarrow \infty} C_\delta(R) = 0$
and
\begin{equation*}
 (\textbf{Ib})_{t, k} (\delta, \varepsilon, R)
\leq
 (k!)^p
 e^{tp}
 C_\delta(R)^k
.
\end{equation*}
\end{Lem}

\begin{Lem}
\label{Lem_C}
Fix small $\delta >0$ and large $R\in \mathbb{Z}_{\geq 0}$ 
as in Lemma \ref{Lem_B}.
Then,
for sufficiently small $\varepsilon >0$,
there exists $C_{\delta, R}(\varepsilon)>0$, 
independent of $t$ and $k$, such that 
$\lim_{\varepsilon \rightarrow 0} C_{\delta, R}(\varepsilon) = 0$
and
\begin{equation*}
 (\textbf{Ia})_{t, k} (\delta, \varepsilon, R)
\leq
 (k!)^p
 e^{tp}
 C_{\delta, R}(\varepsilon)^k
.
\end{equation*}
\end{Lem}


Once 
Lemma \ref{Lem_A}--\ref{Lem_C} are proved, 
we can obtain \eqref{eq_sufficient} as follows.
For each positive integer $n$, 
choose $\delta_n$, $R_n$ and $\varepsilon_n$ such that
$
 C(\delta_n) + C_\delta(R_n) + C_{\delta_n, R_n}(\varepsilon)
\leq
 n^{-1}
$
for all
$
 \varepsilon \leq \varepsilon_n 
$.
We may assume that 
the sequence $\{\varepsilon_n\}$ is decreasing 
and converges to zero as $n\rightarrow \infty$.
Set
$
 C'(\varepsilon)
:=
 3
 \{
  C(\delta_n) + C_\delta(R_n) + C_{\delta_n, R_n}(\varepsilon)
 \}
$
for
$
 \varepsilon_{n+1}
 <
 \varepsilon
 \leq
 \varepsilon_n
$.
For $\varepsilon > \varepsilon_1$, we also set
$
 C'(\varepsilon)
:=
  2
  \|f\|_\infty
  \sum_{i=1} ^p
  \bigl[
   \sup_{y\in E}
   \int_E
    \Gr{i}(x, y)^p
   dx
  \bigr]
$
because of the estimate
\begin{align*}
&
 \mathbb{E}_{x_0} 
 \bigl[
  \bigl(
   \langle f, \ell_{t} ^{\mathrm{IS}} \rangle
  -
   \langle f, \ell_{\varepsilon, t} ^{\mathrm{IS}} \rangle
  \bigr)^k
 ;
  t < \zeta^{(1)} \wedge \cdots \wedge \zeta^{(p)}
 \bigl]
\\
=&
 \sum_{m=0} ^k
 (-1)^m
 \binom{k}{m}
 \int_{E^k}
  f^{\otimes k}
  \cdot
  \prod_{i=1} ^p
  \Biggl[
   \sum_{\sigma\in \mathfrak{S}_k}
   [
    \mathrm{id} ^{\otimes m}
   \underset{\sigma}{\otimes}
    {T_\varepsilon ^{(i)}} ^{\otimes (k-m)}
   ]
   H_t ^{(i)}
  \Biggr]
 dm^{\otimes k}
\\
\leq&
 (k!)^p
 e^{pt}
 \biggl\{
  2
  \|f\|_\infty
  \sum_{i=1} ^p
  \biggl[
   \sup_{y\in E}
   \int_E
    \Gr{i}(x, y)^p
   dx
  \biggr]
 \biggr\}^k
,
&
\end{align*}
which is obtained by the same computation 
as \eqref{eq_uni.int}.
Then we obtain
$
 \lim_{\varepsilon \rightarrow 0}
 C'(\varepsilon)
=
 0
$
and \eqref{eq_sufficient} 
with the constant $C'(\varepsilon)$. 


\vspace{1eM}

First, we prove Lemma \ref{Lem_A} and \ref{Lem_B}.

\begin{proof}[Proof of Lemma \ref{Lem_A}]
Write
\begin{align*}
 C
=
 \sum_{i=1} ^p
 \biggl[
  \sup_{y\in E}
  \int_E
   \Gr{i} (x,y)^p
  dx
 \biggr]
+
 1
,
\quad
 C(\delta)
=
 \sum_{i=1} ^p
 \biggl[
  \sup_{y\in E}
  \int_E
   \Grdel{i} (x,y)^p
  dx
 \biggr]
\end{align*}
and choose small $\delta$ such that $C(\delta)<1$.
For small $\delta > 0$, 
by applying the inequalities
\begin{equation*}
 \prod_{j\in D_i}
 1_
 {
  \{
   r_j\leq \delta
  \}
 }
 \prod_{j\in D_i ^c}
 1_
 {
  \{
   r_j> \delta
  \}
 }
\leq
 \prod_{j\in D_i}
 1_
 {
  \{
   r_j\leq \delta
  \}
 }
\leq
 1
\end{equation*}
to $H_t ^{(i)}$, we have

\noindent
\begin{align*}
 \|
  H_t ^{(i)} (D_i; \cdot)
 \|_{L^p(E^k)}
\leq&
 e^t
 \biggl[
  \sup_{y\in E}
  \int_E
   \Gr{i} (x, y)^p
  dx
 \biggr]^{\#D_i ^c /p}
 \biggl[
  \sup_{y\in E}
  \int_E
   \Grdel{i} (x,y)^p
  dx
 \biggl]^{\#D_i    /p}
\\
\leq&
 e^t
 C^{k/p}
 (1\wedge C(\delta))^{\#D_i    /p}
.
\end{align*}

Then we have for small $\delta > 0$,
\begin{align*}
&
 (\textbf{II})_{t,k} (\eta, \delta, \varepsilon)
\\
=&
 \sumDisome
 \sum_{m=0} ^k
 (-1)^m
 \binom{k}{m}  
 \int_{E^k}
  f^{\otimes k}
  \cdot
  \prod_{i=1} ^p
  \Biggl[
   \sum_{\sigma\in \mathfrak{S}_k}
   [
    \mathrm{id} ^{\otimes m}
   \underset{\sigma}{\otimes}
    {T_\varepsilon ^{(i)}} ^{\otimes (k-m)}
   ]
   H_t ^{(i)} (D_i; \cdot)
  \Biggr]
 dm^{\otimes k}
\\
\leq&
 \|
  f
 \|_{\infty} ^k
 2^k
 (k!)^p
 \sumDisome
 \prod_{i=1} ^p
 \|
  H_t ^{(i)} (D_i; \cdot)
 \|_{L^p(E^k)}
\\
\leq&
 \|
  f
 \|_{\infty} ^k
 2^k
 (k!)^p
 e^{pt}
 \sumDisome
 C^{k}
 C(\delta)^{\eta k}
\\
\leq&
 \|
  f
 \|_{\infty} ^k
 2^k
 (k!)^p
 e^{pt}
 2^{kp}
 C^{k}
 C(\delta)^{\eta k}
\\
=&
 (k!)^p
 e^{pt}
 C(\delta)^{k}
,
\end{align*}
where in the third line, 
we use the equality $\sum_{m=0} ^k \binom{k}{m} = 2^k$ 
and the same estimate as we compute \eqref{eq_uni.int}.
In 
the last line, we take another constant 
$
 C(\delta)
=
 C(p, p^{(i)}, f, \eta, \delta)
$
satisfying the same properties.
\end{proof}


\begin{proof}[Proof of Lemma \ref{Lem_B}]
By \eqref{eq_growth_of_eigen} in Lemma \ref{Lem_pre}, 
there exist 
$C_1, C_2 > 0$ and $R_0 \in \mathbb{Z}_{>0}$,
such that
$
 \|\psi_{n} ^{(i)} \|_\infty ^2
\leq 
 C_1 n
$
and
$
 C_2 n^{1/\rho^{(i)}}
\leq
 \lambda^{(i)} _{n}
$
for all $n > R_0$.
Then for $R \geq R_0$, we have
\begin{equation*}
 \sum_{n=1} ^R
 \|\psi_{n} ^{(i)} \|_\infty ^2
\leq
 \sum_{n=1} ^{R_0}
 \|\psi_{n} ^{(i)} \|_\infty ^2
+
 C_1
 R
\end{equation*}
and have
\begin{align*}
&
 \sum_{n=R+1} ^\infty
 (1 + \lambda^{(i)} _{n})^{-1}
 \exp
 \left\{
  -\delta 
  (1 + \lambda^{(i)} _{n})
 \right\}
 \|\psi_{n} ^{(i)} \|_\infty ^2
\leq
 C_1
 \sum_{n=R+1} ^\infty
 (1 + C_2 n^{1/\rho^{(i)}})^{-1}
 \exp
 \left\{
  -\delta 
  (1 + C_2 n^{1/\rho^{(i)}})
 \right\}
 n
.
\end{align*}

Hence, by setting 
\begin{align*}
 C(R)
=&
 \sum_{i=1} ^p 
 \sum_{n=1} ^R 
 \|\psi_{n} ^{(i)} \|_\infty ^2
+
 1
, \quad
 C_\delta (R)
=
 \sum_{i=1} ^p
 \sum_{n=R+1} ^\infty
 (1 + \lambda^{(i)} _{n})^{-1}
 \exp
 \left\{
  -\delta 
  (1 + \lambda^{(i)} _{n})
 \right\}
 \|\psi_{n} ^{(i)} \|_\infty ^2
,
\end{align*}
we have
\begin{equation}
\label{eq_Lem_B_zero}
 \lim_{R \rightarrow \infty}
 C_\delta(R)
=
 0
,
\quad 
 \lim_{R \rightarrow \infty}
 C_\delta(R)^\gamma
 C(R)^p
=
 0
.
\end{equation}

Since
\begin{align*}
&
 H_t ^{(i)}
 (
  D_i, \mathcal{N}^{(i)}
 ;
  x_1, \cdots, x_k
 )
\\
\begin{split}
\leq&
 \int_{[0, t]^k}
  1_
  {
   \left\{
    \sum_{j=1} ^k r_j \leq t
   \right\}
  }
\\
&
\hspace{5mm}
  \prod_{j\in D_i}
  p^{(i)} _{r_j}(x_{j-1}, x_j)
 \cdot
  \prod_{j\in D_i ^c\setminus E_i}
  1_{r_j > \delta}
  \exp
  \biggl\{
   -r_j \lambda_{n_j ^{(i)}} ^{(i)}
  \biggr\}
  \|\psi_{n_j ^{(i)}} ^{(i)} \|_\infty ^2
 \cdot
  \prod_{j\in E_i}
  \|\psi_{n_j ^{(i)}} ^{(i)} \|_\infty ^2
 \hspace{2mm}
 dr_1\cdots dr_k
,
\end{split}
\end{align*}
by writing $E':= \text{supp}[f]$, 
we have
\begin{align*}
&
 \|
  H_t ^{(i)}
  (
   D_i, \mathcal{N}^{(i)}
  ;
   \cdot
  )
 \|_{L^p({E'}^k)}
\\
\begin{split}
\leq&
 e^{t} 
 \biggl[
  \sup_{y\in E}
  \int_E
   \Gr{i} (x,y)^p
  dx
 \biggr]^{\# D_i /p}
\\
&
\hspace{5mm}
 \prod_{j\in D_i ^c \setminus E_i}
 m(E')^{1/p}
 (1 + \lambda^{(i)} _{n^{(i)} _j})^{-1}
 \exp
 \biggl\{
  -\delta 
  (1 + \lambda^{(i)} _{n^{(i)} _j})
 \biggr\}
 \|\psi_{n_j ^{(i)}} ^{(i)} \|_\infty ^2
\cdot
 \prod_{j\in E_i}
 m(E')^{1/p}
 \|\psi_{n_j ^{(i)}} ^{(i)} \|_\infty ^2
.
\end{split}
\end{align*}

Hence
\begin{align*}
&
 \sum_
 {
  \mathcal{N}^{(i)} |_{D_i ^c\setminus E_i}
 \in
  \left(\mathcal{R}^c\right)^{D_i ^c \setminus E_i} 
 }
 \sum_
 {
  \mathcal{N}^{(i)} |_{E_i}
 \in  
  \mathcal{R}^{E_i}
 }
 \|
  H_t ^{(i)}
  (
   D_i, \mathcal{N}^{(i)}
  ;
   \cdot
  )
 \|_{L^p({E'}^k)}
\\
\begin{split}
\leq&
 e^{t} 
 \biggl[
  \sup_{y\in E}
  \int_E
   \Gr{i} (x,y)^p
  dx
 \biggr]^{\# D_i /p}
 m(E')^{\# D_i ^c /p}
\\
&
\hspace{5mm}
 \biggl(
  \sum_{n=R+1} ^\infty
  (1 + \lambda^{(i)} _{n})^{-1}
  \exp
  \Bigl\{
   -\delta 
   (1 + \lambda^{(i)} _{n})
  \Bigr\}
  \|\psi_{n} ^{(i)} \|_\infty ^2
 \biggr)^{\# D_i ^c \setminus E_i}
\cdot
 \biggl(
  \sum_{n=1} ^R
  \|\psi_{n} ^{(i)} \|_\infty ^2
 \biggr)^{\# E_i}
\end{split}
\\
\leq&
 e^t
 C^k
 C_{\delta}(R) ^{\# D_i ^c \setminus E_i}
 C(R)^k
.
\end{align*}

We may assume that $C_\delta (R) <1$.
Then
$
 (\textbf{Ib})_{t,k} (\eta, \gamma, \delta, \varepsilon, R)
\leq
 e^{pt}
 (k!)^p
 C^k
 C_{\delta}(R) ^{\gamma k}
 C(R)^{p k}
$.
By setting
$
 C'_{\delta}(R)
=
 C
 C_\delta(R)^{\gamma}
 C(R)^p
$, 
we have
$\lim_{R\rightarrow\infty}C'_\delta(R) = 0$ 
by \eqref{eq_Lem_B_zero},
and therefore we have the desired result.
\end{proof}

\subsection{Proof of Lemma \ref{Lem_C}  (computation of \textbf{(Ia)}, part 1)}

From 
this subsection to the end of Section \ref{Sec_exp_app}, 
we prove Lemma \ref{Lem_C}.

In 
this subsection, we rewrite each term of \textbf{(Ia)} 
into the products of ``good'' and  ``not good'' factors.
Here, the notation ``good'' means that 
the factor does not depend on $t$
and is able to estimate with respect to $\varepsilon$,
uniformly in $t$.


Let 
$D_i \subset \{1, \cdots, k\}$
and
$E_i \subset D_i ^c := \{1, \cdots, k\}\setminus D_i$
satisfying
\begin{equation*}
 \# D_i \leq \eta k
 \text{ and }
 \# (D_i ^c\setminus E_i) \leq \gamma k
\quad
 \text{for all }
 i=1, \cdots, p
.
\end{equation*}
Write
$
 E_i - 1
:=
 \{j-1 : j\in E_i\} \cap \{1, \cdots, k\}
$
and
$F_i := E_i \cap (E_i -1)$.
Note that obviously $k\not\in F_i$.
Then 
we can decompose as
\begin{align*}
 \prod_{j\in D_i ^c}
 \psi_{n_j ^{(i)}} ^{(i)} (x_{j-1})
 \psi_{n_j ^{(i)}} ^{(i)} (x_{j  })
=&
 \prod_{j\in D_i ^c -1}
 \psi_{n_{j+1} ^{(i)}} ^{(i)} (x_{j})
 \prod_{j\in D_i ^c}
 \psi_{n_{j  } ^{(i)}} ^{(i)} (x_{j})
\\
=&
 \prod_{j\in (D_i ^c -1)\setminus F_i}
 \psi_{n_{j+1} ^{(i)}} ^{(i)} (x_{j})
 \prod_{j\in D_i ^c \setminus F_i}
 \psi_{n_{j  } ^{(i)}} ^{(i)} (x_{j})
 \prod_{j\in F_i}
 \biggl[
  \psi_{n_{j+1} ^{(i)}} ^{(i)}
  \psi_{n_{j  } ^{(i)}} ^{(i)}
 \biggr]
 (x_{j})
\end{align*}

\noindent
and then
\begin{align*}
&
 \prod_{j\in D_i}
 p_{r_j}(x_{j-1}, x_j)
 \prod_{j\in D_i ^c}
 \exp
 {
  \biggl\{
   -r_j \lambda_{n_j ^{(i)}} ^{(i)} 
  \biggr\}
 }
 \psi_{n_j ^{(i)}} ^{(i)} (x_{j-1})
 \psi_{n_j ^{(i)}} ^{(i)} (x_{j  })
\cdot
 p^{(i)} _{r_{k+1}}(x_k, x_{k+1})
\\
\begin{split}
=&
 \prod_{j\in F_i}
 \biggl[
  \psi_{n_{j  } ^{(i)}} ^{(i)}
  \psi_{n_{j+1} ^{(i)}} ^{(i)}
 \biggr]
 (x_{j})
\\
&
\hspace{5mm}
 \prod_{j\in D_i}
 p_{r_j}(x_{j-1}, x_j)
 \prod_{j\in D_i ^c}
 \exp
 {
  \biggl\{
   -r_j \lambda_{n_j ^{(i)}} ^{(i)} 
  \biggr\}
 }
 \prod_{j\in (D_i ^c -1)\setminus F_i}
 \psi_{n_{j+1} ^{(i)}} ^{(i)} (x_{j})
 \prod_{j\in D_i ^c \setminus F_i}
 \psi_{n_{j  } ^{(i)}} ^{(i)} (x_{j})
\cdot
 p^{(i)} _{r_{k+1}}(x_k, x_{k+1})
.
\end{split}
\end{align*}

\noindent
Since $k\not\in F_i$, we can decompose
\begin{align}
\label{eq_def_G_t}
&
 H^{(i)} _t 
 (D_i, \mathcal{N}^{(i)} ;x_1, \cdots, x_k)
=
 \prod_{j\in F_i}
 \biggl[
  \psi_{n_{j  } ^{(i)}} ^{(i)} 
  \psi_{n_{j+1} ^{(i)}} ^{(i)}
 \biggr]
 (x_{j})
\cdot
 G^{(i)} _t
 (
  D_i, \mathcal{N}^{(i)};
  \{x_j\}_{j\not\in F_i}
 )
,
\end{align}
where we define
a function on $E^{\# F_i ^c}$ by
\begin{align*}
\begin{split}
&
 G^{(i)} _t
 (
  D_i, \mathcal{N}^{(i)};
  \{x_j\}_{j\not\in F_i}
 )
\\
:=&
 \int_{[0, t]^k}
 \int_E
  1_
  {
   \left\{
    \sum_{j=1} ^k r_j \leq t
   \right\}
  }
 \prod_{j\in D_i}
 p_{r_j}(x_{j-1}, x_j)
 \prod_{j\in D_i ^c}
 \exp
 {
  \left\{
   -r_j \lambda_{n_j ^{(i)}} ^{(i)} 
  \right\}
 }
\\
&
\hspace{10mm}
 \prod_{j\in (D_i ^c -1)\setminus F_i}
 \psi_{n_{j+1} ^{(i)}} ^{(i)} (x_{j})
 \prod_{j\in D_i ^c \setminus F_i}
 \psi_{n_{j  } ^{(i)}} ^{(i)} (x_{j})
\cdot
 p^{(i)} _{r_{k+1}}(x_k, x_{k+1})
 dx_{k+1}
 dr_1 \cdots dr_k
.
\end{split}
\end{align*}

\noindent
Hence 
for some operators $U_1 ^{(i)}, \cdots, U_k ^{(i)}$,
we have
\begin{align*}
&
 U^{(i)}_1\otimes\cdots\otimes U^{(i)}_k
 \Biggl[
  \prod_{j\in F_i}
  \biggl[
   \psi_{n_{j  } ^{(i)}} ^{(i)} 
   \psi_{n_{j+1} ^{(i)}} ^{(i)}
  \biggr]
  (x_{j})
 \cdot
  G^{(i)} _t
  (
   D_i, \mathcal{N}^{(i)};
   \{x_j\}_{j\not\in F_i}
  )
 \Biggr]
\\
=&
 \prod_{j\in F_i}
 U^{(i)}_j
 \biggl[
  \psi_{n_{j  } ^{(i)}} ^{(i)} 
  \psi_{n_{j+1} ^{(i)}} ^{(i)}
 \biggr]
 (x_{j})
\cdot
 \bigotimes_{j\not\in F_i}
 U^{(i)}_j
 G_t ^{(i)}
 (
  D_i, \mathcal{N}^{(i)};
  \{x_j\}_{j\not\in F_i}
 )
\end{align*}

\noindent
and have
\begin{align*}
&
 U^{(i)}_1
\otimes
 \underset{\sigma_i}{\cdots}
\otimes
 U^{(i)}_k
 \Biggl[
  \prod_{j\in F_i}
  \biggl[
   \psi_{n_{j  } ^{(i)}} ^{(i)} 
   \psi_{n_{j+1} ^{(i)}} ^{(i)}
  \biggr]
  (x_{j})
 \cdot
  G^{(i)} _t
  (
   D_i, \mathcal{N}^{(i)};
   \{x_j\}_{j\not\in F_i}
  )
 \Biggr]
\\
\begin{split}
=&
 \prod_{l\in S(\sigma)}
 U^{(i)}_{l}
 \biggl[
  \psi_{n_{\sigma_i(l)    } ^{(i)}} ^{(i)} 
  \psi_{n_{\sigma_i(l) + 1} ^{(i)}} ^{(i)}
 \biggr]
 (x_{l})
\\
&
\hspace{10mm}
 \bigotimes_{l\not\in S(\sigma)}
 U^{(i)}_{l}
 \Biggl[
  \prod_{l\in \sigma_i ^{-1}(F_i)\setminus S(\sigma)}
  \biggl[
   \psi_{n_{\sigma_i(l)    } ^{(i)}} ^{(i)} 
   \psi_{n_{\sigma_i(l) + 1} ^{(i)}} ^{(i)}
  \biggr]
  (x_{l})
 \cdot
  G_t ^{(i)}
  (
   D_i, \mathcal{N}^{(i)};
   \{x_{l}\}_{l\not\in\sigma_i ^{-1}(F_i)}
  )
 \Biggr]
,
\end{split}
\end{align*}
where we write 
$
 S(\sigma)
:=
 \bigcap_{i=1} ^p
 \sigma_i ^{-1} (F_i)
$.

In particular, by applying  
$
 U_1 ^{(i)} 
\otimes 
 \cdots 
\otimes
 U_k ^{(i)}
=
 \mathrm{id} ^{\otimes m}
\otimes
 {T_\varepsilon^{(i)}} ^{\otimes (k-m)}
$,
we have
\begin{align*}
&
 U^{(i)}_1
\otimes
 \underset{\sigma_i}{\cdots}
\otimes
 U^{(i)}_k
 \Bigl[
  H_t ^{(i)}
  (
   D_i, \mathcal{N}^{(i)};
   x_1, \cdots, x_k
  )
 \Bigr]
\\
\begin{split}
=&
 \prod_
 {
  l\in S(\sigma)_{\leq m}
 }
 \biggl[
  \psi_{n_{\sigma_i(l)    } ^{(i)}} ^{(i)} 
  \psi_{n_{\sigma_i(l) + 1} ^{(i)}} ^{(i)}
 \biggr]
 (x_{l})
\cdot
 \prod_
 {
  l\in S(\sigma)_{> m}
 }
 T^{(i)}_{\varepsilon}
 \biggl[
  \psi_{n_{\sigma_i(l)    } ^{(i)}} ^{(i)} 
  \psi_{n_{\sigma_i(l) + 1} ^{(i)}} ^{(i)}
 \biggr]
 (x_{l})
\\
&
\hspace{10mm}
 \bigotimes_{l\not\in S(\sigma)}
 U^{(i)}_{l}
 \Biggl[
  \prod_{l\in \sigma_i ^{-1}(F_i)\setminus S(\sigma)}
  \biggl[
   \psi_{n_{\sigma_i(l)    } ^{(i)}} ^{(i)} 
   \psi_{n_{\sigma_i(l) + 1} ^{(i)}} ^{(i)}
  \biggr]
  (x_{l})
 \cdot
  G_t ^{(i)}
  (
   D_i, \mathcal{N}^{(i)};
   \{x_{l}\}_{l\not\in\sigma_i ^{-1}(F_i)}
  )
 \Biggr]
,
\end{split}
\end{align*}

\noindent
where
$
 S(\sigma)_{\leq m}
:=
 S(\sigma)
\cap
 \{1, \cdots, m\}
$
and
$
 S(\sigma)_{> m}
:=
 S(\sigma)
\cap
 \{m+1, \cdots, k\}
$.
We set
\begin{align*}
 a
 (
  \mathcal{N}_{\sigma(j) }
 , 
  \mathcal{N}_{\sigma(j)+1}
 )
:=&
 \biggl\langle
  f
 ,
  \prod_{i=1} ^p
  \biggl[
   \psi_{n_{\sigma_i(l)    } ^{(i)}} ^{(i)} 
   \psi_{n_{\sigma_i(l) + 1} ^{(i)}} ^{(i)}
  \biggr]
 \biggr\rangle
,
\\
 a_{\varepsilon}
 (
  \mathcal{N}_{\sigma(j) }
 , 
  \mathcal{N}_{\sigma(j)+1}
 )
:=&
 \biggl\langle
  f
 ,
  \prod_{i=1} ^p
  T^{(i)}_{\varepsilon}
  \biggl[
   \psi_{n_{\sigma_i(l)    } ^{(i)}} ^{(i)} 
   \psi_{n_{\sigma_i(l) + 1} ^{(i)}} ^{(i)}
  \biggr]
 \biggr\rangle
\end{align*}

\noindent
and set
\begin{align}
\label{eq_Def_Gt}
\begin{split}
&
 G_t(m, D, E, \sigma, \mathcal{N})
\\
:=
&
 \int_{E^{S(\sigma)^c} }
  \prod_{l\not\in S(\sigma)}
  f(y_l)
 \cdot
  \prod_{i=1} ^p
  \bigotimes_{l\not\in S(\sigma)}
  U^{(i)}_{l}
  \Biggl[
   \prod_{l\in \sigma_i ^{-1}(F_i)\setminus S(\sigma)}
   \biggl[
    \psi_{n_{\sigma_i(l)    } ^{(i)}} ^{(i)} 
    \psi_{n_{\sigma_i(l) + 1} ^{(i)}} ^{(i)}
   \biggr]
   (y_{l})
  \cdot
   G_t ^{(i)}
   (
    D_i, \mathcal{N}^{(i)};
    \{y_{l}\}_{l\not\in\sigma_i ^{-1}(F_i)}
   )
  \Biggr]
 dy
,
\end{split}
\end{align}
where
$
 U_l ^{(i)} = \mathrm{id} 
$
for $l \leq m$ and
$
 U_l ^{(i)} = T_\varepsilon ^{(i)}
$
for $l \geq m+1$. 
In the definition of $G_t$,
we shorten the symbol 
$\{dy_i\}_{j\in S(\sigma)^c}$ as $dy$.
Then we have
\begin{align}
\notag
&
 \int_{E^k}
  f^{\otimes k}
 \cdot
  \prod_{i=1} ^p
  \biggl\{
   \biggl[
    \mathrm{id} ^{\otimes m}
   \underset{\sigma_{i}}{\otimes}
    {T_\varepsilon^{(i)}} ^{\otimes (k-m)}
   \biggr]
   H_t ^{(i)} (D_i, \mathcal{N}^{(i)} ; \cdot)
  \biggr\}
 dm^{\otimes k}
\\
\label{eq_step1}
=&
 \prod_
 {
  j\in S(\sigma)_{\leq m}
 }
 a
 (
  \mathcal{N}_{\sigma(j) }
 , 
  \mathcal{N}_{\sigma(j)+1}
 )
 \prod_
 {
  j\in S(\sigma)_{> m}
 }
 a_{\varepsilon}
 (
  \mathcal{N}_{\sigma(j) }
 , 
  \mathcal{N}_{\sigma(j)+1}
 )
\cdot
 G_t(m, D, E, \sigma, \mathcal{N})
.
\end{align}

\begin{Rmk}
\label{Rmk_54-1}
There is a difference between 
our $G_t$ and the $G_t$ 
defined in (3.27) of \cite{MR2999298}.
Our 
$G_t$ is obtained by replacing 
$\sigma_i ^{-1}(F_i)$ by $S(\sigma)$
in the definition of their $G_t$,
yet
we have the same type of decomposition \eqref{eq_step1}
as in the last three lines of the equation (3.26) in
\cite{MR2999298}.
\end{Rmk}

\subsection{Proof of Lemma \ref{Lem_C}  (computation of \textbf{(Ia)}, part 2)}

Following 
the last subsection, we also deform 
$(\textbf{Ia})$ in this subsection.
We regard $G_t$ as a function of 
$m_i$, $\sigma_i$ and $\mathcal{N}^{(i)}$,
and decompose $(\textbf{Ia})$ into 
twelve sums.
In 
the end of this subsection, we rearrange the order
of twelve sums appropriately,
and in the following subsection, 
we will estimate $(\textbf{Ia})$ 
according to the rearranged order.

Fix 
$D_i\subset \{1, \cdots, k\}$, 
$E_i\subset D_i ^c$ and 
$\mathcal{N}^{(i)}\in (\mathbb{Z}_{>0})^{D_i ^c}$
with
$
 \mathcal{N}^{(i)}|_{D_i ^c \setminus E_i}
\in 
 (\mathcal{R}^c)^{D_i ^c \setminus E_i}
$,
$
 \mathcal{N}^{(i)}|_{E_i}
\in 
 (\mathcal{R})^{E_i}
$.
Decompose as
\begin{align*}
&
 \sum_{\sigma_1,\cdots, \sigma_p\in \mathfrak{S}_k}
 \sum_{m=0} ^k
 (-1)^m
 \binom{k}{m}
\\
=&
 \sum_{\sigma_1,\cdots, \sigma_p\in \mathfrak{S}_k}
 \sum_{m=0} ^k
 \sum_
 {
  \substack
  {
   \hspace{-3mm}
   0 \leq m_1 \leq m
  ,
  \\
   0 \leq m_3 \leq k-m
  }
 }
 \sum_
 {
  \substack
  {
   A\subset \{1, \cdots, m\}
  ,
  \\
   \# A= m_1
  }
 }
 \sum_
 {
  \substack
  {
   B\subset \{m+1, \cdots, k\}
  ,
  \\
   \# B= m_3
  }
 }
 (-1)^m
 \binom{k}{m}
 1_
  {
   \left\{
   \substack
    {
     A=S(\sigma)_{\leq m}
    ,
    \\
     B=S(\sigma)_{> m}
    }
   \right\}
  }
\\
=&
 \sum_{m=0} ^k
 \sum_
 {
  \substack
  {
   \hspace{-3mm}
   0 \leq m_1 \leq m
  ,
  \\
   0 \leq m_3 \leq k-m
  }
 }
 \sum_
 {
  \substack
  {
   A\subset \{1, \cdots, m\}
  ,
  \\
   \# A= m_1
  }
 }
 \sum_
 {
  \substack
  {
   B\subset \{m+1, \cdots, k\}
  ,
  \\
   \# B= m_3
  }
 }
 (-1)^m
 \binom{k}{m}
 \sum_{\sigma_1,\cdots, \sigma_p\in \mathfrak{S}_k}
 1_
  {
   \left\{
   \substack
    {
     A=S(\sigma)_{\leq m}
    ,
    \\
     B=S(\sigma)_{> m}
    }
   \right\}
  }
.
\end{align*}

We will see that:
\begin{Lem}
\label{Lem_55}
Fix 
two nonnegative integers 
$m_1 \leq m$ and $m_3 \leq k-m$.
Then
\begin{align}
\label{eq_Lem_55}
 \sum_{\sigma\in \mathfrak{S}_k ^p}
 1_
  {
   \left\{
   \substack
    {
     A=S(\sigma)_{\leq m}
    ,
    \\
     B=S(\sigma)_{> m}
    }
   \right\}
  }
 \prod_{j\in S(\sigma)_{\leq m}}
 a
 (
  \mathcal{N}_{\sigma_{(j)  }}
 ,
  \mathcal{N}_{\sigma_{(j)+1}}
 )
 \prod_{j\in S(\sigma)_{> m}}
 a_\varepsilon
 (
  \mathcal{N}_{\sigma_{(j)  }}
 ,
  \mathcal{N}_{\sigma_{(j)+1}}
 )
 \cdot
 G_t(m, D, E, \sigma, \mathcal{N})
\end{align}
depends only on $m_1$ and $m_3$,
and does not depend on the choice of 
$A, B \subset \{1, \cdots, k\}$
with $\#A = m_1$ and $\#B = m_3$.

Namely, for all $\tau \in \mathfrak{S}_k$ with 
$
 \tau(\{1,\cdots,m\})
=
 \{1,\cdots,m\}
$
and
$
 \tau(\{m+1,\cdots,k\})
=
 \{m+1,\cdots,k\}
$,
the summation
\begin{align*}
 \sum_{\sigma\in \mathfrak{S}_k ^p}
 1_
  {
   \left\{
   \substack
    {
     A=\tau^{-1}( S(\sigma)_{\leq m})
    ,
    \\
     B=\tau^{-1}( S(\sigma)_{> m})
    }
   \right\}
  }
 \prod_{j\in S(\sigma)_{\leq m}}
 a
 (
  \mathcal{N}_{\sigma_{(j)  }}
 ,
  \mathcal{N}_{\sigma_{(j)+1}}
 )
 \prod_{j\in S(\sigma)_{> m}}
 a_\varepsilon
 (
  \mathcal{N}_{\sigma_{(j)  }}
 ,
  \mathcal{N}_{\sigma_{(j)+1}}
 )
 \cdot
 G_t(m, D, E, \sigma, \mathcal{N})
\end{align*}
is equal to \eqref{eq_Lem_55}.
\end{Lem}

\begin{proof}
First, note that
$
 \tau S(\sigma)
:=
 \bigcap_{i=1} ^p
 \tau \sigma_i ^{-1} (F_i)
=
 S(\sigma\circ\tau^{-1})
$,
where
$
 \sigma\circ\tau^{-1}
:=
 (
  \sigma_1\circ \tau^{-1}, \cdots, \sigma_p\circ \tau^{-1}
 )
\in
 {\mathfrak{S}_k }^p
$.
Then we have
$
 \tau (S(\sigma)_{\leq m})
=
 (\tau S(\sigma) )_{\leq m}
=
 S(\sigma\circ\tau^{-1})_{\leq m}
$
and similarly, we also have
$
 \tau (S(\sigma)_{> m})
=
 S(\sigma\circ\tau^{-1})_{> m}
$.
Hence, we have
\begin{align*}
 \prod_{j\in S(\sigma\circ\tau^{-1})_{\leq m}}
 a
 (
  \mathcal{N}_{\sigma\circ\tau^{-1}(j)}
 ,
  \mathcal{N}_{\sigma\circ\tau^{-1}(j)+1}
 )
=&
 \prod_{j\in S(\sigma\circ\tau^{-1})_{\leq m}}
 \biggl\langle
  f
  ,
  \prod_{i=1} ^p
  \biggl[
   \psi_{n_{\sigma_{i}\circ\tau^{-1} (j)+1} ^{(i)}} ^{(i)} 
   \psi_{n_{\sigma_{i}\circ\tau^{-1} (j)  } ^{(i)}} ^{(i)}
  \biggr]
 \biggr\rangle
\\
=&
 \prod_{\tau^{-1}(j)\in S(\sigma)_{\leq m}}
 \biggl\langle
  f
  ,
  \prod_{i=1} ^p
  \biggl[
   \psi_{n_{\sigma_{i}\circ\tau^{-1} (j)+1} ^{(i)}} ^{(i)} 
   \psi_{n_{\sigma_{i}\circ\tau^{-1} (j)  } ^{(i)}} ^{(i)}
  \biggr]
 \biggr\rangle
\\
=&
 \prod_{j\in S(\sigma)_{\leq m}}
 a
 (
  \mathcal{N}_{\sigma(j)}
 ,
  \mathcal{N}_{\sigma(j)+1}
 )
\end{align*}

\noindent
and similarly,
\begin{align*}
 \prod_{j\in S(\sigma\circ\tau^{-1})_{> m}}
 a_\varepsilon
 (
  \mathcal{N}_{\sigma\circ\tau^{-1}(j)}
 ,
  \mathcal{N}_{\sigma\circ\tau^{-1}(j)+1}
 )
=&
 \prod_{j\in S(\sigma)_{> m}}
 a_\varepsilon
 (
  \mathcal{N}_{\sigma(j)}
 ,
  \mathcal{N}_{\sigma(j)+1}
 )
.
\end{align*}

Moreover, we have
\begin{align*} 
&
 G_t (m, D, E, \sigma\circ\tau^{-1}, \mathcal{N})
\\
\begin{split}
=&
 \int_{E^{S(\sigma\circ\tau^{-1})^c} }
  \prod_{l\not\in S(\sigma\circ\tau^{-1})}
  f(y_l)
\\
&
\hspace{5mm}
  \prod_{i=1} ^p
  \bigotimes_{l\not\in S(\sigma\circ\tau^{-1})}
  T^{(i)}_{l}
  \left[
   \prod_
   {
    l
   \in
    \tau\circ\sigma_i ^{-1}(F_i)
   \setminus 
    S(\sigma\circ\tau^{-1})
   }
   \biggl[
    \psi_{n_{\sigma_i\circ\tau^{-1}(l)    } ^{(i)}} ^{(i)} 
    \psi_{n_{\sigma_i\circ\tau^{-1}(l) + 1} ^{(i)}} ^{(i)}
   \biggr]
   (y_{l})
  \cdot
   G_t ^{(i)}
   (
    D_i, \mathcal{N}^{(i)};
    \{y_{l}\}_{l\not\in\tau\circ\sigma_i ^{-1}(F_i)}
   )
  \right]
 dy
\end{split}
\\
\begin{split}
=&
 \int_{E^{\tau(S(\sigma)^c)} }
  \prod_{l\not\in \tau(S(\sigma))}
  f(y_l)
\\
&
\hspace{5mm}
  \prod_{i=1} ^p
  \bigotimes_{l\not\in \tau(S(\sigma))}
  T^{(i)}_{l}
  \left[
   \prod_
   {
    l
   \in 
    \tau
    (
    \sigma_i ^{-1}(F_i)
   \setminus 
    S(\sigma)
    )
   }
   \biggl[
    \psi_{n_{\sigma_i\circ\tau^{-1}(l)    } ^{(i)}} ^{(i)} 
    \psi_{n_{\sigma_i\circ\tau^{-1}(l) + 1} ^{(i)}} ^{(i)}
   \biggr]
   (y_{l})
  \cdot
   G_t ^{(i)}
   (
    D_i, \mathcal{N}^{(i)};
    \{y_{l}\}_{l\not\in\tau\circ\sigma_i ^{-1}(F_i)}
   )
  \right]
 dy
\end{split}
\\
\begin{split}
=&
 \int_{E^{\tau(S(\sigma)^c)} }
  \prod_{j\not\in S(\sigma)}
  f(y_{\tau(j)})
\\
&
\hspace{5mm}
  \prod_{i=1} ^p
  \bigotimes_{j\not\in S(\sigma)}
  T^{(i)}_{\tau(j)}
  \left[
   \prod_
   {
    j
   \in 
    \sigma_i ^{-1}(F_i)
   \setminus 
    S(\sigma)
   }
   \biggl[
    \psi_{n_{\sigma_i(j)    } ^{(i)}} ^{(i)} 
    \psi_{n_{\sigma_i(j) + 1} ^{(i)}} ^{(i)}
   \biggr]
   (y_{\tau(j)})
  \cdot
   G_t ^{(i)}
   (
    D_i, \mathcal{N}^{(i)};
    \{y_{\tau(j)}\}_{j\not\in\sigma_i ^{-1}(F_i)}
   )
  \right]
 dy
\end{split}
\\
=&
 G_t (m, D, E, \sigma, \mathcal{N})
,
\end{align*}
where the symbol $dy$ is defined 
as the same as in \eqref{eq_Def_Gt}.
Hence the desired equality holds.
\end{proof}


\vspace{1eM}

By Lemma \ref{Lem_55},
we may rewrite \eqref{eq_Lem_55} 
by choosing
$A = S^* _{\leq m}$ and $B = S^* _{>m}$,
where we define
$
 S^* _{\leq m}
:=
 \{ 1,\cdots, m_1 \}
$
and
$
 S^* _{> m}
:=
 \{ m + 1,\cdots, m + m_3 \}
$.
Note that
$
 \#
 S(\sigma)_{\leq m}
=
 \#
 S^* _{\leq m}
=
 m_1
$.

Combining \eqref{eq_step1} 
with Lemma \ref{Lem_55}, we have
\begin{align}
&
\notag 
 \sum_{m=0} ^k
 (-1)^m
 \binom{k}{m}
 \sum_
 {
  \substack
  {
   \sigma_i \in \mathfrak{S}_k
  ;\\
   i=1, \cdots p
  }
 }
 \int_{E^k}
  f^{\otimes k}
 \cdot
  \prod_{i=1} ^p
  \biggl\{
   \biggl[
    \mathrm{id} ^{\otimes m}
   \underset{\sigma_{i}}{\otimes}
    {T_\varepsilon^{(i)}} ^{\otimes (k-m)}
   \biggr]
   H_t ^{(i)} (D_i, \mathcal{N}^{(i)} ; \cdot)
  \biggr\}
 dm^{\otimes k}
\\
\notag 
\begin{split}
=&
 \sum_{m=0} ^k
 \sum_
 {
  \substack
  {
   \hspace{-3mm}
   0 \leq m_1 \leq m
  ,
  \\
   0 \leq m_3 \leq k-m
  }
 }
 \sum_
 {
  \substack
  {
   A\subset \{1, \cdots, m\}
  ,
  \\
   \# A= m_1
  }
 }
 \sum_
 {
  \substack
  {
   B\subset \{m+1, \cdots, k\}
  ,
  \\
   \# B= m_3
  }
 }
 (-1)^{m-m_1}
 \binom{k}{m}
 \sum_{\sigma\in \mathfrak{S}_k ^p}
 1_
  {
   \left\{
   \substack
    {
     A=S(\sigma)_{\leq m}
    ,
    \\
     B=S(\sigma)_{> m}
    }
   \right\}
  }
\\
&
\hspace{5mm}
 \prod_{j\in S(\sigma)_{\leq m}}
 \hspace{-3mm}
 \Bigl(
 -a
 (
  \mathcal{N}_{\sigma_{(j)  }}
 ,
  \mathcal{N}_{\sigma_{(j)+1}}
 )
 \Bigr)
 \prod_{j\in S(\sigma)_{> m}}
 \hspace{-3mm}
 a_\varepsilon
 (
  \mathcal{N}_{\sigma_{(j)  }}
 ,
  \mathcal{N}_{\sigma_{(j)+1}}
 )
 \cdot
 G_t(m, D, E, \sigma, \mathcal{N})
\end{split}
\\
\notag 
\begin{split}
=&
 \sum_{m=0} ^k
 \sum_
 {
  \substack
  {
   \hspace{-3mm}
   0 \leq m_1 \leq m
  ,
  \\
   0 \leq m_3 \leq k-m
  }
 }
 \binom{m  }{m_1}
 \binom{k-m}{m_3}
 (-1)^{m-m_1}
 \binom{k}{m}
 \sum_{\sigma\in \mathfrak{S}_k ^p}
 1_
  {
   \left\{
   \substack
    {
     S_{\leq m} ^*=S(\sigma)_{\leq m}
    ,
    \\
     S_{> m} ^*=S(\sigma)_{\leq m}
    }
   \right\}
  }
\\
&
\hspace{5mm}
 \prod_{j\in S(\sigma)_{\leq m}}
 \hspace{-3mm}
 \Bigl(
 -a
 (
  \mathcal{N}_{\sigma_{(j)  }}
 ,
  \mathcal{N}_{\sigma_{(j)+1}}
 )
 \Bigr)
 \prod_{j\in S(\sigma)_{> m}}
 \hspace{-3mm}
 a_\varepsilon
 (
  \mathcal{N}_{\sigma_{(j)  }}
 ,
  \mathcal{N}_{\sigma_{(j)+1}}
 )
 \cdot
 G_t(m, D, E, \sigma, \mathcal{N})
\end{split}
\\
\label{eq_Ia-3}
\begin{split}
=&
 \sum_
 {
  \substack
  {
   0\leq m_i \leq k
  ,
  \\
   m_1 + m_2 + m_3 + m_4 =k
  }
 }
 (-1)^{m_2}
 \frac{k!}{m_1 ! m_2 ! m_3 ! m_4 !}
 \sum_{\sigma\in \mathfrak{S}_k ^p}
 1_
  {
   \left\{
   \substack
    {
     S_{\leq m} ^*=S(\sigma)_{\leq m}
    ,
    \\
     S_{> m} ^*=S(\sigma)_{\leq m}
    }
   \right\}
  }
\\
&
\hspace{5mm}
 \prod_{j\in S(\sigma)_{\leq m}}
 \hspace{-3mm}
 \Bigl(
 -a
 (
  \mathcal{N}_{\sigma_{(j)  }}
 ,
  \mathcal{N}_{\sigma_{(j)+1}}
 )
 \Bigr)
 \prod_{j\in S(\sigma)_{> m}}
 \hspace{-3mm}
 a_\varepsilon
 (
  \mathcal{N}_{\sigma_{(j)  }}
 ,
  \mathcal{N}_{\sigma_{(j)+1}}
 )
 \cdot
 G_t(m, D, E, \sigma, \mathcal{N})
.
\end{split}
\end{align}


\vspace{1eM}

From now on, 
we fix $m_1$, $m_2$, $m_3$ and $m_4$
with $m_1 + m_2 = m$ and $m_1 + m_2 + m_3 + m_4 = k$.
For 
each subset $F_i '\subset F_i$ with 
$\#F_i ' = \#S^* = m_1 + m_3$,
we split $\sigma_i \in \mathfrak{S}_k$ as
$
 \sigma_i
=
 (
  {\sigma_i |}_{S^*}
 ,
  {\sigma_i |}_{{S^*} ^c}
 )
$
and 
regard them as a tuple of bijections
\begin{equation}
\label{eq_bij}
 {\sigma_i |}_{S^*}
\in
 \mathrm{Bij}(S^*, F_i ')
,\quad
 {\sigma_i |}_{{S^*} ^c}
\in
 \mathrm{Bij}({S^*} ^c, {F_i '} ^c)
.
\end{equation}
Then 
we have
\begin{equation}
\label{eq_F_i'}
 \sum_{\sigma\in \mathfrak{S}_k ^p}
 1_
 {
  \{
   \sigma 
  \in
   \mathfrak{S}_k ^p
  |
   S(\sigma) = S^*
  \}
 }
=
 \sum_
 {
  \substack
  {
   F_i ' \subset F_i
  ,
  \\
   \# F_i ' = \# S^*
   ;\text{ for all }i
  }
 }
 \sumsigmaS
 \sumsigmaSc
.
\end{equation}

Set
$
 F_i ' + 1
:=
 \{j + 1 : j\in F_i '\}
\cap
 \{1, \cdots, k\}
$
and
$
 J_i
:=
 F_i '
\cap
 (F_i ' +1)
\hspace{1mm}
 (
  \subset F_i
  \subset E_i
 )
$.
We decompose as
\begin{equation}
\label{eq_J_i}
 \sumNiE
=
 \sumNiEJ
 \sumNiJ
.
\end{equation}

\begin{Rmk}
\label{Rmk_54-2}
Note that
there are differences here from the proof of 
\cite[Proposition 2.3]{MR2999298}.
Firstly, 
we decompose $\sigma_i$ according to $S^*$,
so we define a new symbol $F'_i$. 
Secondly,
we decompose $\mathcal{N}^{(i)}$ 
according to another new symbol $J_i$, 
in order to make sure that
$G_t$ is independent of the values of
$\mathcal{N}^{(i)}|_{J_i}\in \mathcal{R}^{J_i}$
(see Proposition \ref{Prop_55a-1}).
\end{Rmk}

We conclude
\begin{align}
\notag
&
 (\textbf{Ia})_{t, k}
 (\delta, \varepsilon, R)
\\
\notag
=&
 \sumDiall
 \sumEiall
 \sumNiDE
 \sumNiE
 (
  \text{\ref{eq_Ia-3}}
 )
\\
\label{eq_Ia-4}
\begin{split}
=&
 \sumDiall
 \sumEiall
 \sum_
 {
  \substack
  {
   0\leq m_i \leq k
  ,
  \\
   m_1 + m_2 + m_3 + m_4 =k
  }
 }
 \sumFipr
 (-1)^{m_2}
 \frac{k!}{m_1 ! m_2 ! m_3 ! m_4 !}
\\
&
\hspace{10mm}
 \sumNiDE
 \sumNiEJ
 \sumNiJ
 \sumsigmaS
 \sumsigmaSc
\\
&
\hspace{35mm}
 \prod_{j\in S^* _{\leq m}}
 \hspace{-2mm}
 \Bigl(
 -a
 (
  \mathcal{N}_{\sigma_{(j)  }}
 ,
  \mathcal{N}_{\sigma_{(j)+1}}
 )
 \Bigr)
 \prod_{j\in S^* _{> m}}
 \hspace{-3mm}
 a_\varepsilon
 (
  \mathcal{N}_{\sigma_{(j)  }}
 ,
  \mathcal{N}_{\sigma_{(j)+1}}
 )
 \cdot
 G_t(m, D, E, \sigma, \mathcal{N})
.
\end{split}
\end{align}

Now, define
\begin{align*}
 M_{m_1, m_3}
:=&
 \left\{
  (A, r)
 ;\hspace{1mm}
  \begin{array}{l}
   A \text{ and } r
   \text{ are $\mathbb{Z}_{\geq 0}$-valued measures on $(\mathcal{R}^p)^2$ with}
  \\
   \text{the total mass $m_1+m_3$ and $m_1$, respectively, and}
  \\
   r(l)\leq A(l) \text{ for all }l\in (\mathcal{R}^p)^2
  .
  \end{array}
 \right\}
,
\\
\begin{split}
 \Psi(A, r)
:=&
 \left\{
  (
   \mathcal{N}^{(i)}|_{J_i}
  , 
   \sigma_i|_{S^*}
  )_{1\leq i \leq p}
 \in
  \bigotimes_{i=1}^p
  (\mathcal{R}^{J_i} 
 \times
  \mathrm{Bij}(S^*, F_i'))
 ;
 \right.
\\
&
\hspace{50mm}
 \left.
  A
 =
  \sum_{j\in S^*}
  \delta_
  {
   (
    \mathcal{N}_{\sigma(j)}, \mathcal{N}_{\sigma(j)+1}
   )
  }
 ,
  r
 =
  \sum_{j\in S_{\leq m} ^*}
  \delta_
  {
   (
    \mathcal{N}_{\sigma(j)}, \mathcal{N}_{\sigma(j)+1}
   )
  }
 \right\}
.
\end{split}
\end{align*}

For a given 
$
 (
  \mathcal{N}^{(i)}|_{J_i}
 , 
  \sigma_i|_{S^*}
 )_{1\leq i \leq p}
\in
 \bigotimes_{i=1}^p
 (\mathcal{R}^{J_i} \times \mathrm{Bij}(S^*, F_i'))
$,
there is only one $(A, r)\in M_{m_1, m_3}$ such that
$
 (
  \mathcal{N}^{(i)}|_{J_i}
 , 
  \sigma_i|_{S^*}
 )_{1\leq i \leq p}
\in
 \Psi(A, r)
$.
Then we have
\begin{align*}
&
 \prod_{j\in S^*_{\leq m}}
 \hspace{-2mm}
 \Bigl(
 -a
 (
  \mathcal{N}_{\sigma_{(j)  }}
 ,
  \mathcal{N}_{\sigma_{(j)+1}}
 )
 \Bigr)
 \prod_{j\in S^*_{> m}}
 \hspace{-3mm}
 a_\varepsilon
 (
  \mathcal{N}_{\sigma_{(j)  }}
 ,
  \mathcal{N}_{\sigma_{(j)+1}}
 )
\\
=&
 \prod_{l\in (\mathcal{R}^2)^p}
 \biggl[
  (-a(l))^
  {
   \#
   \left\{
    j\in S^*_{\leq m}
   ;
    (
     \mathcal{N}_{\sigma_{(j)  }}
    ,
     \mathcal{N}_{\sigma_{(j)+1}}
    )
   =
    l
   \right\}
  }
  a_\varepsilon (l)^
  {
   \#
   \left\{
    j\in S^*_{> m}
   ;
    (
     \mathcal{N}_{\sigma_{(j)  }}
    ,
     \mathcal{N}_{\sigma_{(j)+1}}
    )
   =
    l
   \right\}
  }
 \biggr]
\\
=&
 \sum_{(A, r)\in M_{m_1, m_3}}
 1_
  {
   \left\{
    (\mathcal{N}, \sigma)
   \in
    \Psi(A, r)
   \right\}
  }
 \prod_{l\in (\mathcal{R}^2)^p}
 \bigl[
  (-a(l))^{r(l)}
  a_\varepsilon (l)^{A(l) - r(l)}
 \bigr]
.
\end{align*}

\noindent
Hence 
we furthermore can rewrite the last two lines of 
\eqref{eq_Ia-4} as follows:
\begin{align}
\notag
\begin{split}
&
 \sumNiDE
 \sumNiEJ
 \sumNiJ
 \sumsigmaS
 \sumsigmaSc
\\
&
\hspace{15mm}
 \prod_{j\in S^* _{\leq m}}
 \hspace{-3mm}
 \Bigl(
 -a
 (
  \mathcal{N}_{\sigma_{(j)  }}
 ,
  \mathcal{N}_{\sigma_{(j)+1}}
 )
 \Bigr)
 \prod_{j\in S^* _{> m}}
 \hspace{-3mm}
 a_\varepsilon
 (
  \mathcal{N}_{\sigma_{(j)  }}
 ,
  \mathcal{N}_{\sigma_{(j)+1}}
 )
 \cdot
 G_t(m, D, E, \sigma, \mathcal{N})
\end{split}
\\
\label{eq_Ia-7}
\begin{split}
=&
 \sumNiDE
 \sumNiEJ
 \sumNiJ
 \sumsigmaS
 \sumsigmaSc
\\
&
\hspace{15mm}
 \sum_{(A, r)\in M_{m_1, m_3}}
 1_
  {
   \left\{
    (\mathcal{N}, \sigma)
   \in 
     \Psi(A, r)
   \right\}
  }
 \prod_{l\in (\mathcal{R}^2)^p}
 \left[
  (-a(l))^{r(l)}
  a_\varepsilon (l)^{A(l) - r(l)}
 \right]
\cdot
 G_t(m, D, E, \sigma, \mathcal{N})
.
\end{split}
\end{align}

For two measures $A$ and $r$, 
we write $r\leq A$ 
if
$r(l)\leq A(l)$ for all $l\in (\mathcal{R}^2)^p$.
Then we have
$
 \sum_{(A, r)\in M_{m_1, m_3}}
=
 \sum_{A\in M_{k - m_2 - m_4}}
 \sum_
 {
  r\in M_{m_1}
 ;
  r \leq A 
 }
$,
where we set 
\begin{equation*}
 M_{m_1}
:=
 \left\{
  \text
  {%
   $\mathbb{Z}_{\geq 0}$-valued measures 
   on $(\mathcal{R}^p)^2$
   with total mass $m_1$%
  }
 \right\}
.
\end{equation*}

Summarizing our argument up to here, 
we have the following decomposition 
of $(\textbf{Ia})$ in \eqref{eq_Ia}: 
\begin{align}
&
\notag
 (\textbf{Ia})_{t, k}
 (\delta, \varepsilon, R)
\\
\notag
=&
 \eqref{eq_Ia-4}
\\
\label{eq_Ia-5}
\begin{split}
=&
 \sumDiall
 \sumEiall
 \sum_
 {
  0\leq m_2 + m_4 \leq k
 }
 \sum_{m_1 = 0} ^{k - m_2 - m_4 }
 \sumFipr
\\
&
\hspace{10mm}
 \sumNiDE
 \sumNiEJ
 \sumNiJ
 \sumsigmaS
 \sumsigmaSc
 \sum_{A\in M_{k - m_2 - m_4}}
 \sum_
 {
  \substack
  {
   r\in M_{m_1}
  ;\\
   r \leq A 
  }
 }
\\
&
\hspace{20mm}
 (-1)^{m_2}
 \frac{k!}{m_1 ! m_2 ! (k - m_1 - m_2 - m_4)! m_4 !}
 1_
  {
   \left\{
    (\mathcal{N}, \sigma)
   \in 
    \Psi(A, r)
   \right\}
  }
\\
&
\hspace{30mm}
 \prod_{l\in (\mathcal{R}^2)^p}
 \left[
  (-a(l))^{r(l)}
  a_\varepsilon (l)^{A(l) - r(l)}
 \right]
\cdot
 G_t(m_1 + m_2, D, E, \sigma, \mathcal{N})
.
\end{split}
\end{align}

We rearrange the order of twelve summations 
in \eqref{eq_Ia-5} as follows:
\begin{equation}
\label{eq_sums}
 \sum_{4}
 \sum_{3}
 \sum_{2}
 \sum_{1}
,
\end{equation}
where each sum $\sum_1, \cdots, \sum_4$ means that

\begin{align*}
&
 \sum_1
=
 \sumNiJ
 \sumsigmaS
\\
&
 \sum_2
=
 \sumFipr
 \sumNiDE
 \sumNiEJ
 \sumsigmaSc
\\
&
 \sum_3
=
 \sum_{m_1 = 0} ^{k - m_2 - m_4 }
 \sum_
 {
  \substack
  {
   r\in M_{m_1}
  ;\\
   r \leq A 
  }
 }
\\
&
 \sum_{4-1}
=
 \sum_
 {
  0\leq m_2 + m_4 \leq k
 }
 \sum_{A\in M_{k - m_2 - m_4}}
\\
&
 \sum_{4-2}
=
 \sumDiall
 \sumEiall
\end{align*}
and 
$
 \displaystyle
 \sum_{4} 
=
 \sum_{4-2} 
 \sum_{4-1}
$.
We 
will keep the order of the summations as above.
In 
the following subsections, we estimate these sums
one by one.

\subsection{Proof of Lemma \ref{Lem_C} (estimate of the summations $\sum_1$-$\sum_4$)}
\label{Sec_55a}


We 
first consider the summation $\sum_1$.
Fix 
$D_i$, 
$E_i$, 
$m_1, m_2, m_4$,
$A$,
$r$
and
$F_i'$.
We also fix 
$
 \mathcal{N}^{(i)}
\in
 \left(\mathcal{R}^c\right)^{D_i ^c \setminus E_i} 
$,
$
 \mathcal{N}^{(i)}
\in
 \left(\mathcal{R}^c\right)^{E_i \setminus J_i} 
$
and
$
 \sigma_i|_{{S^*}^c} 
\in
 \mathrm{Bij}({S^*}^c, {F_i'}^c)
$.
Recall the definitions
$
 S^*
:=
 \{1, \cdots, m_1\}
\cup
 \{m_1 + m_2 + 1, \cdots, m_1 + m_2+ m_3\}
$,
$
 S(\sigma)
:=
 \bigcap_{i=1}^p
 \sigma_i ^{-1}(F_i)
(
 =
  \sigma_i ^{-1}(F_i ')
 =
  S^*
 ,
  \text{for all }i
  \text{ and }\sigma_i|_{S^*} 
)
$
and
$
 J_i
:=
 F_i ' \cap (F_i ' +1)
 (\subset F_i \subset E_i)
$.

We also recall the definition of $G_t$ in \eqref{eq_Def_Gt}; 
for
$
 \mathcal{N}^{(i)} 
\in
 \mathcal{R}^{J_i}
$
and
$
 {\sigma_i |}_{{S^*}}
\in
 \mathrm{Bij}({S^*}, {F_i'})
$,
\begin{align*}
&
 G_t(m_1 + m_2, D, E, \sigma, \mathcal{N})
\\
=&
 \int_{E^{{S^*}^c} }
  \prod_{l\not\in S^*}
  f(y_l)
 \cdot
  \prod_{i=1} ^p
  \bigotimes_{l\not\in S^*}
  U^{(i)}_{l}
  \left[
   \prod_{l\in \sigma_i ^{-1}(F_i)\setminus S^*}
   \biggl[
    \psi_{n_{\sigma_i(l)    } ^{(i)}} ^{(i)} 
    \psi_{n_{\sigma_i(l) + 1} ^{(i)}} ^{(i)}
   \biggr]
   (y_{l})
  \cdot
   G_t ^{(i)}
   (
    D_i, \mathcal{N}^{(i)};
    \{y_{l}\}_{l\not\in\sigma_i ^{-1}(F_i)}
   )
  \right]
 dy
,
\end{align*}
where
$
 U_l ^{(i)} = \mathrm{id} 
$
for $l \leq m$
and
$
 U_l ^{(i)} = T_\varepsilon ^{(i)}
$
for $l \geq m+1$. 

The 
following Proposition \ref{Prop_55a-1} says that 
$G_t$ 
is independent of the choice of 
$
 \mathcal{N}^{(i)} 
\in
 \mathcal{R}^{J_i}
$
and
$
 {\sigma_i |}_{{S^*}}
\in
 \mathrm{Bij}({S^*}, {F_i'})
$,
under the condition
$
 (\mathcal{N}^{(i)}, \sigma_i|_{S^*})
\in
 \Psi(A, r)
$: 

\begin{Prop}
\label{Prop_55a-1}
Let
$
 \sigma|_{S^*}, \sigma'|_{S^*}
\in
 \mathrm{Bij}(S^*, F_i ')
$
and
$
 \mathcal{N}^{(i)}, {\mathcal{N}'}^{(i)}
\in
 \mathcal{R}^{J_i}
$.
If
\begin{equation*}
 \sum_{j\in S^*}
 \delta_
 {
  (
   \mathcal{N}_{\sigma(j)}
  ,
   \mathcal{N}_{\sigma(j)+1}
  )
 }
=
 \sum_{j\in S^*}
 \delta_
 {
  (
   \mathcal{N}'_{\sigma'(j)}
  ,
   \mathcal{N}'_{\sigma'(j)+1}
  )
 }
,
\end{equation*}
then 
$
 G_t(m_1 + m_2, D, E, \sigma, \mathcal{N}) 
=
 G_t(m_1 + m_2, D, E, \sigma', \mathcal{N}') 
$.

Hence, for each
$
 (
  \{\sigma_i |_{S^*}\}
 , 
  \{\mathcal{N}^{(i)}\}
 )
\in
 \Psi(A, r)
$,
we may define
\begin{equation*}
 G_t
 (
  m_1 + m_2
 ,
  D
 ,
  E
 ,
  \{\sigma_i |_{{S^*} ^c}\}
 , 
  \{\mathcal{N}^{(i)}|_{E_i \setminus J_i}\}
 ,
  A
 )
:=
 G_t(m_1 + m_2, D, E, \sigma, \mathcal{N}) 
.
\end{equation*}
\end{Prop}

\begin{proof}
When $l\in \sigma_i ^{-1}(F_i)\setminus S^*$,
we have
$
 \sigma_i (l) 
\in
 F_i \setminus \sigma_i(S^*)
=
 F_i \setminus F_i '
\subset
 E_i \setminus J_i
$
and
$
 \sigma_i (l) + 1
\in
 (F_i + 1) \setminus (F_i ' + 1)
\subset
 E_i \setminus J_i
$.
Recall that $G_t$ is defined via $G_t ^{(i)}$.
By the assumption, for each $i$,
we have
\begin{align*}
  \prod_{l\in \sigma_i ^{-1}(F_i)\setminus S^*}
  \biggl[
   \psi_{n_{\sigma_i(l)    } ^{(i)}} ^{(i)} 
   \psi_{n_{\sigma_i(l) + 1} ^{(i)}} ^{(i)}
  \biggr]
  (x_{l})
=
  \prod_{l\in {\sigma'_i} ^{-1}(F_i)\setminus S^*}
  \biggl[
   \psi_{n_{\sigma_i'(l)    } ^{(i)}} ^{(i)} 
   \psi_{n_{\sigma_i'(l) + 1} ^{(i)}} ^{(i)}
  \biggr]
  (x_{l})
.
\end{align*}

Since
$
 \sigma_i ^{-1} (D_i ^c)
\supset
 \sigma_i ^{-1} (E_i)
\supset
 \sigma_i ^{-1} (F_i)
\supset
 \bigcap_i
 \sigma_i ^{-1} (F_i)
=
 S^*
$,
we may decompose
\begin{align*}
 \prod_{j\in \sigma_i ^{-1}(D_i ^c)}
 \exp
 {
  \biggl\{
   -r_{\sigma_i(j)} 
   \lambda_{n_{\sigma_i(j)} ^{(i)}} ^{(i)} 
  \biggr\}
 }
=&
 \prod_{j\in \sigma_i ^{-1}(D_i ^c) \setminus S^*}
 \exp
 {
  \biggl\{
   -r_{\sigma_i(j)} 
   \lambda_{n_{\sigma_i(j)} ^{(i)}} ^{(i)} 
  \biggr\}
 }
 \prod_{j\in S^*}
 \exp
 {
  \biggl\{
   -r_{\sigma_i(j)} 
   \lambda_{n_{\sigma_i(j)} ^{(i)}} ^{(i)} 
  \biggl\}
 }
.
\end{align*}

Fix 
$
 \{r_{\sigma_i(j)}\}_{j\in (S^*)^c}
$ 
and write
$
 r' 
:=
 \sum_{j\in (S^*)^c} 
 r_{\sigma_i(j)}
$.
Note that
\begin{align*}
&
 \int_{(\delta, t]^{S^*} }
  1_
  {
   \left\{
    \sum_{j\in S^*} r_{\sigma_i(j)} \leq t - r'
   \right\}
  }
  \prod_{j\in S^*}
  \exp
  {
   \biggl\{
    -r_{\sigma_i(j)} 
    \lambda_{n_{\sigma_i(j)} ^{(i)}} ^{(i)} 
   \biggr\}
  }
 dr
\\
=&
 \int_{(\delta, t]^{S^*} }
  1_
  {
   \bigl\{
    \sum_{l\in \mathcal{R}} 
    \sum_
    {
     j \in S^*
    ; 
     n^{(i)} _{\sigma_i(j)} = l
    } 
    r_{\sigma_i(j)}
   \leq
    t - r'
   \bigr\}
  }
  \prod_{l\in \mathcal{R}} 
  \exp
  {
   \Biggl\{
    -
    \Biggl(
     \sum_
     {
      j\in S^*
     ;
      n^{(i)} _{\sigma_i(j)} = l
     } 
     r_{\sigma_i(j)}
    \Biggr)
    \lambda^{(i)} _l
   \Biggr\}
  }
 dr
\end{align*}
depends only on $A$.
Here we abbreviate
the symbol $\{dr_j\}_{j\in S^*}$ as $dr$.

Now we have
\begin{align*}
\begin{split}
&
 \prod_{j\in D_i}
 p^{(i)} _{r_j}
 (
  x_{\sigma_i ^{-1} (j-1)}
 ,
  x_{\sigma_i ^{-1} (j  )}
 )
 \prod_{j\in D_i ^c}
 \exp
 {
  \biggl\{
   -r_j \lambda_{n_j ^{(i)}} ^{(i)} 
  \biggr\}
 }
\\
&
\hspace{10mm}
 \prod_{j\in (D_i ^c -1)\setminus F_i}
 \psi_{n_{j+1} ^{(i)}} ^{(i)} 
 (x_{\sigma_i^{-1}(j)})
 \prod_{j\in D_i ^c \setminus F_i}
 \psi_{n_{j  } ^{(i)}} ^{(i)} 
 (x_{\sigma_i^{-1}(j)})
\cdot
 p^{(i)} _{r_{k+1}}
 (x_{\sigma_i^{-1}(k)}, x_{k+1})
\end{split}
\\
\begin{split}
=&
 \prod_{l\in \sigma_i ^{-1}(D_i)}
 p_{r_{\sigma_i(l)}}
 (
  x_{\sigma_i ^{-1} (\sigma_i(l) -1)}
 ,
  x_{l}
 )
 \prod_{l\in \sigma_i ^{-1}(D_i ^c)}
 \exp
 {
  \biggl\{
   -r_{\sigma_i(l)} 
   \lambda_{n_{\sigma_i(l)} ^{(i)}} ^{(i)} 
  \biggr\}
 }
\\
&
\hspace{10mm}
 \prod_
 {
  l
 \in
  \sigma_i^{-1}
  ((D_i ^c -1)\setminus F_i)
 }
 \psi_{n_{\sigma_i(l)+1} ^{(i)}} ^{(i)} 
 (x_{l})
 \prod_
 {
  l
 \in
  \sigma_i ^{-1}
  (D_i ^c \setminus F_i)
 }
 \psi_{n_{\sigma_i(l)} ^{(i)}} ^{(i)} 
 (x_{l})
\cdot
 p^{(i)} _{r_{k+1}}
 (x_{\sigma_i^{-1}(k)}, x_{k+1})
.
\end{split}
\end{align*}

It remains to prove that, the factor
\begin{align*}
\begin{split}
&
 \prod_{l\in \sigma_i ^{-1}(D_i)}
 p_{r_{\sigma_i(l)}}
 (
  x_{\sigma_i ^{-1} (\sigma_i(l) -1)}
 ,
  x_{l}
 )
 \prod_
 {
  l
 \in
  \sigma_i ^{-1}(D_i ^c)
 \setminus
  S^*
 }
 \exp
 {
  \biggl\{
   -r_{\sigma_i(l)} 
   \lambda_{n_{\sigma_i(l)} ^{(i)}} ^{(i)} 
  \biggr\}
 }
\\
&
\hspace{10mm}
 \prod_
 {
  l
 \in
  \sigma_i^{-1}
  ((D_i ^c -1)\setminus F_i)
 }
 \psi_{n_{\sigma_i(l)+1} ^{(i)}} ^{(i)} 
 (x_{l})
 \prod_
 {
  l
 \in
  \sigma_i ^{-1}
  (D_i ^c \setminus F_i)
 }
 \psi_{n_{\sigma_i(l)} ^{(i)}} ^{(i)} 
 (x_{l})
\cdot
 p^{(i)} _{r_{k+1}}
 (x_{\sigma_i^{-1}(k)}, x_{k+1})
\end{split}
\end{align*}

\noindent
depends only on
$\mathcal{N}^{(i)}|_{D_i ^c\setminus J_i}$,
$\sigma_i|_{{S^*}^c}$,
and is independent of
$\mathcal{N}^{(i)}|_{J_i}$,
${\mathcal{N}^{(i)}}'|_{J_i}$,
$\sigma_i|_{S^*}$
and
$\sigma_i '|_{S^*}$.

We can see that 
$
 E_i ^c -1
\subset
 F_i ^c
$.
We can also see that
$
 l
\in
 \sigma_i ^{-1}
 (
  (D_i^c -1) \setminus F_i
 )
$
implies
$
 \sigma_i(l) + 1
\in
 D_i ^c
\setminus
 J_i
$.
Furthermore,
$
 l
\in 
 \sigma_i ^{-1}
 (
  D_i ^c \setminus F_i
 )
$
implies 
$
 \sigma_i (l) 
\in
 D_i ^c \setminus J_i
$.
Therefore we obtain the desired result.
%
%
%
%
%
%
%
%
\end{proof}


The following proposition says that
the summation of $G_t$ with respect to 
$\sigma_i|_{{S^*}^c}$ is
independent of the value of $m_1$:

\begin{Prop}
\label{Prop_55a-2}
For all
$
 0 \leq m_1, m'_1\leq k - m_2 - m_4
$,
it holds that
\begin{align*}
&
 \sumsigmaSc
 G_t
 (
  m_1 + m_2
 , 
  D
 ,
  E
 , 
  \{\sigma_i|_{{S^*}^c}\}
 ,
  \{\mathcal{N}^{(i)}|_{E_i\setminus J_i}\}
 ,
  A
 )
\\
=&
 \sumsigmaSc
 G_t
 (
  m'_1 + m_2
 , 
  D
 ,
  E
 , 
  \{\sigma_i|_{{S^*}^c}\}
 ,
  \{\mathcal{N}^{(i)}|_{E_i\setminus J_i}\}
 ,
  A
 )
.
\end{align*}
Hence we can define
\begin{equation}
\label{eq_Gtilde}
 \widetilde{G}_t
 (
  m_2
 ,
  m_4
 ,
  D
 ,
  E
 ,
  \{\mathcal{N}^{(i)}|_{E_i \setminus J_i}\}
 ,
  A
 )
:=
 \sumsigmaSc
 G_t
 (
  m_1 + m_2
 , 
  D
 ,
  E
 , 
  \{\sigma_i|_{{S^*}^c}\}
 ,
  \{\mathcal{N}^{(i)}|_{E_i\setminus J_i}\}
 ,
  A
 )
.
\end{equation}
\end{Prop}

\begin{Rmk}
\label{Rmk_55-1}
In 
page 292 of \cite{MR2999298}, they do not 
make the explicit definition of $\widetilde{G}_t$.
In 
our Proposition \ref{Prop_55a-2}, we define 
a new $\widetilde{G}_t$ explicitly, 
with the use of $J_i$ and the summation over 
$\mathrm{Bij}({S^*} ^c, {F_i '}^c)$
in order 
our $\widetilde{G}_t$ to work similarly
to their $\widetilde{G}_t$ in the proof below.
\end{Rmk}

\begin{proof}
Define $\tau\in \mathfrak{S}_k$ by 
\begin{equation*}
 \tau(j)
=
 \begin{cases}
  j + m_2 
 & 
  \text{when }
  1\leq j \leq m_1 
 ,
 \\
  j - m_1 
 & 
  \text{when }
  m_1 + 1\leq j \leq m_1 + m_2
 ,
 \\
  j 
 & 
  \text{when }
  m_1 + m_2 + 1\leq j \leq k
 .
 \end{cases}
\end{equation*}

\noindent
Then we can see that
$
 \tau(S^*)
=
 \{1, \cdots, m_2\}
\cup
 \{k - m_4 + 1, \cdots, k\}
$,
$
 \tau(({S^*}^c)_{\leq m_1 + m_2})
=
 \{1, \cdots, m_2\}
$
and
$
 \tau(({S^*}^c)_{> m_1 + m_2})
=
 \{k - m_4 + 1, \cdots, k\}
$.
Since
\begin{align*}
 \bigotimes_{l\not\in S^*}
 U^{(i)}_{l}
=
 \bigotimes_
 {
  \substack
  {
   l\not\in S^*
  ;\\
   l\leq m
  }
 }
 \mathrm{id}
\otimes
 \bigotimes_
 {
  \substack
  {
   l\not\in S^*
  ;\\
   l> m
  }
 }
 T^{(i)}_{\varepsilon}
\end{align*}
in the definition of 
$G_t(m_1 + m_2, D, E, \sigma, \mathcal{N})$,
we have by 
the same argument as in the proof of Lemma \ref{Lem_55},
\begin{align*}
&
 G_t(m_1 + m_2, D, E, \sigma, \mathcal{N})
\\
\begin{split}
=&
 \int_{E^{\tau({S^*}^c)}}
  \prod_{l\not\in \tau(S^*)}
  f(y_l)
\\
&
\hspace{10mm}
  \prod_{i=1} ^p
  \bigotimes_{l\not\in \tau(S^*)}
  U^{(i)}_{l}
  \left[
   \prod_{l\in \tau(\sigma_i ^{-1}(F_i)\setminus S^*)}
   \biggl[
    \psi_{n_{\sigma_i\circ\tau^{-1}(l)    } ^{(i)}} ^{(i)} 
    \psi_{n_{\sigma_i\circ\tau^{-1}(l) + 1} ^{(i)}} ^{(i)}
   \biggr]
   (y_{l})
  \cdot
   G_t ^{(i)}
   (
    D_i, \mathcal{N}^{(i)};
    \{y_{l}\}_{l\not\in\sigma_i ^{-1}(F_i)}
   )
  \right]
 dy
\end{split}
\\
\begin{split}
=&
 \int_{E^{m_2 + m_4} }
  \prod_
  {
   \substack
   {
    l\leq m_2
   ,
   \\
    k-m_4 < l
   }
  }
  f(y_l)
  \prod_{i=1} ^p
  \left[
   \bigotimes_
   {
    l\leq m_2
   }
   \mathrm{id}
  \otimes
   \bigotimes_
   {
    l > k - m_4
   }
   T^{(i)}_{\varepsilon}
  \right]
\\
&
\hspace{30mm}
  \left[
   \prod_{l\in \tau(\sigma_i ^{-1}(F_i)\setminus S(\sigma))}
   \biggl[
    \psi_{n_{\sigma_i\circ\tau^{-1}(l)    } ^{(i)}} ^{(i)} 
    \psi_{n_{\sigma_i\circ\tau^{-1}(l) + 1} ^{(i)}} ^{(i)}
   \biggr]
   (y_{l})
  \cdot
   G_t ^{(i)}
   (
    D_i, \mathcal{N}^{(i)};
    \{y_{l}\}_{l\not\in\tau\circ\sigma_i ^{-1}(F_i)}
   )
  \right]
 dy
\end{split}
\\
\begin{split}
=&
 \int_{E^{m_2 + m_4} }
  \prod_
  {
   \substack
   {
    l\leq m_2
   ,
   \\
    k-m_4 < l
   }
  }
  f(y_l)
 \cdot
  \prod_{i=1} ^p
  \left[
   \bigotimes_
   {
    l\leq m_2
   }
   \mathrm{id}
  \otimes
   \bigotimes_
   {
    l > k - m_4
   }
   T^{(i)}_{\varepsilon}
  \right]
\\
&
\hspace{30mm}
  \left[
   \prod_
   {
    \substack
    {
     l\in \tau\circ\sigma_i ^{-1}(F_i)
    ,
    \\ 
     l\leq m_2, k-m_4 <l
    }
   }
   \biggl[
    \psi_{n_{\sigma_i\circ\tau^{-1}(l)    } ^{(i)}} ^{(i)} 
    \psi_{n_{\sigma_i\circ\tau^{-1}(l) + 1} ^{(i)}} ^{(i)}
   \biggr]
   (y_{l})
  \cdot
   G_t ^{(i)}
   (
    D_i, \mathcal{N}^{(i)};
    \{y_{l}\}_{l\not\in\tau\circ\sigma_i ^{-1}(F_i)}
   )
  \right]
 dy
.
\end{split}
\end{align*}
Since the last integral does not depend on $m_1$,
the conclusion follows by taking summation with respect to
$\sigma_i|_{{S^*}^c}$.
\end{proof}


\vspace{1eM}

Now, consider the summation $\sum_1$
which appears in \eqref{eq_sums}.
Recall 
\begin{align*}
\begin{split}
 \Psi(A, r)
:=&
 \left\{
  (
   \mathcal{N}^{(i)}|_{J_i}
  , 
   \sigma_i|_{S^*}
  )_{1\leq i \leq p}
 \in
  \bigotimes_{i=1}^p
  (\mathcal{R}^{J_i} 
 \times
  \mathrm{Bij}(S^*, F_i'))
 ;
 \right.
\\
&
\hspace{50mm}
 \left.
  A
 =
  \sum_{j\in S^*}
  \delta_
  {
   (
    \mathcal{N}_{\sigma(j)}, \mathcal{N}_{\sigma(j)+1}
   )
  }
 ,
  r
 =
  \sum_{j\in S_{\leq m} ^*}
  \delta_
  {
   (
    \mathcal{N}_{\sigma(j)}, \mathcal{N}_{\sigma(j)+1}
   )
  }
 \right\}
\end{split}
\end{align*}
and let $A_i$ be the $i$-th marginal of $A$, 
which is a measure on $\mathcal{R}^2$.
We define
\begin{align}
\label{Def_PhiA}
 \Phi(A)
:=&
 \left\{
  \{\mathcal{N}^{(i)}|_{J_i} \}
 \in
  \bigotimes_{i=1} ^p
  \mathcal{R}^{J_i}
 \left|
  \begin{array}{l}
   \text{There exists } 
   \sigma_i ^0 \in \mathrm{Bij}(S^*, F_i'),
   \text{ such that}
  \\
   A_i 
  =
   \displaystyle
   \sum_{j\in S^*}
   \delta_
   {
    (
     \mathcal{N}^{(i)}_{\sigma_i ^0(j)}, \mathcal{N}^{(i)}_{\sigma_i ^0 (j)+1}
    )
   }
   \text{ for all }1\leq i\leq p
  .
  \end{array}
 \right.
 \right\}
,
\end{align}
\begin{align}
 \Psi(A, r, \{\mathcal{N}^{(i)}|_{J_i}\})
:=&
 \left\{
  \{\sigma_i|_{S^*}\}
 \in
  \bigotimes_{i=1} ^p
  \mathrm{Bij}(S^*, F_i')
  \left|
   A
  =
   \sum_{j\in S^*}
   \delta_
   {
    (
     \mathcal{N}_{\sigma(j)}, \mathcal{N}_{\sigma(j)+1}
    )
   }
  \text{ and }
   r
  =
   \sum_{j\in S_{\leq m} ^*}
   \delta_
   {
    (
     \mathcal{N}_{\sigma(j)}, \mathcal{N}_{\sigma(j)+1}
    )
   }
  \right.
 \right\}
.
\end{align}
By 
Lemma \ref{Lem_meas5}, we have
\begin{align*}
 \#
 \Psi(A, r)
=&
 \sum_{\mathcal{N}^{(i)}\in \Phi(A)}
 \#
 \Psi(A, r, \{\mathcal{N}^{(i)}|_{J_i}\})
\\
=&
 \#
 \Phi(A)
 m_1 ! 
 (k - m_1 - m_2 - m_4)!
 \frac
 {
  \displaystyle
  \prod_{i=1}^p 
  \prod_{l^{(i)} \in \mathcal{R}^2}
  A_i (l^{(i)} ) !
 }
 {
  \displaystyle
  \prod_{l \in (\mathcal{R}^2)^p }
  A (l) !
 }
 \prod_{l \in (\mathcal{R}^2)^p }
 \binom{A(l)}{r(l)}
.
\end{align*}

Combining this with Proposition \ref{Prop_55a-1}, 
we have
\begin{align}
\notag
&
 \sum_
 {
  \substack
  {
   \mathcal{N}^{(i)} \in \mathcal{R}^{J_i};\\
   i=1,\cdots, p
  }
 }
 \sum_
 {
  \substack
  {
   {\sigma_i |}_{S^*}
  \in
   \mathrm{Bij}(S^*, F_i')
  ;
  \\
   \text{ for all }i
  }
 }
 1_
  {
   \left\{
    (\mathcal{N}, \sigma)
   \in
    \Psi(A, r)
   \right\}
  }
 \prod_{l\in (\mathcal{R}^2)^p}
 \left[
  (-a(l))^{r(l)}
  a_\varepsilon (l)^{A(l) - r(l)}
 \right]
\cdot
 G_t(m_1 + m_2, D, E, \sigma, \mathcal{N}) 
\\
\notag
=&
 \#
 \Psi(A, r)
 \prod_{l\in (\mathcal{R}^2)^p}
 \left[
  (-a(l))^{r(l)}
  a_\varepsilon (l)^{A(l) - r(l)}
 \right]
\cdot
 G_t
 (
  m_1 + m_2
 ,
  D
 ,
  E
 ,
  \{\sigma_i |_{{S^*} ^c}\}
 , 
  \{\mathcal{N}^{(i)}|_{E_i \setminus J_i}\}
 ,
  A
 )
\\
\label{eq_55a-1}
\begin{split}
=&
 \#
 \Phi(A)
 m_1 !
 (k - m_1 - m_2 - m_4)!
 \frac
 {
  \displaystyle
  \prod_{i=1}^p 
  \prod_{l^{(i)} \in \mathcal{R}^2}
  A_i (l^{(i)} ) !
 }
 {
  \displaystyle
  \prod_{l \in (\mathcal{R}^2)^p }
  A (l) !
 }
 \prod_{l \in (\mathcal{R}^2)^p }
 \binom{A(l)}{r(l)}
 \left[
  (-a(l))^{r(l)}
  a_\varepsilon (l)^{A(l) - r(l)}
 \right]
\\
&
\hspace{70mm}
 G_t
 (
  m_1 + m_2
 ,
  D
 ,
  E
 ,
  \{\sigma_i |_{{S^*} ^c}\}
 , 
  \{\mathcal{N}^{(i)}|_{E_i \setminus J_i}\}
 ,
  A
 )
.
\end{split}
\end{align}

Next, consider the summations $\sum_2$
which appear in \eqref{eq_Ia-5}.
By 
using $\widetilde{G}_t$ 
defined in Proposition \ref{Prop_55a-2},
we have
\begin{align}
\notag
&
 \sumFipr
 \sumNiDE
 \sumNiEJ
 \sumsigmaSc
 \eqref{eq_55a-1}
\\
\label{eq_55a-2}
\begin{split}
=&
 m_1 ! 
 (k - m_1 - m_2 - m_4)!
 \frac
 {
  \displaystyle
  \prod_{i=1}^p 
  \prod_{l^{(i)} \in \mathcal{R}^2}
  A_i (l^{(i)} ) !
 }
 {
  \displaystyle
  \prod_{l \in (\mathcal{R}^2)^p }
  A (l) !
 }
\cdot
 \prod_{l \in (\mathcal{R}^2)^p }
 \binom{A(l)}{r(l)}
 \left[
  (-a(l))^{r(l)}
  a_\varepsilon (l)^{A(l) - r(l)}
 \right]
\\
&
 \sumFipr
 \sumNiDE
 \sumNiEJ
 \#
 \Phi(A)
\cdot
 \widetilde{G}_t
 (
  m_2
 ,
  m_4 
 ,
  D
 ,
  E
 , 
  \{\mathcal{N}^{(i)}|_{E_i \setminus J_i}\}
 ,
  A
 )
.
\end{split}
\end{align}

\noindent
By definition, 
$\#\Phi(A)$ depends only on $m_2 + m_4$
and is independent of $m_1$.
By 
Proposition \ref{Prop_55a-2}, 
the last line of \eqref{eq_55a-2} depends
only on $A$, $m_2$, $m_4$ 
and is independent of $m_1$ and $r$.

Note that
$
 \sum_{m_1 = 0} ^{k - m_2 - m_4 }
 \sum_
 {
  r\in M_{m_1}
 ;
  r \leq A 
 }
=
 \sum_
 {
  r\leq A
 }
$.


Now, consider 
the summation $\sum_4$ 
which appears in \eqref{eq_sums}.
We have
\begin{align}
\notag
&
 \sum_{m_1 = 0} ^{k - m_2 - m_4 }
 \sum_
 {
  \substack
  {
   r\in M_{m_1}
  ;\\
   r \leq A 
  }
 }
 (-1)^{m_2} 
 \frac{k!}{m_1! m_2! (k - m_1 - m_2 - m_4)! m_4!}
 \eqref{eq_55a-2}
\\
\notag
\begin{split}
=&
 (-1)^{m_2}
 \frac{k!}{m_2! m_4!}
 \frac
 {
  \displaystyle
  \prod_{i=1}^p 
  \prod_{l^{(i)} \in \mathcal{R}^2}
  A_i (l^{(i)} ) !
 }
 {
  \displaystyle
  \prod_{l \in (\mathcal{R}^2)^p }
  A (l) !
 }
 \sum_
 {
  r\leq A
 }
 \prod_{l \in (\mathcal{R}^2)^p }
 \binom{A(l)}{r(l)}
 \left[
  (-a(l))^{r(l)}
  a_\varepsilon (l)^{A(l) - r(l)}
 \right]
\\
\notag
&
 \sumFipr
 \sumNiDE
 \sumNiEJ
 \#
 \Phi(A)
\cdot
 \widetilde{G}_t
 (
  m_2
 ,
  m_4 
 ,
  D
 ,
  E
 , 
  \{\mathcal{N}^{(i)}|_{E_i \setminus J_i}\}
 ,
  A
 )
\end{split}
\\
\label{eq_55a-3}
\begin{split}
=&
 (-1)^{m_2}
 \frac{k!}{m_2! m_4!}
 \frac
 {
  \displaystyle
  \prod_{i=1}^p 
  \prod_{l^{(i)} \in \mathcal{R}^2}
  A_i (l^{(i)} ) !
 }
 {
  \displaystyle
  \prod_{l \in (\mathcal{R}^2)^p }
  A (l) !
 }
 \prod_{l \in (\mathcal{R}^2)^p }
 \left[
  a_\varepsilon (l) -a(l))
 \right]^{A(l)}
\\
&
 \sumFipr
 \sumNiDE
 \sumNiEJ
 \#
 \Phi(A)
\cdot
 \widetilde{G}_t
 (
  m_2
 ,
  m_4 
 ,
  D
 ,
  E
 , 
  \{\mathcal{N}^{(i)}|_{E_i \setminus J_i}\}
 ,
  A
 )
.
\end{split}
\end{align}

\noindent
where we used the equality
$
 \sum_{m_1 = 0} ^{k - m_2 - m_4 }
 \sum_
 {
  r\in M_{m_1}
 ;
  r \leq A 
 }
=
 \sum_
 {
  r\leq A
 }
$
and the multinomial formula
\begin{equation*}
 \sum_
 {
  r\leq A
 }
 \prod_{l \in (\mathcal{R}^2)^p }
 \binom{A(l)}{r(l)}
 \left[
  (-a(l))^{r(l)}
  a_\varepsilon (l)^{A(l) - r(l)}
 \right]
=
 \prod_{l \in (\mathcal{R}^2)^p }
 \left[
  a_\varepsilon (l) -a(l))
 \right]^{A(l)}
.
\end{equation*}


\vspace{1eM}

Before 
considering the summation $\sum_4$
which appears in \eqref{eq_sums}, 
we
give some estimates
in order to find an upper bound 
of the last two lines of \eqref{eq_55a-3}.

\begin{Lem}
\label{Lem_56-2}
There exists a constant
$
 C
=
 C(p, p^{(i)}, f, \delta)
>
 0
$
such that the following inequality holds
for all $D$, $E$, $m_2$, $m_4$, $A$ and $t$:
\begin{align*}
&
 \sumFipr
 \sumNiDE
 \sumNiEJ
\\
&
\hspace{40mm}
 \#
 \Phi(A)
\cdot
 \widetilde{G}_t
 (
  m_2
 ,
  m_4 
 ,
  D
 ,
  E
 , 
  \{\mathcal{N}^{(i)}|_{E_i \setminus J_i}\}
 ,
  A
 )
\\
\begin{split}
=&
 \sumFipr
 \sumNiDE
 \sumNiEJ
 \sumsigmaSc
\\
&
\hspace{40mm}
 \#
 \Phi(A)
\cdot
 G_t
 (
  m_2 , m_4
 ,
  D
 ,
  E
 ,
  \{\sigma_i |_{{S^*} ^c}\}
 , 
  \{\mathcal{N}^{(i)}|_{E_i \setminus J_i}\}
 ,
  A
 )
\end{split}
\\
\leq&
 C^k
 (k!)^p
 e^{pt}
 k^p
 \Biggl(
  \prod_{i=1} ^p
  \prod_{l^{(i)} \in \mathcal{R}^2 }
  A_i (l^{(i)}) !
 \Biggr)^{-1}
.
\end{align*}
\end{Lem}

\begin{proof}
The 
first equality is obtained by the definition of 
$\widetilde{G}_t$ in Proposition \ref{Prop_55a-2}.
So we will prove the inequality in the third line.

First, we estimate $\Phi(A)$. 
Recall
\begin{align*}
 \Phi(A)
:=&
 \left\{
  \{\mathcal{N}^{(i)}|_{J_i} \}
  \in \otimes\mathcal{R}^{J_i}
 \left|
  \begin{array}{l}
   \text{There exists }
   \sigma_i ^0 \in \mathrm{Bij}(S^*, F_i'),
   \text{ such that}
  \\
   A_i 
  =
   \displaystyle
   \sum_{j\in S^*}
   \delta_
   {
    (
     \mathcal{N}^{(i)}_{\sigma_i ^0(j)}
    ,
     \mathcal{N}^{(i)}_{\sigma_i ^0 (j)+1}
    )
   }
  ,
   \text{ for all }1\leq i\leq p
  .
  \end{array}
 \right.
 \right\}
\\
=&
 \left\{
  \{\mathcal{N}^{(i)}|_{J_i} \}
  \in \otimes\mathcal{R}^{J_i}
 \left|
   A_i 
  =
   \sum_{j\in F_i'}
   \delta_
   {
    (
     \mathcal{N}^{(i)}_{j}
    ,
     \mathcal{N}^{(i)}_{j+1}
    )
   }
  ,
   \text{ for all }1\leq i\leq p
 \right.
 \right\}
\end{align*}
and then, by Lemma \ref{Lem_meas4},
\begin{align*}
 \# \Phi(A)
\leq&
 \prod_{i=1} ^p
 \#
 \left\{
  \{\mathcal{N}^{(i)}|_{F_i'} \}
 \in 
  \mathcal{R}^{F_i'}
 \left|
   A_i 
  =
   \sum_{j\in F_i'}
   \delta_
   {
    (
     \mathcal{N}^{(i)}_{j}
    ,
     \mathcal{N}^{(i)}_{j+1}
    )
   }
 \right.
 \right\}
\leq
 \prod_{i=1} ^p
 \Biggl(
  k
  \cdot
  \frac
  {
   \prod_{l_1 ^{(i)} \in \mathcal{R} }
   \overline{A}_i (l_1 ^{(i)}) !
  }
  {
   \prod_{l^{(i)} \in \mathcal{R}^2 }
   A_i (l^{(i)}) !
  }
 \Biggr)
,
\end{align*}
where the second inequality follows since 
$
 A\in M_{k-m_2 - m_4}
$
and
$
 \sum_{l_1 ^{(i)} \in \mathcal{R}} 
 \overline{A}_i (l_1 ^{(i)}) 
=
 k - m_2 - m_4
$.

Next we estimate $G_t$.
Write $E' := \text{supp}[f]$. 
For each $i$, we have
\begin{align*}
 \Biggl\|
  \prod_{j\in F_i \setminus F_i '}
  \biggl[
   \psi_{n_{j    } ^{(i)}} ^{(i)} 
   \psi_{n_{j + 1} ^{(i)}} ^{(i)}
  \biggr]
  (y_{l})
 \Biggr\|_
 {
  L^p
  (
   {E'}^{F_i \setminus F_i '}
  )
 }
\leq
 \Biggl(
 \prod_{j\in F_i \setminus F_i '}
 \|
  \psi_{n_{j    } ^{(i)}} ^{(i)} 
 \|_\infty
 \|
  \psi_{n_{j + 1} ^{(i)}} ^{(i)}
 \|_\infty
 \Biggr)
\cdot
 m(E')^{\# (F_i\setminus F_i ')/p}
\end{align*}

\noindent
and have
\begin{align*}
&
 \Bigl\|
  G_t ^{(i)}
  (
   D_i, \mathcal{N}^{(i)};
   \{x_{l}\}_{l\not\in F_i}
  )
 \Bigr\|_
 {
  L^p
  (
   {E'}^{{F_i}^c}
  )
 }
\\
\begin{split}
\leq&
 e^{t} 
 \biggl[
  \sup_{y\in E}
  \int_E
   \Gr{i} (x, y)^p
  dx
 \biggr]^{\# D_i/p}
\cdot
 \prod_{j\in D_i ^c\setminus J_i}
 \exp
 \biggl\{
  -\delta(1 + \lambda^{(i)} _{n^{(i)} _j})
 \biggr\}
\\
&
\hspace{50mm}
 \Biggl(
  \prod_{j\in (D_i ^c - 1) \setminus F_i}
  \|
   \psi^{(i)} _{n_{j+1} ^{(i)}}
  \|_\infty
  \prod_{j\in D_i ^c \setminus F_i}
  \|
   \psi^{(i)} _{n_{j} ^{(i)}}
  \|_\infty 
 \Biggr)
\cdot
 m(E')^{\# (D_i ^c \setminus F_i)/p}
.
\end{split}
\end{align*}

\noindent
Since
\begin{align*}
&
 \Biggl(
  \prod_{j\in F_i \setminus F_i '}
  \|
   \psi_{n_{j    } ^{(i)}} ^{(i)} 
  \|_\infty
  \|
   \psi_{n_{j + 1} ^{(i)}} ^{(i)}
  \|_\infty
 \Biggr)
 \Biggl(
  \prod_{j\in (D_i ^c - 1) \setminus F_i}
  \|
   \psi^{(i)} _{n_{j+1} ^{(i)}}
  \|_\infty
  \prod_{j\in D_i ^c \setminus F_i}
  \|
   \psi^{(i)} _{n_{j} ^{(i)}}
  \|_\infty 
 \Biggr)
\\
\leq&
 \Biggl(
  \prod_{j\in F_i \setminus F_i '}
  \Bigl(
   \|
    \psi_{n_{j    } ^{(i)}} ^{(i)} 
   \|_\infty
  +
   1
  \Bigr)
  \Bigl(
   \|
    \psi_{n_{j + 1} ^{(i)}} ^{(i)} 
   \|_\infty
  +
   1
  \Bigr)
 \Biggr)
 \Biggl(
  \prod_{j\in (D_i ^c - 1) \setminus F_i}
  \Bigl(
   \|
    \psi_{n_{j + 1} ^{(i)}} ^{(i)} 
   \|_\infty
  +
   1
  \Bigr)
  \prod_{j\in D_i ^c \setminus F_i}
  \Bigl(
   \|
    \psi_{n_{j    } ^{(i)}} ^{(i)} 
   \|_\infty
  +
   1
  \Bigr)
 \Biggr)
\\
=&
 \Biggl(
  \prod_{j\in F_i '}
  \Bigl(
   \|
    \psi_{n_{j    } ^{(i)}} ^{(i)} 
   \|_\infty
  +
   1
  \Bigr)^{-1}
  \Bigl(
   \|
    \psi_{n_{j + 1} ^{(i)}} ^{(i)} 
   \|_\infty
  +
   1
  \Bigr)^{-1}
 \Biggr)
 \Biggl(
  \prod_{j\in (D_i ^c - 1) }
  \Bigl(
   \|
    \psi_{n_{j + 1} ^{(i)}} ^{(i)} 
   \|_\infty
  +
   1
  \Bigr)
  \prod_{j\in D_i ^c }
  \Bigl(
   \|
    \psi_{n_{j    } ^{(i)}} ^{(i)} 
   \|_\infty
  +
   1
  \Bigr)
 \Biggr)
\\
\leq&
 \prod_{j\in D_i ^c\setminus J_i}
 \Bigl(
  \|
   \psi^{(i)} _{n_{j} ^{(i)}}
  \|_\infty
 +
  1
 \Bigr)^2
,
\end{align*}

\noindent
we have
\begin{align*}
&
 G_t
 (
  m_2 , m_4
 ,
  D
 ,
  E
 ,
  \{\sigma_i |_{{S^*} ^c}\}
 , 
  \{\mathcal{N}^{(i)}|_{E_i \setminus J_i}\}
 ,
  A
 )
\\
=&
 \int_{E^{{S^*}^c} }
  \prod_{l\not\in S^*}
  f(y_l)
 \cdot
  \prod_{i=1} ^p
  \bigotimes_{l\not\in S^*}
  T^{(i)}_{l}
  \left[
   \prod_{l\in \sigma_i ^{-1}(F_i)\setminus S^*}
   \biggl[
    \psi_{n_{\sigma_i(l)    } ^{(i)}} ^{(i)} 
    \psi_{n_{\sigma_i(l) + 1} ^{(i)}} ^{(i)}
   \biggr]
   (y_{l})
  \cdot
   G_t ^{(i)}
   (
    D_i, \mathcal{N}^{(i)};
    \{y_{l}\}_{l\not\in\sigma_i ^{-1}(F_i)}
   )
  \right]
 dy
\\
\leq&
 \| f \|_\infty ^{m_2 + m_4}
 \prod_{i=1} ^p
 \left\|
  \bigotimes_{l\not\in S^*}
  T^{(i)}_{l}
  \left[
   \prod_{l\in \sigma_i ^{-1}(F_i)\setminus S^*}
   \biggl[
    \psi_{n_{\sigma_i(l)    } ^{(i)}} ^{(i)} 
    \psi_{n_{\sigma_i(l) + 1} ^{(i)}} ^{(i)}
   \biggr]
   (y_{l})
  \cdot
   G_t ^{(i)}
   (
    D_i, \mathcal{N}^{(i)};
    \{y_{l}\}_{l\not\in\sigma_i ^{-1}(F_i)}
   )
  \right]
 \right\|_{L^p({E'}^{{S^*}^c})}
\\
\leq&
 \| f \|_\infty ^{m_2 + m_4}
 \prod_{i=1} ^p
 \left\|
  \prod_{l\in \sigma_i ^{-1}(F_i)\setminus S^*}
  \biggl[
   \psi_{n_{\sigma_i(l)    } ^{(i)}} ^{(i)} 
   \psi_{n_{\sigma_i(l) + 1} ^{(i)}} ^{(i)}
  \biggr]
  (y_{l})
 \cdot
  G_t ^{(i)}
  (
   D_i, \mathcal{N}^{(i)};
   \{y_{l}\}_{l\not\in\sigma_i ^{-1}(F_i)}
  )
 \right\|_{L^p({E'}^{{S^*}^c})}
\\
\begin{split}
\leq&
 \| f \|_\infty ^{m_2 + m_4}
 \prod_{i=1} ^p
 \Biggl\{
  e^t
  \biggl[
   \sup_{y\in E}
   \int_E
    \Gr{i} (x, y)^p
   dx
  \biggr]^{\# D_i/p}
 \cdot
  \prod_{j\in D_i ^c\setminus J_i}
  \exp
  \biggl\{
   -\delta(1 + \lambda^{(i)} _{n^{(i)} _j})
  \biggr\}
\\
&
\hspace{80mm}
 \prod_{j\in D_i ^c\setminus J_i}
 \Bigl(
  \|
   \psi^{(i)} _{n_{j} ^{(i)}}
  \|_\infty
 +
  1
 \Bigr)^2
  m(E')^{\# (D_i ^c\setminus F_i ')/p}
 \Biggr\}
\end{split}
\\
\leq&
 C(p, p^{(i)}, f)^k
 e^{pt}
 \prod_{i=1} ^p
 \prod_{j\in D_i ^c\setminus J_i}
 \exp
 \biggl\{
  -\delta(1 + \lambda^{(i)} _{n^{(i)} _j})
 \biggr\}
 \Bigl(
  \|
   \psi^{(i)} _{n_{j} ^{(i)}}
  \|_\infty
 +
  1
 \Bigr)^2
.
\end{align*}

\noindent
We also have
\begin{align*}
&
 \sumNiDE
 \sumNiEJ
 \prod_{i=1} ^p
 \prod_{j\in D_i ^c\setminus J_i}
 \exp
 \biggl\{
  -\delta(1 + \lambda^{(i)} _{n^{(i)} _j})
 \biggr\}
 \Bigl(
  \|
   \psi^{(i)} _{n_{j} ^{(i)}}
  \|_\infty
 +
  1
 \Bigr)^2
\\
\leq&
 \sum_
 {
  \substack
  {
   \mathcal{N}^{(i)} 
  \in 
   (\mathbb{Z}_{>0})^{D_i ^c\setminus J_i};
  \\
   i=1,\cdots, p
  }
 }
 \prod_{i=1} ^p
 \prod_{j\in D_i ^c\setminus J_i}
 \exp
 \biggl\{
  -\delta(1 + \lambda^{(i)} _{n^{(i)} _j})
 \biggl\}
 \Bigl(
  \|
   \psi^{(i)} _{n_{j} ^{(i)}}
  \|_\infty
 +
  1
 \Bigr)^2
\\
=&
 \prod_{i=1} ^p
 \Biggl[
  \sum_{n=1} ^\infty
  \exp
  \Bigl\{
   -\delta(1 + \lambda^{(i)} _{n})
  \Bigr\}
  \Bigl(
   \|
    \psi^{(i)} _{n}
   \|_\infty
  +
   1
  \Bigr)^2
 \Biggr]^{\# (D_i^c \setminus J_i)}
\\
\leq&
 C(p, p^{(i)}, \delta)^k
.
\end{align*}

\noindent
Inequalities
$
 \binom{\# F_i}{\# S^*}
\leq
 \sum_{n=0} ^{\# F_i}
 \binom{\# F_i}{\# S^*}
=
 2^{\# F_i}
\leq
 2^k
$
and
$
 (k - m_2 - m_4)!
 (m_2 + m_4)!
\leq
 k!
$
yield the desired result.
\end{proof}


\vspace{1eM}

Now 
we will consider the summations $\sum_4$ 
which appear in \eqref{eq_sums}.
Since
$
 \# D_i \leq \eta k
$
and
$
 \# (D_i ^c \setminus E_i) \leq \gamma k
$,
we have
$
 \# D_i ^c \geq (1-\eta)k
$
and
$
 \# (D_i ^c \setminus (E_i - 1))
\leq 
 \gamma k
$.
Therefore
\begin{align*}
 \# (E_i \cap (E_i -1))
=&
 \# D_i ^c
-
 \#(D_i ^c \setminus (E_i \cap (E_i -1)) )
\\
\geq&
 \# D_i ^c
-
 \#(D_i ^c \setminus  E_i    )
-
 \#(D_i ^c \setminus (E_i -1))
\\
\geq&
 (1 - 2(\eta + \gamma)) k
.
\end{align*}

\noindent
We also have
\begin{align*}
 m_2 + m_4
=&
 \# (S^*)^c
=
 \#  
 \Bigl(
  \bigcap_i \sigma_i^{-1}(F_i)
 \Bigr)^c
\\
\leq&
 \sum_i \# \sigma_i ^{-1}(F_i)^c
=
 \sum_i \# (F_i)^c
=
 \sum_i \# (E_i\cap (E_i - 1))^c
\\
\leq&
 2p(\eta + \gamma) k
\end{align*}
and hence
$
 k - m_2 - m_4
\geq
 (1 - 2p(\eta + \gamma))k
$.

For 
$
 l
=
 (l_i ^{(1)}, l_i ^{(2)})_{1\leq i\leq p}
\in
 (\mathcal{R}^2)^p
$,
we have
$
 \psi ^{(i)} _{l_i ^{(1)}}
,
 \psi ^{(i)} _{l_i ^{(2)}}
\in
 L^{2p}(E)
$
and the similar estimate 
as \eqref{eq_Prop4-1-3} yields 
\begin{align*}
 |a(l) - a_\varepsilon(l)|
=&
 \Biggl|
  \biggl\langle
   f
  ,
   \prod_{i=1} ^p
   \biggl[
    \psi ^{(i)} _{l_i ^{(1)}}
    \psi ^{(i)} _{l_i ^{(2)}}
   \biggr]
  \biggr\rangle
 -
  \biggl\langle
   f
  ,
   \prod_{i=1} ^p
   T^{(i)}_\varepsilon
   \biggl[
    \psi ^{(i)} _{l_i ^{(1)}}
    \psi ^{(i)} _{l_i ^{(2)}}
   \biggr]
  \biggr\rangle
 \Biggr|
\\
\leq&
 \|f\|_\infty
 \biggl\|
  \prod_{i=1} ^p
  [
   \psi ^{(i)} _{l_i ^{(1)}}
   \psi ^{(i)} _{l_i ^{(2)}}
  ]
 -
  \prod_{i=1} ^p
  T^{(i)}_\varepsilon
  [
   \psi ^{(i)} _{l_i ^{(1)}}
   \psi ^{(i)} _{l_i ^{(2)}}
  ]
 \biggr\|_{1}
\\
\leq&
 \|f\|_\infty
 \sum_{i=1} ^p
 \Biggl(
  \biggl\|
   \biggl[
    \psi ^{(i)} _{l_i ^{(1)}}
    \psi ^{(i)} _{l_i ^{(2)}}
   \biggr]
  -
   T^{(i)}_\varepsilon
   \biggl[
    \psi ^{(i)} _{l_i ^{(1)}}
    \psi ^{(i)} _{l_i ^{(2)}}
   \biggr]
  \biggr\|_{p}
 \cdot
  \prod_{j\not=i}
  \Bigl\|
   \psi ^{(i)} _{l_j ^{(1)}}
  \Bigr\|_{2p} ^2
  \Bigl\|
   \psi ^{(i)} _{l_j ^{(2)}}
  \Bigr\|_{2p} ^2
 \Biggr)
\\
\rightarrow&
 0
\quad
 \text{as } \varepsilon \rightarrow 0
.
\end{align*}

Write
$
 C_{R}(\varepsilon)
=
 \sum_{l \in (\mathcal{R}^2)^p }
 \left[
  a_\varepsilon (l) -a(l))
 \right]
$
and choose small $\varepsilon_R >0$ such that
$
 C_{R}(\varepsilon)
<
 1/2
,
 \text{ for all }
 \varepsilon < \varepsilon_R
$.
Then
\begin{equation*}
 C_{R}(\varepsilon)^
 {
  (2^{-1}) (1 - 2p(\eta + \gamma))k
 }
 k^p
\leq
 \sup_{k\geq 0}
 \left[
  (1/2)^
  {
   (2^{-1}) (1 - 2p(\eta + \gamma))k
  }
  k^p
 \right]
<
 \infty
.
\end{equation*}

\vspace{1eM}

Consider 
the summation $\sum_{4-1}$
which appears in \eqref{eq_sums}.
By 
Lemma \ref{Lem_56-2}, we have
\begin{align}
\notag
&
 \sum_
 {
  0\leq m_2 + m_4 \leq k
 }
 \sum_{A\in M_{k - m_2 - m_4}}
 \eqref{eq_55a-3}
\\
\notag
\begin{split}
\leq&
 \sum_
 {
  0 \leq  m_2  + m_4 \leq k
 }
 \sum_{A\in M_{k - m_2-m_4}}
 \frac{k!}{m_2! m_4!}
 \frac
 {
  \displaystyle
  \prod_{i=1}^p 
  \prod_{l^{(i)} \in \mathcal{R}^2}
  A_i (l^{(i)} ) !
 }
 {
  \displaystyle
  \prod_{l \in (\mathcal{R}^2)^p }
  A (l) !
 }
\cdot
 \prod_{l \in (\mathcal{R}^2)^p }
 \left[
  a_\varepsilon (l) -a(l))
 \right]^{A(l)}
\\
\notag
&
\hspace{20mm}
\cdot
 C(p, p^{(i)}, f, \delta)^k
 (k!)^p
 e^{pt}
 k^p
 \Biggl(
  \prod_{i=1} ^p
  \prod_{l^{(i)} \in \mathcal{R}^2 }
  A_i (l^{(i)}) !
 \Biggr)^{-1}
\end{split}
\\
\notag
\leq&
 (k!)^p
 k^p
 C(p, p^{(i)}, f, \delta)^k
 e^{pt}
 \sum_
 {
  0 \leq  m_2  + m_4 \leq k
 }
 \frac{k!}{m_2! m_4!(k-m_2 - m_4)!}
\cdot
 \Biggl(
  \sum_{l \in (\mathcal{R}^2)^p }
  [
   a_\varepsilon (l) -a(l))
  ]
 \Biggr)^{k-m_2 - m_4}
\\
\notag
\leq&
 (k!)^p
 k^p
 C(p, p^{(i)}, f, \delta)^k
 e^{pt}
 \sum_
 {
  0 \leq  m_2  + m_4 \leq k
 }
 \frac{k!}{m_2! m_4! (k-m_2 - m_4)!}
\cdot
 C_R (\varepsilon) ^{(1 - 2p(\eta + \gamma))k}
\\
\notag
\leq&
 (k!)^p
 k^p
 C(p, p^{(i)}, f, \delta)^k
 e^{pt}
 C_R (\varepsilon) ^{ (1 - 2p(\eta + \gamma))k}
\\
\leq&
\label{eq_56a-1}
 (k!)^p
 C(p, p^{(i)}, f, \delta)^k
 e^{pt}
 C_R (\varepsilon) ^{(1/2)(1 - 2p(\eta + \gamma))k}
\quad
 \text{for small } 
 \varepsilon < \varepsilon_R
.
\end{align}

\vspace{1eM}

Finally, 
consider the summation $\sum_{4-2}$ 
which appears in \eqref{eq_sums}. 
Then
\begin{align*}
&
 \sumDiall
 \sumEiall
 \eqref{eq_56a-1}
\\
\leq&
 2^{kp}
\cdot
 2^{kp}
\cdot 
 (k!)^p
 C(p, p^{(i)}, f, \delta)^k
 e^{pt}
 C_R (\varepsilon) ^{(1/2)(1 - 2p(\eta + \gamma))k}
\quad
 \text{for small } 
 \varepsilon < \varepsilon_R
.
\end{align*}
Therefore 
we complete the proof of Lemma \ref{Lem_C}.

\begin{flushright}
 \vspace{-2eM}
 $\square$ 
\end{flushright}


\section{Proof of Theorem \ref{Prop_mgf}}
\label{Sec_proof_mgf}

\subsection{Applying Varadhan's integral lemma}

First,
recall that
$t^{-1}\ell_t ^{(i)}\in \mathcal{M}_1(E)$ satisfies 
the LDP
as $t\rightarrow \infty$,
with probability $\widetilde{\mathbb{P}}_t$,
scale $t$ and the good rate function $J^{(i)}$,
and that
$t^{-1}\ell_{t,\varepsilon} ^{(i)}\in \mathcal{M}_{\leq 1}(E)$ satisfies 
the LDP
as $t\rightarrow \infty$,
with probability $\widetilde{\mathbb{P}}_t$,
scale $t$ and the good rate function $J_\varepsilon ^{(i)}$,
where
\begin{align*}
 J_\varepsilon ^{(i)}  (\nu)
:=&
 \inf
 \{
  J ^{(i)} (\mu)
 ;
  \mu \in \mathcal{M}_1(E),
  p_\varepsilon ^{(i)}  [\mu] = \nu
 \}
\\
=&
 \inf
 \{
  \mathcal{E}^{(i)} (\psi, \psi)- \lambda_1^{(i)} 
 ;
  \psi \in \mathcal{F}^{(i)} , \|\psi\|_2=1,
  p_\varepsilon^{(i)}  [\psi^2] = \nu
 \}
\end{align*}
for $\nu \in \mathcal{M}_{\leq 1}(E)$.

Fix 
$
 f\in C_b (E)
$,
$\theta>0$ 
and define 
$
 \Phi
:
 \mathcal{M}_{\leq 1}(E) 
\rightarrow
 \mathbb{R}
$
by
$
 \Phi(\nu)
:=
 \theta
 \langle f, \nu \rangle
$
for
$
 \nu\in \mathcal{M}_{\leq 1}(E)
$.
Then $\Phi$ is continuous 
with respect to the weak topology,
hence we obtain
\begin{equation*}
 \limsup_{t\rightarrow \infty}
 \frac{1}{t}
 \log
 \widetilde{\mathbb{E}}_{t}
 [
  e^{\gamma t \Phi (t^{-1}\ell_{t, \varepsilon} ^{(i)}) }
 ]
\leq
 \gamma \theta \|f\|_\infty
<
 \infty
\end{equation*}
for any $\gamma > 1$.

Now, 
applying Varadhan's integral lemma
(Lemma \ref{Lem_DZ_Varadhan})
with $t^{-1}\ell^{(i)} _{t, \varepsilon}$
and we obtain 
\begin{equation*}
 \lim_{t\rightarrow \infty}
 \frac{1}{t}
 \log
 \widetilde{\mathbb{E}}_{t}
 [
  e^
  {
   \theta 
   \langle 
    f, \ell_{t, \varepsilon} ^{(i)}
   \rangle
  }
 ]
=
 \lim_{t\rightarrow \infty}
 \frac{1}{t}
 \log
 \widetilde{\mathbb{E}}_{t}
 [
  e^{t \Phi (t^{-1} \ell_{t, \varepsilon} ^{(i)}) }
 ]
=
 \sup_{\nu \in \mathcal{M}_{\leq 1}(E)}
 \left\{
  \Phi(\nu)
 -
  J_\varepsilon ^{(i)}(\nu)
 \right\}
.
\end{equation*}
By 
Lemma \ref{Lem_6-1} stated below, 
we can rewrite the equality as
\begin{equation}
\label{eq_Varadhan}
 \lim_{t\rightarrow \infty}
 \frac{1}{t}
 \log
 \widetilde{\mathbb{E}}_{t}
 [
  e^
  {
   \theta 
   \langle 
    f, \ell_{t, \varepsilon} ^{(i)} 
   \rangle
  }
 ]
=
 \sup_{\psi \in \mathcal{F} ^{(i)}, \|\psi\|_2 = 1}
 \biggl\{
  \theta
  \int_E
   p_\varepsilon ^{(i)} [\psi^2] f
  dm
 -
  \mathcal{E} ^{(i)}(\psi, \psi)
 +
  \lambda_1 ^{(i)}
 \biggr\}
.
\end{equation}

\begin{Lem}
\label{Lem_6-1}
For each $i = 1, \cdots, p$ and $\theta > 0$,
\begin{equation*}
 \sup_{\nu \in \mathcal{M}_{\leq 1}(E)}
 \bigl\{
  \Phi(\nu)
 -
  J^{(i)} _\varepsilon(\nu) 
 \bigr\}
=
 \sup_{\psi \in \mathcal{F} ^{(i)}, \|\psi\|_2 = 1}
 \biggl\{
  \theta
  \int_E
   p_\varepsilon ^{(i)} [\psi^2] f
  dm
 -
  \mathcal{E} ^{(i)}(\psi, \psi)
 +
  \lambda_1 ^{(i)}
 \biggl\}
.
\end{equation*}
\end{Lem}

\begin{proof}
First, fix $\nu\in \mathcal{M}_{\leq 1}(E)$ and
suppose $J_\varepsilon ^{(i)}(\nu)<\infty$.
For each $n$,
choose $\psi_n\in \mathcal{F} ^{(i)}$ with $\|\psi_n\|_2 = 1$,
$
 p_\varepsilon ^{(i)}[\psi_n^2] 
=
 \frac{d\nu}{dm}
$
and
$
 \mathcal{E} ^{(i)}(\psi_n, \psi_n) - \lambda_1
\leq
 J_\varepsilon ^{(i)}(\nu) + n^{-1}
$.
We have
\begin{align*}
 \Phi(\nu) - J_\varepsilon ^{(i)}(\nu)
\leq&
 \theta
 \langle
  f, p_\varepsilon ^{(i)} [\psi_n ^2]
 \rangle
-
 \mathcal{E} ^{(i)}(\psi_n, \psi_n) 
+
 \lambda_1 ^{(i)}
+
 n^{-1}
\\
\leq&
 \sup_{\psi \in \mathcal{F} ^{(i)}, \|\psi\|_2 = 1}
 \biggl\{
  \theta
  \int_E
   p_\varepsilon ^{(i)} [\psi^2] f
  dm
 -
  \mathcal{E} ^{(i)}(\psi, \psi)
 +
  \lambda_1 ^{(i)}
 \biggr\}
+
 n^{-1}
\end{align*}
and the upper bound follows.

Conversely, fix $\psi \in \mathcal{F} ^{(i)}$ 
with $\|\psi\|_2 =1$.
We have
\begin{align*}
 \theta
 \int_E
  (p_\varepsilon ^{(i)} [f]) \psi^2
 dm
-
 \mathcal{E} ^{(i)}(\psi, \psi)
+
 \lambda_1 ^{(i)}
\leq&
 \Phi (p_\varepsilon ^{(i)}[\psi^2])
-
 J_\varepsilon ^{(i)}(p_\varepsilon ^{(i)}[\psi^2])
\\
\leq&
 \sup_{\nu \in \mathcal{M}_{\leq 1}}
 \bigl\{
  \Phi(\nu)
 -
  J_\varepsilon ^{(i)}(\nu)
 \bigr\}
\end{align*}
and the lower bound follows.
\end{proof}


\vspace{1eM}

Next, we will extend \cite[Theorem 7]{MR2165257}
under our setting.
Throughout this section, 
with some abuse of notations,
we denote
\begin{equation*}
 \frac{1}{t}
 \ell_{t, \varepsilon} ^{(i)}
=
 L_{t, \varepsilon} ^{(i)}
:=
 \frac{1}{t}
 \frac
 {d \ell_{t, \varepsilon} ^{(i)}}
 {dm}
,\quad 
 \frac{1}{t^p}
 \ell_{t, \varepsilon} ^{\mathrm{IS}}
=
 L_{t, \varepsilon} ^{\mathrm{IS}} 
:=
 \frac{1}{t^p}
 \frac
 {d \ell_{t, \varepsilon} ^{\mathrm{IS}} } 
 {dm}
.
\end{equation*}

\begin{Prop}
\label{Prop_6-2}
Let $h\in \mathcal{B}_b(E)$ 
be nonnegative and compactly supported.
Write 
$m_h(dx) := h(x) m(dx)$.
Suppose 
$E' := \text{supp}[m_h]$ is compact and
each $p_t ^{(i)} (\cdot, \cdot)$ 
is uniformly continuous in $E\times E$. 
Then,
\begin{align}
\notag
&
 \lim_{t \rightarrow \infty}
 \frac{1}{t}
 \log
 \widetilde{\mathbb{E}}_{t}
 \exp
 \left\{
  \theta
   \|
    \ell_{t, \varepsilon} ^{(i)}
   \|_{L^p(m_h)}
 \right\}
\\
\label{eq_prop62_1}
&
\hspace{30mm}
=
 \sup_{\psi\in \mathcal{F} ^{(i)}, \|\psi\|_{L^2(m)}=1}
 \left\{
  \theta
  \|
   p_\varepsilon ^{(i)}[\psi^2]
  \|_{L^p(m_h)}
 -
  \mathcal{E} ^{(i)}(\psi, \psi)
 +
  \lambda_1 ^{(i)}
 \right\}
,
\\
\notag
&
 \lim_{t\rightarrow \infty}
 \frac{1}{t}
 \log
 \widetilde{\mathbb{E}}_{t}
 \exp
 \Biggl\{
  \theta 
  \left(
   \int_E
    \prod_{i=1} ^p
    \ell_{t, \varepsilon} ^{(i)} 
   dm_h
  \right)^{1/p}
 \Biggr\}
\\
\label{eq_prop62_2}
&
\hspace{30mm}
=
 \frac{1}{p}
 \sum_{i=1} ^p
 \sup_{\psi\in \mathcal{F} ^{(i)}, \|\psi\|_{L^2(m)}=1}
 \left\{
  \theta
  \|
   p_\varepsilon ^{(i)} [\psi^2]
  \|_{L^{p}(m_h)}
 -
  p\mathcal{E}^{(i)} (\psi,\psi) 
 +
  p\lambda_1 ^{(i)}
 \right\}
.
\end{align}
\end{Prop}

\begin{proof}
[Proof of Proposition \ref{Prop_6-2}]
Our proof is based on the proof of \cite[Theorem 7]{MR2165257}.

First 
we prove the lower bound of \eqref{eq_prop62_1}.
Take 
$q>0$ such that $p^{-1} + q^{-1} =1$.
For 
any function $f\in C_b(E)$ with $\|f\|_{L^q(m_h)} =1$,
we have
$
 \bigl(
  \int_E
   {\ell_{t, \varepsilon} ^{(i)} }^p
  dm_h
 \bigr)^{1/p}
\geq
 \int_E
  f
  \ell_{t, \varepsilon} ^{(i)}
 dm_h
=
 \int_E
  f
 \cdot
  \ell_{t, \varepsilon} ^{(i)}
 \cdot
  h
 dm
$
and 
by the equality \eqref{eq_Varadhan}, 
\begin{align*}
 \liminf_{t\rightarrow \infty}
 \frac{1}{t}
 \log
 \widetilde{\mathbb{E}}_{t}
 \exp
 \Biggl\{
  \theta 
  \biggl(
   \int_E
    {\ell_{t, \varepsilon} ^{(i)} }^p
   dm_h
  \biggr)^{1/p}
 \Biggr\}
\geq
 \sup_{\psi\in \mathcal{F}^{(i)}, \|\psi\|_2=1}
 \biggl\{
  \theta
  \int_E
   p_\varepsilon^{(i)}[\psi^2] f
  dm_h
 -
  \mathcal{E}^{(i)}(\psi, \psi)
 +
  \lambda_1^{(i)}
 \biggr\}
.
\end{align*}
Since 
$
 p_\varepsilon ^{(i)} [\psi^2]
\in
 L^p(E; m)
\subset
 L^p(E; m_h)
$
, 
by letting 
$
 f
\rightarrow 
 \frac
 {
  p_\varepsilon ^{(i)}[\psi^2] ^{p-1} 
 }
 {
  \|
   p_\varepsilon ^{(i)}[\psi^2] 
  \|_{L^p(m_h)} ^{p-1} 
 }
$
and then
\begin{equation*}
 \liminf_{t\rightarrow \infty}
 \frac{1}{t}
 \log
 \widetilde{\mathbb{E}}_{t}
 \exp
 \Biggl\{
  \theta
  \biggl(
   \int_E
    {\ell_{t, \varepsilon} ^{(i)} }^p
   dm_h
  \biggr)^{1/p}
 \Biggr\}
\geq
 \sup_{\psi\in \mathcal{F} ^{(i)}, \|\psi\|_2=1}
 \Biggl\{
  \theta
  \biggl(
   \int_E
    {p_\varepsilon ^{(i)} [\psi^2]}^p
   dm_h
  \biggr)^{1/p}
 -
  \mathcal{E} ^{(i)} (\psi, \psi)
 +
  \lambda_1 ^{(i)} 
 \Biggr\}
.
\end{equation*}

\vspace{1eM}

We next 
prove the upper bound of \eqref{eq_prop62_1}.
Let $E' := \text{supp}[h]$.
Fix $\delta >0$ and
take 
$q>0$ with $p^{-1} + q^{-1} =1$.
For each 
$f \in C_K(E') (\subset L^q(E'; m_h))$ with 
$\|f\|_{L^q(m_h)} = 1$, set
$
 U_f
:=
 \left\{
  g\in L^p(E'; m_h)
 :
  \|g\|_{L^p(m_h)}
 -
  \int_E
   f g
  dm_h
 < 
  \delta
 \right\}
$.
We note that 
$K_\varepsilon ^{(i)}$
is $L^p(E'; m_h)$-bounded
and $C_K(E')$ is dense in $L^q(E'; m_h)$. 
By the Hahn-Banach theorem, the family of such $U_f$'s 
is an open cover of $K_\varepsilon ^{(i)}$.

Since $K^{(i)} _\varepsilon$ is compact 
in $L^p(E'; m_h)$,
there are finitely many $f_1, \cdots, f_N \in C_K(E')$ 
such that, for every $i$,
$
 \|g\|_{L^p(m_h)}
\leq
 \max_{1\leq i\leq N}
 \int_E
  f_i g
 dm_h
+
 \delta
$
for all
$
 g\in K^{(i)} _\varepsilon
$.
In particular,
\begin{equation*}
 \widetilde{\mathbb{E}}_{t}
 \Biggl[
  \exp
  \Biggl\{
   \theta t
   \biggl(
    \int_E
     {L_{t, \varepsilon}^{(i)}} ^p
    dm_h 
   \biggr)^{1/p}
  \Biggr\}
 \Biggr]
\leq
 e^{\delta t}
 \sum_{i=1} ^N
 \widetilde{\mathbb{E}}_{t}
 \exp
 \Bigl\{
  \theta t
  \langle
   f_i
  \cdot
   h
  ,
   L_{t, \varepsilon} ^{(i)}
  \rangle
 \Bigr\}
\end{equation*}
and by the equality \eqref{eq_Varadhan}, 
\begin{align*}
&
 \limsup_{t\rightarrow \infty}
 \frac{1}{t}
 \log
 \widetilde{\mathbb{E}}_{t}
 \Biggl[
  \exp
  \Biggl\{
   \theta t
   \biggl(
    \int_E
     {L_{t, \varepsilon}^{(i)}} ^p
    dm_h
   \biggr)^{1/p}
  \Biggr\}
 \Biggr]
\\
\leq&
 \delta
+
 \max_{1\leq i\leq N}
 \sup_{\psi\in \mathcal{F}^{(i)}, \|\psi\|_2=1}
 \biggl\{
  \theta
  \int_E
    f_i  p^{(i)} _\varepsilon[\psi^2] 
  dm_h
 -
  \mathcal{E}^{(i)}(\psi, \psi)
 +
  \lambda^{(i)}_1
 \biggr\}
\\
\leq&
 \delta
+
 \sup_{\psi\in \mathcal{F}^{(i)}, \|\psi\|_2=1}
 \biggl\{
  \theta
  \biggl(
   \int_E
     p^{(i)}_\varepsilon[\psi^2] ^p
   dm_h
  \biggr)^{1/p}
 -
  \mathcal{E}^{(i)}(\psi, \psi)
 +
  \lambda^{(i)}_1
 \biggr\}
.
\end{align*}
By taking $\delta\rightarrow 0$, 
we thus obtain \eqref{eq_prop62_1}.

\vspace{1eM}

We next 
prove the lower bound of \eqref{eq_prop62_2}.
Fix $\delta >0$ and 
let $d$ be a metric of $E$. 
We define the sets
$
 A_\varepsilon ^{(i)}
\subset 
 C(E')
\subset 
 L^p(E'; m_h) 
$
by
\begin{equation}
 A_\varepsilon ^{(i)}
:=
 \biggl\{
  \frac{1}{t}
  \int_0 ^t
   p^{(i)}_\varepsilon (\cdot, X^{(i)}_s(\omega))
  ds
 :
  t\in (0, \infty),
  \omega
 \in  
  \{
   t < \zeta^{(1)}\wedge \cdots \wedge \zeta^{(p)}
  \}
 \biggr\}
.
\end{equation}

\noindent
We can see that $A_\varepsilon ^{(i)}$ is 
uniformly bounded and equi-continuous.
By the Arzel\`a-Ascoli theorem, each $A_\varepsilon ^{(i)}$ is 
a relatively compact set of $C(E')$, and hence, 
a relatively compact set of $L^p(E'; m_h)$. 
Write $K_\varepsilon ^{(i)}$ 
as the $L^p(E'; m_h)$-closure of $A_\varepsilon ^{(i)}$.

Consider the continuous, non-negative function
\begin{equation}
\label{eq_6-1}
 (L^p(E'; m_h))^p
\ni
 (f_1, \cdots, f_p)
\longmapsto
 \frac{1}{p}
 \sum_{i=1} ^p
 \left(
  \int_{E'}
   |f_i|^p
  dm_h
 \right)^{1/p}
-
 \biggl(
  \int_{E'}
   \prod_{i=1} ^p
   |f_i|
  dm_h
 \biggr)^{1/p}
\in
 [0, \infty)
.
\end{equation}

\noindent
Since 
this function is equal to $0$ on the diagonal set
$\{f_1 = \cdots = f_p\} \subset (L^p(E'; m_h))^p$, 
for each $g\in L^p(E'; m_h)$
there exists a constant $b= b(g, \delta)>0$ such that
$
 \frac{1}{p}
 \sum_{i=1} ^p
 \left(
  \int_E
   |f_i|^p
  dm
 \right)^{1/p}
-
 \bigl(
  \int_E
   \prod_{i=1} ^p
   |f_i|
  dm
 \bigr)^{1/p}
<
 \delta
$
if $f_i \in B_b(g)$ for all $i$,
where 
$B_b(g)$ is the open ball in 
$L^p(E'; m_h)$ of radius $b$ centered at $g$.
Since 
$K_\varepsilon ^{(i)}$ is compact in $L^p(E'; m_h)$, 
take finite $g_1, \cdots g_N \in L^p(E'; m_h)$ and
$b_1, \cdots b_N >0$ such that
$
 K_\varepsilon ^{(i)}
\subset
 \bigcup_{l=1} ^N
 B_{b_l} (g_l)
$
for all $i$.

Now,
$
 t^{-1}
 \ell_{t, \varepsilon} ^{(i)}
=
 L_{t, \varepsilon} ^{(i)}
$ 
belongs to some $B_{b_l} (g_l)$
and hence
\begin{align*}
 \sum_{l=1} ^N
 \widetilde{\mathbb{E}}_{t}
 \Biggl[
  \exp
  \Biggl\{
   \frac{\theta}{p} t
   \biggl(
    \int_{E'}
     {L^{(i)} _{t, \varepsilon} }^p
    dm_h
   \biggr)^{1/p}
  \Biggr\}
 ;
  {L^{(i)} _{t, \varepsilon} }^p
 \in
  B_{b_l} (g_l)
 \Biggr]
\geq
 \widetilde{\mathbb{E}}_{t}
 \exp
 \Biggl\{
  \frac{\theta}{p} t
  \biggl(
   \int_{E'}
    {L^{(i)} _{t, \varepsilon} }^p
   dm_h
  \biggr)^{1/p}
 \Biggr\}
.
\end{align*}
For each $l$, 
the continuity of the function 
\eqref{eq_6-1} implies that
\begin{align*}
&
 \widetilde{\mathbb{E}}_{t}
 \exp
 \Biggl\{
  \theta t
  \biggl(
   \int_{E'}
    \prod_{i=1} ^p
    L_{t, \varepsilon} ^{(i)}
   dm_h
  \biggr)^{1/p}
 \Biggr\}
\\
\geq&
 e^{-\delta t}
 \widetilde{\mathbb{E}}_{t}
 \Biggl[
  \exp
  \biggl\{
   \frac{\theta}{p} t
   \sum_{i=1} ^p
   \biggl(
    \int_{E'}
     {L_{t, \varepsilon} ^{(i)} }^p
    dm_h
   \biggr)^{1/p}
  \biggr\}
 ;
  L_{t, \varepsilon} ^{(i)}
 \in
  B_{b_l} (g_l)
 \text{ for all }i
 \Biggr]
\\
=&
 e^{-\delta t}
 \prod_{i=1} ^p
  \widetilde{\mathbb{E}}_{t}
  \Biggl[
   \exp
   \Biggl\{
    \frac{\theta}{p} t
    \biggl(
     \int_{E'}
      {L_{t, \varepsilon} ^{(i)} }^p
     dm_h
    \biggr)^{1/p}
   \Biggr\}
  ;
   L_{t, \varepsilon} ^{(i)}
  \in
   B_{b_l}(g_l)
  \Biggr]
.
\end{align*}

\noindent
Combining the above two inequalities with 
\eqref{eq_prop62_1},
we conclude
\begin{align*}
&
 \liminf_{t\rightarrow \infty}
 \frac{1}{t}
 \log
 \widetilde{\mathbb{E}}_{t}
 \exp
 \Biggl\{
  \theta t
  \biggl(
   \int_{E}
    \prod_{i=1} ^p
    L_{t, \varepsilon} ^{(i)} 
   dm_h
  \biggr)^{1/p}
 \Biggr\}
\\
\geq&
 -\delta
+
 \sum_{i=1} ^p
 \liminf_{t\rightarrow\infty}
 \frac{1}{t}
 \log
 \widetilde{\mathbb{E}}_{x_0}
 \exp
 \Biggl\{
  \frac{\theta}{p} t
  \biggl(
   \int_E
    {L_{t, \varepsilon}^{(i)} }^p
   dm_h
  \biggr)^{1/p}
 \Biggr\}
\\
=&
 -\delta
+
 \frac{1}{p}
 \sum_{i=1} ^p
 \sup_{\psi\in \mathcal{F}^{(i)}, \|\psi\|_{L^2(m)} = 1}
 \Biggl\{
  \theta
  \biggl(
   \int_E
    |p^{(i)}_\varepsilon[\psi^2] |^p
   dm_h
  \biggl)^{1/p}
 -
  p \mathcal{E}^{(i)}(\psi, \psi)
 +
  p \lambda^{(i)}_1
 \Biggr\}
.
\end{align*}

\vspace{1eM}

Finally,
we prove the upper bound of \eqref{eq_prop62_2}.
Since
\begin{equation*}
 \Biggl\{
  \int_E
   \prod_{i=1} ^p
   L_{t, \varepsilon} ^{(i)}
  dm_h
 \Biggr\}^{1/p}
\leq
 \Biggl\{
 \prod_{i=1} ^p
 \biggl(
  \int_E
   {L_{t, \varepsilon} ^{(i)} }^p
  dm_h
 \biggr) ^{1/p}
 \Biggr\}^{1/p}
\leq
 \sum_{i=1} ^p
 \frac{1}{p}
 \biggl(
  \int_E
   {L_{t, \varepsilon} ^{(i)} }^p
  dm_h
 \biggl)^{1/p}
,
\end{equation*}
we have from the upper bound of 1st equality,
\begin{align*}
&
 \limsup_{t\rightarrow \infty}
 \frac{1}{t}
 \log
 \widetilde{\mathbb{E}}_{t}
 \Biggl[
  \exp
  \Biggl\{
   \theta t
   \biggl(
    \int_E
     \prod_{i=1} ^p
     L_{t, \varepsilon} ^{(i)} 
    dm_h
   \biggl)^{1/p}
  \Biggr\}
 \Biggr]
\\
\leq&
 \sum_{i=1} ^p
 \limsup_{t\rightarrow \infty}
 \frac{1}{t}
 \log
 \widetilde{\mathbb{E}}_{t}
 \Biggl[
  \exp
  \Biggl\{
   \frac{\theta}{p} t
   \biggl(
    \int_E
     {L_{t, \varepsilon} ^{(i)} }^p
    dm_h
   \biggr)^{1/p}
  \Biggr\}
 \Biggr]
\\
\leq&
 \frac{1}{p}
 \sum_{i=1} ^p
 \sup_{\psi\in \mathcal{F}^{(i)}, \|\psi\|_2=1}
 \Biggl\{
  \theta
  \biggl(
   \int_E
     p^{(i)}_\varepsilon[\psi^2] ^p
   dm_h
  \biggr)^{1/p}
 -
  p \mathcal{E}^{(i)}(\psi, \psi)
 +
  p \lambda^{(i)}_1
 \Biggr\}
.
\end{align*}
\end{proof}

\subsection{Exponential approximation}

Recall that
$
 t^p
 \int_E
  \prod_{i=1} ^p
  L_{t, \varepsilon} ^{(i)}
 dm_h
=
 \langle
  \ell_{t, \varepsilon} ^{\mathrm{IS}}
 ,
  h
 \rangle
$
and
$
 \lambda_1 ^{(i)}
:=
 \inf
 \left\{
  \mathcal{E}^{(i)} (\psi,\psi)
 :
  \psi \in \mathcal{F}^{(i)}
  ,
  \int_E
   \psi^2
  dm
 =
  1
 \right\}
$.

Next, we prove
a similar result with \cite[Theorem 6]{MR2165257}.
As 
we have already noted in the end of 
Section \ref{Sec_Main_results},
Chen and Rosen
consider the stable processes on $\mathbb{R}^N$,
so the Fourier transformation method is valid.
Instead, 
we use Proposition \ref{Prop_exp_app}.

%

\begin{Lem}
\label{Lem_6-3}
For any $\theta>0$, it holds that
\begin{align*}
 \limsup_{t\rightarrow \infty}
 \frac{1}{t}
 \log
 \widetilde{\mathbb{E}}_{x_0}
 \bigl[
  \exp
  \{
   \theta
   |
   \langle
    \ell_{t, \varepsilon} ^{\mathrm{IS}}
   ,
    h
   \rangle
  -
   \langle
    \ell_{t} ^{\mathrm{IS}}
   ,
    h
   \rangle
   |^{1/p}
  \}
 \bigr]
\leq
 1
+
 \frac{\sum_{i=1} ^p \lambda_1 ^{(i)}}{p}
.
\end{align*}
\end{Lem}

\begin{proof}
Define $q$ by $1/p + 1/q = 1$.
By H\"older's inequality and Proposition \ref{Prop_exp_app}, we get
\begin{align*}
&
 \mathbb{E}_{x_0}
 \bigl[
  |
   \langle
    \ell_{t, \varepsilon} ^{\mathrm{IS}}
   ,
    h
   \rangle
  -
   \langle
    \ell_{t} ^{\mathrm{IS}}
   ,
    h
   \rangle
  |^{k/p}
 ;
  t< \zeta^{(1)}\wedge \cdots \wedge \zeta^{(p)}
 \biggr]
\\
\leq&
 \mathbb{E}_{x_0}
 \bigl[
  |
   \langle
    \ell_{t, \varepsilon} ^{\mathrm{IS}}
   ,
    h
   \rangle
  -
   \langle
    \ell_{t} ^{\mathrm{IS}}
   ,
    h
   \rangle
  |^{k}
 ;
  t< \zeta^{(1)}\wedge \cdots \wedge \zeta^{(p)}
 \bigr]^{1/p}
 \mathbb{P}_{x_0}
 (
  t< \zeta^{(1)}\wedge \cdots \wedge \zeta^{(p)}
 )^{1/q}
\\
\leq&
 k! 
 e^{t} 
 C(\varepsilon)^{k/p}
 \mathbb{P}_{x_0}
 (
  t< \zeta^{(1)}\wedge \cdots \wedge \zeta^{(p)}
 )^{1/q}
\end{align*}

\noindent
and hence, for small $\varepsilon > 0$ 
(which depends on $\theta$), 
we have
\begin{align*}
&
 \mathbb{E}_{x_0}
 \bigl[
  \exp
  \{
   \theta
   |
   \langle
    \ell_{t, \varepsilon} ^{\mathrm{IS}}
   ,
    h
   \rangle
  -
   \langle
    \ell_{t} ^{\mathrm{IS}}
   ,
    h
   \rangle
   |^{1/p}
  \}
 ;
  t< \zeta^{(1)}\wedge \cdots \wedge \zeta^{(p)}
 \bigr]
\\
=&
 \sum_{k=0} ^\infty
 \frac{\theta^{k}}{k!}
 \mathbb{E}_{x_0}
 \bigl[
  |
   \langle
    \ell_{t, \varepsilon} ^{\mathrm{IS}}
   ,
    h
   \rangle
  -
   \langle
    \ell_{t} ^{\mathrm{IS}}
   ,
    h
   \rangle
  |^{k/p}
 ;
  t< \zeta^{(1)}\wedge \cdots \wedge \zeta^{(p)}
 \bigr]
\\
\leq&
 e^{t} 
 \sum_{k=0} ^\infty
 \theta^k
 C(\varepsilon)^{k/p}
 \mathbb{P}_{x_0}
 (
  t< \zeta^{(1)}\wedge \cdots \wedge \zeta^{(p)}
 )^{1/q}
\\
=&
 e^{t} 
 \bigl\{
  1
 -
  \theta
  C(\varepsilon)^{1/p}
 \bigr\}^{-1}
 \mathbb{P}_{x_0}
 (
  t< \zeta^{(1)}\wedge \cdots \wedge \zeta^{(p)}
 )^{1/q}
.
\end{align*}

\noindent
By normalization, we rewrite the above inequality as
\begin{align*}
&
 \widetilde{\mathbb{E}}_{t}
 \bigl[
  \exp
  \{
   \theta
   |
   \langle
    \ell_{t, \varepsilon} ^{\mathrm{IS}}
   ,
    h
   \rangle
  -
   \langle
    \ell_{t} ^{\mathrm{IS}}
   ,
    h
   \rangle
   |^{1/p}
  \}
 \bigr]
\leq
 e^{t} 
 \bigl\{
  1
 -
  \theta
  C(\varepsilon)^{1/p}
 \bigr\}^{-1}
 \mathbb{P}_{x_0}
 (
  t< \zeta^{(1)}\wedge \cdots \wedge \zeta^{(p)}
 )^{- 1/p}
.
\end{align*}

\noindent
By \cite[Corollary 6.4.2]{MR2778606}, 
we have the equality
$
 \lim_{t\rightarrow\infty}
 t^{-1}
 \log
 \mathbb{P}_{x_0}
 (
  t< \zeta^{(i)}
 )
=
 -\lambda_1 ^{(i)}
$.
Since $\zeta^{(1)}, \cdots, \zeta^{(p)}$ are independent,
we reach the conclusion.
\end{proof}


\vspace{1eM}

Fix 
a nonnegative and compactly supported function
$h\in \mathcal{B}_b(E)$.
For $\theta \geq0$ and $\varepsilon>0$, 
define 
\begin{align*}
 N^{(i)}(\theta, \varepsilon, h) 
:=&
 \sup_{\psi\in\mathcal{F}^{(i)}, \|\psi\|_2 =1}
 \Biggl\{
  \theta 
  \biggl(
   \int_E
    p_\varepsilon^{(i)} 
    [
     \psi^{2}
    ]^p
   dm_h
  \biggr)^{1/p}
 -
  p\mathcal{E}^{(i)}(\psi, \psi)
 +
  p\lambda^{(i)}_1
 \Biggr\}
,
\\
 N^{(i)}(\theta, 0, h) 
:=&
 \sup_{\psi\in\mathcal{F}^{(i)}, \|\psi\|_2 =1}
 \Biggl\{
  \theta 
  \biggl(
   \int_E
    \psi^{2p}
   dm_h
  \biggr)^{1/p}
 -
  p\mathcal{E}^{(i)}(\psi, \psi)
 +
  p\lambda^{(i)}_1
 \Biggr\}
.
\end{align*}

The
next Lemma is used in the proof of the
lower bound of Proposition \ref{Prop_mgf} below.
\begin{Lem}
\label{Lem_66}
For fixed $\theta\geq 0$, it holds that
\begin{align}
 \limsup_{\varepsilon \rightarrow 0}
 N^{(i)}(\theta, \varepsilon, h)
\leq&
 N^{(i)}(\theta, 0, h)
,
\\
 \limsup_{\theta' \rightarrow \theta}
 N^{(i)}(\theta', 0, h)
\leq&
 N^{(i)}(\theta, 0, h)
.
\end{align}
\end{Lem}

\begin{proof}
For convenience, 
we fix $i$ and omit the index $^{(i)}$
only in this proof.

Fix $\theta>0$ and $\delta >0$.
Take $L^2(E; m)$-normalized 
$\psi^{(\delta, \varepsilon)} \in \mathcal{F}$ 
such that
\begin{equation*}
 N(\theta, \varepsilon, h) 
-
 \delta
<
  \theta 
  \Bigl(
   \int_E
    p_\varepsilon
    [
     {\psi^{(\delta, \varepsilon)}}^{2}
    ]^p
   dm_h
  \Bigr)^{1/p}
 -
  p
  \mathcal{E}
  (
   \psi^{(\delta, \varepsilon)}
  ,
   \psi^{(\delta, \varepsilon)}
  )
 +
  p\lambda_1
.
\end{equation*}

\noindent
We have
$
 0
=
 N(0, \varepsilon, h) 
\leq
 N(\theta, \varepsilon, h)
$,
and by \eqref{eq_Sobolev}, we also have
\begin{equation*}
 \biggl(
  \int_E
   p_\varepsilon
   \bigl[
    {\psi^{(\delta, \varepsilon)}}^{2}
   \bigr]^p
  dm_h
 \biggr)^{1/p}
\leq
 \|
  \psi^{(\varepsilon, \delta)}
 \|_{L^{2p}(m)} 
\leq
 C_\delta
+
 \delta
 \mathcal{E}
 (
  \psi^{(\delta, \varepsilon)}
 ,
  \psi^{(\delta, \varepsilon)}
 )
\end{equation*}
and hence
$
 (
  p
 -
  \theta \delta
 )
 \mathcal{E}
 (
  \psi^{(\delta, \varepsilon)}
 ,
  \psi^{(\delta, \varepsilon)}
 )
<
 \theta 
 C_\delta
+
 p\lambda_1
+
 \delta
$.
By choosing a small $\delta >0$, 
$\{\psi^{(\delta, \varepsilon)}\}$ is
bounded in $(\mathcal{F}, \mathcal{E}_1)$.
By \eqref{eq_compact2p}, we can take  
a subsequence such that
$
  \psi^{(\delta, \varepsilon)}
$
converges to some
$
 \psi^{(\delta)}
$
as $\varepsilon \rightarrow 0$,
$
 \text{ in } L^2(E; m)
 \text{ and in } L^{2p}(E; m)
$.
Then we have
\begin{align*}
&
 N(\theta, \varepsilon, h)
-
 N(\theta, 0, h) 
\\
\leq&
 \theta
 \|
  p_\varepsilon
  [
   {\psi^{(\delta, \varepsilon)}}^2
  ]
 \|_{L^p(m_h)}
-
 \theta
 \|
  {\psi^{(\delta, \varepsilon)}}^2
 \|_{L^p(m_h)}
+
 \delta
\\
\leq&
 \theta
 \|
  p_\varepsilon
  [
   {\psi^{(\delta, \varepsilon)}}^2
  -
   {\psi^{(\delta)}}^2
  ]
 \|_{L^p(m_h)}
+
 \theta
 \|
  p_\varepsilon
  [
   {\psi^{(\delta)}}^2
  ]
 \|_{L^p(m_h)}
-
 \theta
 \|
  {\psi^{(\delta, \varepsilon)}}^2
 \|_{L^p(m_h)}
+
 \delta
\\
\leq&
 \theta
 \|
  h
 \|_{\infty}
 \|
  {\psi^{(\delta, \varepsilon)}}^2
 -
  {\psi^{(\delta, \varepsilon)}}^2
 \|_{L^p(m)}
+
 \theta
 \|
  p_\varepsilon
  [
   {\psi^{(\delta)}}^2
  ]
 \|_{L^p(m_h)}
-
 \theta
 \|
  {\psi^{(\delta)}}^2
 \|_{L^p(m_h)}
+
 \delta
.
\end{align*}
Hence
$
 \limsup_{\varepsilon\rightarrow 0}
 N(\theta, \varepsilon, h)
\leq
 N(\theta, 0, h)
$.

To prove the second inequality, 
fix $\theta'\geq 0$.
For each $\delta > 0$, 
take $\psi^{(\delta, \theta')} \in \mathcal{F}$ 
such that $\|\psi^{(\delta, \theta')} \|_2=1$ and
$
 N(\theta', 0, h)
<
 \theta'
 \|
  \psi^{(\delta, \theta')}
 \|_{L^{2p}(m_h)} ^2
-
 p\mathcal{E}(\psi^{(\delta, \theta')}, \psi^{(\delta, \theta')})
+
 p\lambda_1
+
 \delta
$.
By the same manner as the proof of the first inequality,
we can find that 
$
 (
  p
 -
  \theta' \delta
 )
 \mathcal{E}
 (
  \psi^{(\delta, \theta')}
 ,
  \psi^{(\delta, \theta')}
 )
<
 \theta'
 C_\delta
+
 p\lambda_1
+
 \delta
$.
We have
\begin{align*}
&
 N(\theta', 0, h)
-
 N(\theta , 0, h) 
\leq
 (\theta' - \theta) 
 \|
  {\psi^{(\delta, \theta')}}
 \|_{L^{2p}(m_h)} ^2
+
 \delta
\leq
 (\theta' - \theta) 
 (
  C_\delta
 +
  \delta
  \mathcal{E}
  (
   \psi^{(\delta, \theta')}
  ,
   \psi^{(\delta, \theta')}
  )
 )
+
 \delta
.
\end{align*}
Hence
$
 \limsup_{\theta' \rightarrow \theta}
 N(\theta', 0, h)
\leq
 N(\theta, 0, h)
$.
\end{proof}


\vspace{1eM}

Finally we prove Proposition \ref{Prop_mgf}.
Our proof is based on the proof of 
\cite[Theorem 1]{MR2165257}.

\begin{proof}
[Proof of Proposition \ref{Prop_mgf}]
We first prove the upper bound.
Fix $a, b>0$ with $a^{-1} + b^{-1} =1$.
By H\"older's inequality,
\begin{equation*}
 \widetilde{\mathbb{E}}_{t}
 \exp
 \{
  \theta a^{-1}
  \langle
   \ell_{t, \varepsilon} ^{\mathrm{IS}}
  ,
   h
  \rangle^{1/p} 
 \}
\leq
 \Bigl(
  \widetilde{\mathbb{E}}_{t}
  \exp
  \{
   \theta 
   \langle
    \ell_{t} ^{\mathrm{IS}}
   ,
    h
   \rangle^{1/p} 
  \}
 \Bigr)^{1/a}
 \Bigl(
  \widetilde{\mathbb{E}}_{t}
  \exp
  \{
   \theta a^{-1} b
   |
    \langle
     \ell_{t,\varepsilon} ^{\mathrm{IS}}
    ,
     h
    \rangle 
   -
    \langle
     \ell_{t} ^{\mathrm{IS}}
    ,
     h
    \rangle 
   |^{1/p} 
  \}
 \Bigr)^{1/b}
\end{equation*}
and by Proposition \ref{Prop_6-2},
\begin{align*}
&
 \frac{1}{p}
 \sum_{i=1} ^p
 \sup_{\psi\in \mathcal{F}^{(i)}, \|\psi\|_2=1}
 \Bigl\{
  \theta a^{-1}
  \|p^{(i)}_\varepsilon[\psi^2]\|_{L^{p}(m_h)}
 -
  p\mathcal{E}^{(i)}(\psi, \psi)
 +
  p\lambda^{(i)}_1
 \Bigr\}
\\
\leq&
 \liminf_{t\rightarrow\infty}
 \frac{1}{at}
 \log
 \widetilde{\mathbb{E}}_{t}
 \exp
 \{
  \theta 
  \langle
   \ell_{t} ^{\mathrm{IS}}
  ,
   h
  \rangle^{1/p} 
 \}
+
 \limsup_{t\rightarrow\infty}
 \frac{1}{bt}
 \log
 \widetilde{\mathbb{E}}_{t}
 \exp
 \{
  \theta a^{-1} b
  |
   \langle
    \ell_{t} ^{\mathrm{IS}}
   ,
    h
   \rangle
  -
   \langle
    \ell_{t, \varepsilon} ^{\mathrm{IS}}
   ,
    h
   \rangle
  |^{1/p} 
 \}
.
\end{align*}

\noindent
We can see that
$
 p^{(i)} _\varepsilon[\psi^2] 
\rightarrow 
 \psi^2
$ 
in $L^p$ for each $i$ and 
$\psi \in \mathcal{F}^{(i)}$.
By taking $\limsup_{\varepsilon \rightarrow 0}$,
Lemma \ref{Lem_6-3} implies that 
\begin{align*}
&
 \frac{1}{p}
 \sum_{i=1} ^p
 \sup_{\psi\in \mathcal{F}^{(i)}, \|\psi\|_2=1}
 \Bigl\{
  \theta a^{-1}
  \|p^{(i)}_\varepsilon[\psi^2]\|_{L^{p}(m_h)}
 -
  p\mathcal{E}^{(i)}(\psi, \psi)
 +
  p\lambda^{(i)}_1
 \Bigr\}
\\
\leq&
 \liminf_{t\rightarrow\infty}
 \frac{1}{at}
 \log
 \widetilde{\mathbb{E}}_{t}
 \exp
 \{
  \theta 
  \langle
   \ell_{t} ^{\mathrm{IS}}
  ,
   h
  \rangle^{1/p} 
 \}
+
 \frac{1}{bp}
 \cdot
 \sum_{i=1} ^p
 \Bigl(
  1
 +
  \lambda^{(i)} _1
 \Bigr)
.
\end{align*}

\noindent
Letting $a\rightarrow 1$, $b\rightarrow \infty$ 
and hence
\begin{align*}
&
 \frac{1}{p}
 \sum_{i=1} ^p
 \sup_{\psi\in \mathcal{F}^{(i)}, \|\psi\|_2=1}
 \Bigl\{
  \theta 
  \|\psi\|_{L^{2p}(m_h)} ^2
 -
  p\mathcal{E}^{(i)}(\psi, \psi)
 +
  p\lambda^{(i)}_1
 \Bigr\}
\leq
 \liminf_{t\rightarrow\infty}
 \frac{1}{t}
 \log
 \widetilde{\mathbb{E}}_{t}
 \exp
 \{
  \theta 
  \langle
   \ell_{t} ^{\mathrm{IS}}
  ,
   h
  \rangle^{1/p} 
 \}
.
\end{align*}

\vspace{1eM}

We next prove the lower bound.
Fix $a, b>0$ with $a^{-1} + b^{-1} =1$.
By H\"older's inequality,
\begin{equation*}
 \widetilde{\mathbb{E}}_{t}
 \exp
 \{
  \theta
  \langle
   \ell_{t} ^{\mathrm{IS}}
  ,
   h
  \rangle^{1/p} 
 \}
\leq
 \Bigl(
  \widetilde{\mathbb{E}}_{t}
  \exp
  \{
   \theta a 
   \langle
    \ell_{t, \varepsilon} ^{\mathrm{IS}}
   ,
    h
   \rangle^{1/p} 
  \}
 \Bigr)^{1/a}
 \Bigl(
  \widetilde{\mathbb{E}}_{t}
  \exp
  \{
   \theta b
   |
    \langle
     \ell_{t, \varepsilon} ^{\mathrm{IS}}
    ,
     h
    \rangle 
   -
    \langle
     \ell_{t} ^{\mathrm{IS}}
    ,
     h
    \rangle 
   |^{1/p} 
  \}
 \Bigr)^{1/b}
\end{equation*}
and hence
\begin{align*}
&
 \limsup_{t\rightarrow\infty}
 \frac{1}{t}
 \log
 \widetilde{\mathbb{E}}_{t}
 \exp
 \{
  \theta 
  \langle
   \ell_{t} ^{\mathrm{IS}}
  ,
   h
  \rangle^{1/p} 
 \}
\\
\begin{split}
\leq&
 \frac{1}{ap}
 \sum_{i=1} ^p
 \sup_{\psi\in \mathcal{F}^{(i)}, \|\psi\|_2=1}
 \Bigl\{
  \theta a
  \|p^{(i)} _\varepsilon[\psi^2]\|_{L^p(m_h)}
 -
  p\mathcal{E}^{(i)} (\psi, \psi)
 +
  p\lambda^{(i)} _1
 \Bigl\}
\\
&
\hspace{40mm}
+
 \limsup_{t\rightarrow\infty}
 \frac{1}{bt}
 \log
 \widetilde{\mathbb{E}}_{t}
 \exp
 \{
  \theta  b
  |
   \langle
    \ell_{t, \varepsilon} ^{\mathrm{IS}}
   ,
    h
   \rangle 
  -
   \langle
    \ell_{t} ^{\mathrm{IS}}
   ,
    h
   \rangle 
  |^{1/p} 
 \}
\end{split}
.
\end{align*}

\noindent
Take 
$\limsup_{\varepsilon \rightarrow 0}$ and
$\limsup_{a\rightarrow 1}$, successively.
Then 
Lemma \ref{Lem_6-3} and Lemma \ref{Lem_66} imply
\begin{align*}
&
 \limsup_{t\rightarrow\infty}
 \frac{1}{t}
 \log
 \widetilde{\mathbb{E}}_{t}
 \exp
 \{
  \theta 
  \langle
   \ell_{t} ^{\mathrm{IS}}
  ,
   h
  \rangle^{1/p} 
 \}
\\
\leq&
 \limsup_{a\rightarrow 1}
 \frac{1}{ap}
 \sum_{i=1} ^p
 \sup_{\psi\in \mathcal{F}^{(i)}, \|\psi\|_2=1}
 \Bigl\{
  \theta a^{-1}
  \|\psi\|_{L^{2p}(m_h)} ^2
 -
  p\mathcal{E}^{(i)}(\psi, \psi)
 +
  p\lambda^{(i)}_1
 \Bigr\}
+
 \limsup_{a\rightarrow 1}
 \frac{1}{bp}
 \sum_{i=1} ^p
 \Bigl(
  1
 +
  \lambda^{(i)} _1
 \Bigr)
\\
&
 \leq
 \frac{1}{p}
 \sum_{i=1} ^p
 \sup_{\psi\in \mathcal{F}^{(i)}, \|\psi\|_2=1}
 \Bigl\{
  \theta
  \|\psi\|_{L^{2p}(m_h)} ^2
 -
  p\mathcal{E}^{(i)}(\psi, \psi)
 +
  p\lambda^{(i)}_1
 \Bigr\}
.
\end{align*}
\end{proof}


\section{Examples}
\label{Sec_7}

In this section, we give sufficient conditions for 
Assumption \textbf{(A)}.

\subsection{Diffusion processes with (sub-)Gaussian type estimates} 

Let
$(E, d, m)$ be a locally compact, separable 
metric measure space
with $\text{diam}(E) =1$ and $m(E)<\infty$.
Let
$X$ be a Hunt process on $E$ 
with the transition density $p_t(\cdot, \cdot)$.
Suppose there exist positive constants
$c_1, \cdots, c_5$, $d_{\rm f}\geq 1$ and $d_{\rm w}\geq 2$
such that
\begin{equation}
\label{eq_7-1}
 c_1
 t^{-d_{\rm f} / d_{\rm w}}
\leq
 p_t(x, x)
,
\end{equation}
\begin{equation}
\label{eq_7-2}
 \quad
 p_t (x, y)
\leq
 c_2
 t^{-d_{\rm f} / d_{\rm w}}
 \exp
 \left\{
  -c_3
  \left(
   \frac{d(x, y)^{d_{\rm w}}}{t}
  \right)^{1/(d_{\rm w} -1)}
 \right\}
\end{equation}
and
\begin{equation}
\label{eq_7-3}
 c_4 
 r^{d_{\rm f}}
\leq
 m( B(x, r) )
\leq
 c_5
 r^{d_{\rm f}}
\end{equation}
for all $x, y\in E$, $r\in (0, 1]$
and $t\in (0, 1]$.

\begin{Prop}
\label{Prop_Gaussian}
Under the above conditions,
if an integer $p$ with $p\geq 2$ satisfies
\begin{equation*}
 d_{\rm f} - p (d_{\rm f} - d_{\rm w}) > 0
,
\end{equation*}
then Assumption \textbf{(A)} holds.
\end{Prop}

\begin{proof}
Write
$
 d_{\rm s}
:=
 2 d_{\rm f} / d_{\rm w}
$.
We have by \eqref{eq_7-1} and \eqref{eq_7-2},
$
 c_1 t^{-d_{\rm s}/ 2}
\leq
 p_t(x, x)
\leq
 c_2 t^{-d_{\rm s}/ 2}
$
for all $x\in E$ and $t\in (0, 1]$,
hence (A3) follows because $m(E)<\infty$.

Next, we have by \eqref{eq_7-2},
$
 p_t(x, y)
\leq
 c_2 t^{-d_{\rm s}/ 2}
$
for all $x, y\in E$ and $t\in (0, 1]$.
When 
$d_{\rm s} >2$, we have $p< d_{\rm s}/ (d_{\rm s} - 2)$
and (A4) holds for $\mu = d_{\rm s}$.
When
$d_{\rm s} \leq 2$, we have
$
 p 
<
 (2 + \varepsilon) / \{(2 + \varepsilon) - 2\}
=
 (2 +\varepsilon) / \varepsilon
$
for sufficiently small $\varepsilon>0$
and
$
 p_t(x, y)
\leq
 c_2
 t^{-d_{\rm s} / 2}
\leq
 c_2
 t^{-(2+\varepsilon)/2}
$
for all $x, y\in E$ and $t\in (0, 1]$.
Hence (A4) holds for $\mu = 2+\varepsilon$.

To
prove \eqref{eq_Green1} in (A5), fix $x, y\in E$.
We have
\begin{align*}
 R_1 (x, y)
=&
 \int_0 ^\infty
  e^{-t} p_t(x, y)
 dt
\\
\leq&
 \int_0 ^1
  e^{-t} 
  \Biggl[
   c_2
   t^{-d_{\rm f} / d_{\rm w}}
   \exp
   \Biggl\{
    -c_3
    \biggl(
     \frac{d(x, y)^{d_{\rm w}}}{t}
    \biggr)^{1/(d_{\rm w} -1)}
   \Biggr\}
  \Biggr]
 dt
+
 \int_1 ^\infty
  e^{-t}
  [
   c_2
   t^{-d_{\rm s}/2}
  ]
 dt
\\
\leq&
 \int_0 ^1
  e^{-t} 
  \Biggl[
   c_2
   t^{-d_{\rm f} / d_{\rm w}}
   \exp
   \Biggl\{
    -c_3
    \biggl(
     \frac{d(x, y)^{d_{\rm w}}}{t}
    \biggr)^{1/(d_{\rm w} -1)}
   \Biggr\}
  \Biggr]
 dt
+
 c_2
.
\end{align*}

\noindent
Since
$
 m( B(y; 2^{-n+1}) )
-
 m( B(y; 2^{-n}) )
\leq
 c_5 2^{-(n-1)d_{\rm f}}
-
 c_4 2^{-n d_{\rm f}} 
\leq
 C
 2^{-n d_{\rm f}} 
$
for each $n$, we also have
\begin{align*}
&
 \int_E
  \Biggl(
   \int_0 ^1
    e^{-t} 
    \Biggl[
     c_2
     t^{-d_{\rm f} / d_{\rm w}}
     \exp
     \Biggl\{
      -c_3
      \biggl(
       \frac{d(x, y)^{d_{\rm w}}}{t}
      \biggr)^{1/(d_{\rm w} -1)}
     \Biggr\}
    \Biggr]
   dt
  \Biggr)^p
 m(dx)
\\
=&
 \sum_{n=1} ^\infty
 \int_{2^{-n} < d(x, y) \leq 2^{-n+1}}
  \Biggl(
   \int_0 ^1
    e^{-t} 
    \Biggl[
     c_2
     t^{-d_{\rm f} / d_{\rm w}}
     \exp
     \Biggl\{
      -c_3
      \biggl(
       \frac{d(x, y)^{d_{\rm w}}}{t}
      \biggr)^{1/(d_{\rm w} -1)}
     \Biggr\}
    \Biggr]
   dt
  \Biggr)^p
 m(dx)
\\
\leq&
 \sum_{n=1} ^\infty
 \{ 
  m( B(y; 2^{-n+1}) )
 -
  m( B(y; 2^{-n}) )
 \} 
 \Biggl(
  \int_0 ^1
   e^{-t} 
   \Biggl[
    c_2
    t^{-d_{\rm f} / d_{\rm w}}
    \exp
    \Biggl\{
     -c_3
     \biggl(
      \frac{2^{-n d_{\rm w}}}{t}
     \biggr)^{1/(d_{\rm w} -1)}
    \Biggr\}
   \Biggr]
  dt
 \Biggr)^p
\\
\leq&
 C
 \sum_{n=1} ^\infty
 2^{-n d_{\rm f}} 
 \biggl(
  \int_0 ^1
   t^{-\frac{d_{\rm f}}{d_{\rm w}}}
   \exp
   \Bigl\{
    -c_3
    2^{-n \frac{d_{\rm w}}{d_{\rm w} -1}}
    t^{-\frac{1}{d_{\rm w} - 1}}
   \Bigr\}
  dt
 \biggr)^p
.
\end{align*}

\noindent
Now, since
$
 d_{\rm f} - p (d_{\rm f} - d_{\rm w}) > 0
$,
standard calculations imply 
\begin{equation}
\label{eq_7-4}
 \sum_{n=1} ^\infty
 2^{-n d_{\rm f}} 
 \biggl(
  \int_0 ^1
   t^{-\frac{d_{\rm f}}{d_{\rm w}}}
   \exp
   \Bigl\{
    -c_3
    2^{-n \frac{d_{\rm w}}{d_{\rm w} -1}}
    t^{-\frac{1}{d_{\rm w} - 1}}
   \Bigr\}
  dt
 \biggl)^p
<
 \infty
.
\end{equation}
Consequently, \eqref{eq_Green1} in (A5) holds. 

To 
see \eqref{eq_Green2}, 
we can similarly find that for $\delta<1$,
\begin{align*}
&
 \int_E
  \Biggl(
   \int_0 ^\delta
    e^{-t} 
    \Biggl[
     c_2
     t^{-d_{\rm f} / d_{\rm w}}
     \exp
     \Biggl\{
      -c_3
      \biggl(
       \frac{d(x, y)^{d_{\rm w}}}{t}
      \biggr)^{1/(d_{\rm w} -1)}
     \Biggr\}
    \Biggr]
   dt
  \Biggr)^p
 m(dx)
\\
\leq&
 C
 \sum_{n=1} ^\infty
 2^{-n d_{\rm f}} 
 \biggl(
  \int_0 ^\delta
   t^{-\frac{d_{\rm f}}{d_{\rm w}}}
   \exp
   \Bigl\{
    -c_3
    2^{-n \frac{d_{\rm w}}{d_{\rm w} -1}}
    t^{-\frac{1}{d_{\rm w} - 1}}
   \Bigr\}
  dt
 \biggr)^p
.
\end{align*}
By
regarding $\sum_{n}$ as the integral
with respect to the counting measure on $\mathbb{Z}_{>0}$,
\eqref{eq_Green2} follows by \eqref{eq_7-4} and 
the dominated convergence theorem.
\end{proof}

\subsection{Jump-type processes}

Let
$(E, d, m)$ be a locally compact, separable 
metric measure space
with $\text{diam}(E) =1$ and $m(E)<\infty$.
Let
$X$ be a Hunt process on $E$ 
with the transition density $p_t(\cdot, \cdot)$.
Suppose there exist positive constants
$c_1, \cdots, c_4$, $d_{\rm f}\geq 1$ and $d_{\rm w}\geq 2$
such that
\begin{equation}
 c_1
 t^{-d_{\rm f} / d_{\rm w}}
\leq
 p_t(x, x)
,
\end{equation}
\begin{equation}
 p_t (x, y)
\leq
 c_2
 \left\{
  t^{-d_{\rm f} / d_{\rm w}}
 \wedge
  \frac{t}{d(x, y)^{d_{\rm f} + d_{\rm w}}}
 \right\}
\end{equation}
and
\begin{equation}
 c_3 
 r^{d_{\rm f}}
\leq
 m( B(x, r) )
\leq
 c_4
 r^{d_{\rm f}}
\end{equation}
for all $x, y\in E$, $r\in (0, 1]$
and $t\in (0, 1]$.

By similar calculations as in the proof of Proposition
\ref{Prop_Gaussian}, we can prove the following.
\begin{Prop}
Under the above conditions,
if an integer $p$ with $p\geq 2$ satisfies
\begin{equation*}
 d_{\rm f} - p (d_{\rm f} - d_{\rm w}) > 0
,
\end{equation*}
then Assumption \textbf{(A)} holds.
\end{Prop}
\section*{Acknowledgements}
\addcontentsline{toc}{section}{Acknowledgements}
We 
would like to thank Professor Takashi Kumagai, 
my master's thesis supervisor, and 
Professor Ryoki Fukushima,
for helpful discussions.
We
also wish to thank Professor Wolfgang K\"onig and 
Professor Chiranjib Mukherjee for their comments.

\part*{}
\addcontentsline{toc}{section}{References}
\bibliographystyle{alpha}

\end{document}